\documentclass[12pt, reqno]{amsart}
\usepackage{mathtools, amscd, amsfonts, amssymb}
\usepackage{bm}
\usepackage{graphicx}
\usepackage[shortlabels]{enumitem}
\usepackage{mathrsfs}
\usepackage{cite}
\usepackage{lineno}
\usepackage[
colorlinks=true,
linkcolor=blue,
citecolor=blue,
urlcolor=blue,
pagebackref=true
]{hyperref}
\newtheorem{theorem}{Theorem}[section]
\newtheorem{lemma}[theorem]{Lemma}
\newtheorem{proposition}[theorem]{Proposition}
\newtheorem{corollary}[theorem]{Corollary}
\theoremstyle{definition}
\newtheorem{definition}[theorem]{Definition}
\newtheorem{example}[theorem]{Example}

\theoremstyle{remark}
\newtheorem{remark}[theorem]{Remark}
\numberwithin{equation}{section}

\numberwithin{equation}{section}

\newcommand{\mb}{\mathbb}

\newcommand{\mc}{\mathcal}

\DeclareMathOperator{\im}{range}

\newcommand{\TFAE}{The following conditions are equivalent:}

\newcommand{\ov}{\overline}
\newcommand{\sub}{\subseteq}

\newcommand{\wh}{\widehat}
\newcommand{\wt}{\widetilde}

\DeclareMathOperator{\diag}{diag}

\begin{document}
	
	\title[$k$-Regular Factorizations and Invariant Subspaces of C.N.U.  Contractions]{$k$-Regular Factorizations and Invariant Subspaces of Completely Non-Unitary Contractions}
	
	\author[Kalpesh J. Haria]{Kalpesh J. Haria}
	\address{School of Mathematics and Computer Science, Indian Institute of Technology Goa, Goa 403401, India}
	\email{kalpesh@iitgoa.ac.in, hikalpesh.haria@gmail.com}
	
	\author[Aashish Kumar Maurya]{Aashish Kumar Maurya}
	\address{School of Mathematics and Computer Science, Indian Institute of Technology Goa, Goa 403401, India}
	\email{aashish21232101@iitgoa.ac.in, a.k.maurya.math@gmail.com}


	\begin{abstract}
		We introduce the notion of $k$-regular factorizations for contractions into $k$ factors, generalizing the classical notion of regular factorization due to Sz.-Nagy and Foia\c{s}, and develop a systematic framework for their analysis. Using this concept, a one-to-one correspondence is established between chains of invariant subspaces
		\[
		\mathcal{M}_1 \sub \cdots \sub \mathcal{M}_{k-1},
		\]
		associated with a completely non-unitary contraction and the class of all $k$-regular factorizations of its characteristic function. An explicit functional model for the corresponding completely non-unitary contraction is constructed, and the associated functional model representations of the chain of invariant subspaces are obtained. Finally, examples illustrating the applicability of these results are provided. Furthermore, we introduce symmetric $k$-regular tuples for commuting $k$-contractions, proving this property holds when the product of contractions has a finite-dimensional defect space and is $k$-regular under at least one permutation. Importantly, we demonstrate that the classical counterexamples for commuting $3$-tuples provided by Parrott, Crabb--Davie, and Kaijser--Varopoulos fail to be symmetric $3$-regular tuples. This structural failure highlights the significance of symmetric $k$-regularity and offers a promising framework that encourages further research into this property and the commutative dilation theory of commuting $k$-contractions.
	\end{abstract}
	
	\maketitle
	
	\tableofcontents

	\section{Introduction }
	In this article, all Hilbert spaces are assumed to be separable and defined over the field of complex numbers. Let $\mc{H}$ be a Hilbert space, and let $B(\mc{H})$ denote the algebra of bounded linear operators acting on $\mc{H}$. Consider an operator \(T \in B(\mc{H})\). A closed subspace \(\mc{M} \sub \mc{H}\) is called an invariant subspace for \(T\) if \(T(\mc{M}) \sub \mc{M}\). Furthermore, $\mc M$ is called a reducing subspace for $T$ if it is invariant under both $T$ and its adjoint $T^*$. An invariant subspace $\mc M$ of a contraction is called \emph{hyperinvariant} if it is invariant under every operator that commutes with \(T\). Within operator theory, one of the most enduring open problems is the invariant subspace problem, which asks whether every bounded linear operator on a separable infinite-dimensional Hilbert space possesses a non-trivial closed invariant subspace.
	
	An operator $T \in B(\mc H)$ is defined as a contraction if $\|T\|\leq 1$. Furthermore, $T \in B(\mc H)$ is called unitary if $T^*T=I_{\mc H}=TT^*$. A contraction $T$ is said to be a completely non-unitary (c.n.u.) contraction if there exists no non-zero reducing subspace on which $T$ acts as a unitary operator. It is easy to observe that the invariant subspace problem for general bounded linear operators is equivalent to the problem posed for contractions alone. According to the canonical decomposition of contractions, every contraction can be uniquely decomposed as the orthogonal direct sum of a unitary operator and a c.n.u. contraction. The spectral theorem for normal operators guarantees that if $\dim \mc H>1$, every unitary operator possesses a non-trivial invariant subspace. Consequently, the general invariant subspace problem effectively reduces to the study of c.n.u. contractions.
	
	For a contraction $T \in B(\mc H)$, Sz.-Nagy and Foia\c{s} defined the characteristic function $\Theta_T:\mb{D}\to B(\mc{D}_T, \mc{D}_{T^*})$ by
	\[\Theta_T(z) \coloneqq (-T+zD_{T^*}(I-zT^*)^{-1}D_T)|_{\mc{D}_T}  ~\text{ for all } z\in\mb{D},\]
	where $D_T=(I-T^*T)^{1/2}$ and $D_{T^*}=(I-TT^*)^{1/2}$ are the defect operators, with the associated defect spaces $\mc{D}_T \coloneqq \overline{\im(D_T)}$ and $\mc{D}_{T^*} \coloneqq \overline{\im(D_{T^*})}.$
	The characteristic function $\Theta_T$ is a purely contractive operator-valued analytic function. Moreover, it serves as a complete unitary invariant for c.n.u.\ contractions (see Chapter VI of \cite{NFBK10}), where a corresponding functional model for c.n.u.\ contractions is also constructed.  Extending these concepts to the multivariable setting, the theory of characteristic functions and functional models for noncommuting row contractions was developed in the foundational works of A.~E.~Frazho \cite{Fr82a} and G.~Popescu \cite{Po95a,Po89a,Po89b}. In this context, G.~Popescu demonstrated that the multivariable characteristic function serves as a complete unitary invariant for completely non-coisometric (c.n.c.) row contractions. The corresponding results for commuting c.n.c.\ row contractions were established by T.~Bhattacharyya, J.~Eschmeier, and J.~Sarkar in \cite{Bh05a,Bh06a}.

	In Chapter VII of \cite{NFBK10}, Sz.-Nagy and Foia\c{s} established a one-to-one correspondence between the invariant subspaces of a c.n.u. contraction and the regular factorizations of its characteristic function; see also \cite{NF64a,NF64b}. In \cite{Po06}, G.~Popescu established a one-to-one correspondence between the joint invariant subspaces of a c.n.c.\ row contraction and the regular factorizations of its characteristic function. For a pair of commuting contractions, regular factorizations of purely contractive analytic function in the characteristic triple and the associated joint invariant subspaces were studied in \cite{Ba23a}.  Thus, the theory of regular factorization provides a concrete analytical framework for investigating the structure of invariant subspaces.

	Regular factorizations have been studied widely in the literature; some notable contributions are listed below. Using the theory of regular factorization, several properties of invariant and hyperinvariant subspaces have been established for operators in the class $C_{11}$ (see Chapter VI of \cite{NFBK10}). Hyperinvariant subspaces of weak contractions were studied in \cite{Wu79b}; see also \cite{Wu78a,Wu79c,Wu78b}. For a c.n.u.\ contraction \(T \in B(\mc H)\), R.~I.~Teodorescu \cite{Te76a} proved, using regular factorization techniques, that \(T\) admits an upper triangular representation
	\[
	T = \begin{bmatrix} A & * \\ 0 & B \end{bmatrix},
	\]
	where \(A\) is a unilateral shift and \(B\) is the adjoint of a unilateral shift, if and only if the characteristic function \(\Theta_T\) is constant. In \cite{Te75a,Te77a}, using the regular factorization of \(\Theta_T\) associated with a non-trivial invariant subspace \(\mc H_1\) of $T$, necessary and sufficient conditions were established for the existence of a closed subspace \(\mc H' \sub \mc H\), invariant under \(T\), such that
	\[
	\mc H = \mc H' + \mc H_1
	\qquad
	\text{(not necessarily an orthogonal direct sum)}.
	\]
	Furthermore, necessary and sufficient conditions for the uniqueness of such a complementary subspace \(\mc H'\) were obtained in \cite{Te79a}. When the invariant subspace \(\mc H_1\) is hyperinvariant, a necessary and sufficient condition was established in \cite{Te78b}; for related results, see \cite{Te80a}. Building upon the results of R.~I.~Teodorescu in \cite{Te75a}, P.~Y.~Wu obtained conditions for a c.n.u.\ contraction to be spectral in\cite{Wu79a}.   Finally, in \cite{Tim20}, D.~Timotin used the concept of regular factorization to describe all invariant subspaces of the operator \(S \oplus S^*\), where \(S\) denotes the unilateral shift on the Hardy space \(H^2(\mb D)\).
	
	D. K. Khan introduced the notions of $(+)$-regular and $(-)$-regular factorizations for
	contractive operator-valued analytic functions and employed these concepts to establish
	criteria for minimality, controllability, and observability in the cascade coupling of
	infinite-dimensional linear dynamical systems, specifically within the context of conservative
	and passive scattering systems (see \cite{Do94}). Subsequently, in \cite{Khan90a,Khan90b}, D. K. Khan generalized the concepts of $(+)$-regularity and $(-)$-regularity
	to the setting where the contractive analytic function is a product of $k$ contractive
	factors. Furthermore, utilizing the Sz.-Nagy--Foia\c{s} functional model, he derived
	necessary and sufficient conditions for $(+)$-regularity and demonstrated that the
	cascade coupling of $k$ minimal passive scattering systems is minimal if the underlying
	factorization is simultaneously $(+)$-regular and $(-)$-regular.

	In Section~2, for $k \geq 2$, motivated by the definition of $(+)$-regular factorization for contractive analytic functions with $k$ factors and the classical definition of regular factorization by Sz.-Nagy and Foia\c{s}, we introduce the concept of $k$-regular factorization for products of $k$ contractions. Observe that the classical definition of regular factorization
	given by Sz.-Nagy and Foia\c{s} in \cite{NFBK10} coincides with the case of
	$2$-regular factorization. Sz.-Nagy and Foia\c{s} \cite{NF74} established equivalent conditions characterizing regular factorizations of the product of two contractions. The present work extends these equivalent conditions to the more general setting of $k$-regular factorizations.
	
	In Section 3, we apply this framework to contractive analytic functions and establish the following correspondence: if
	\(\mc M_1 \sub \cdots \sub \mc M_{k-1}\) is a chain of invariant subspaces for a c.n.u.\ contraction $T$, then its characteristic function $\Theta_T$ admits a $k$-regular factorization; conversely, if $\Theta_T$ admits a $k$-regular factorization, then there exist invariant subspaces $\mc M_1, \dots, \mc M_{k-1}$ for $T$ satisfying
	\(\mc M_1 \sub \cdots \sub \mc M_{k-1}\).

	Moreover, given a $k$-regular factorization of a purely contractive analytic function, we construct a functional model whose model operator is unitarily equivalent to a c.n.u.\ contraction whose characteristic function coincides with the given function. Furthermore, the corresponding functional model representations of the associated invariant subspaces are obtained. Finally, the induced chain of invariant subspaces associated with the $k$-regular factorization yields a natural block upper triangular matrix representation of the c.n.u.\ contraction, where the characteristic function of each diagonal block coincides with the purely contractive part of the corresponding factor $\Theta_i$.

	A multivariable analogue of the results obtained in the present article is developed for c.n.c. row contractions in \cite{HM2026II}. In Proposition 2.4 of the monograph \cite{NF70}, Sz.-Nagy and Foia\c{s} proved that if $\mc M$ and $\mc M'$ are invariant subspaces of a c.n.u.\ contraction $T$ corresponding to the regular factorizations $\Theta_T=\Theta_2\Theta_1$ and $\Theta_T=\Theta_2'\Theta_1'$, respectively, then the inclusion $\mc M \sub \mc M'$ implies that $\Theta_1$ is a divisor of $\Theta_1'$. Subsequently, L.~K\'erchy \cite[Theorem 7]{Kerchy03} strengthened this result by proving that $\Theta_1$ is, in fact, a regular divisor of $\Theta_1'$. An analogous result for c.n.c.\ row contractions was established by G.~Popescu \cite[Theorem 3.8]{Po06}, who proved that if $\mc M$ and $\mc M'$ are joint invariant subspaces of a c.n.c.\ row contraction $T$ corresponding to the regular factorizations $\Theta_T=\Theta_2\Theta_1$ and $\Theta_T=\Theta_2'\Theta_1'$, respectively, then the inclusion $\mc M \sub \mc M'$ implies that $\Theta_1$ is a divisor of $\Theta_1'$. In \cite{HM2026II}, we further prove that $\Theta_1$ is indeed a regular divisor of $\Theta_1'$, thereby extending K\'erchy's result to the setting of c.n.c.\ row contractions, and also establish the converse implication. In Section~4, several examples are presented to illustrate the applicability of these results.

	In multivariable operator theory, the dilation theory of commuting $k$-tuples of contractions remains a complex challenge for $k \ge 3$. It is a fundamental result that any tuple admitting a commuting isometric dilation necessarily satisfies the von Neumann inequality; equivalently, $\mathcal{CID}_k(\mathcal H) \sub \mathcal{VNI}_k(\mathcal H)$. For $k \ge 3$, however, this inclusion is strictly proper. S. Parrott's classical counterexample establishes that satisfying the von Neumann inequality is insufficient to guarantee a commuting isometric dilation. Furthermore, constructions by S.~Kaijser and N.~Th.~Varopoulos, as well as M. J. Crabb and A. M. Davie, demonstrate that commuting contractions can fail the von Neumann inequality entirely.
	
	In Section 5, we introduce the concept of a symmetric $k$-regular tuple for commuting contractions and investigate the associated class, denoted by $\mathcal{SR}_k(\mathcal H)$. A critical pattern emerges upon examining the classical pathological commuting $3$-tuples that fail to admit commuting isometric dilations---namely, the counterexamples developed by S.~Kaijser and N.~Th.~Varopoulos,  M. J. Crabb and A. M. Davie, and S. Parrott, (see \cite{CD75a,Pa70a,Va74a}). It is observed that none of these examples forms a symmetric $3$-regular tuple. This naturally leads to the question of whether every symmetric $k$-regular tuple of commuting contractions admits a commuting isometric dilation and, if such a dilation exists, to explore the conditions guaranteeing its uniqueness up to unitary equivalence.  For $k=2$, J.~A.~Ball and H.~Sau \cite{Ba23a} proved that if the commuting pair $(T_1,T_2)$ is symmetric $2$-regular, then it admits a strongly minimal And\^o isometric lift $(V_1,V_2)$, and that all minimal And\^o isometric lifts are strongly minimal and mutually unitarily equivalent as lifts; see also \cite{Sa18a,BS20a}.

	\section{Structure of $k$-Regular Factorizations for Contraction Operators}

	We begin by defining the notion of a $k$-regular factorization for contractions, where $k \geq 2$ is an integer, and  let \( \{ \mc{H}_i \}_{i=1}^{k+1} \) denote a sequence of Hilbert spaces. Suppose that $A \in \mc{B}(\mc H_1, \mc H_{k+1})$ is a contraction that admits a factorization into a product of contractions
	\begin{equation}\label{k-factorization}
		A = A_k \cdots A_1,
	\end{equation}
	where $A_i \in \mc{B}(\mc H_i, \mc H_{i+1})$ for each $i = 1, \dots, k$. Associated with this factorization, we define a mapping $Z_k$ from the dense subspace $D_A \mc H_1 \sub \overline{D_A \mc H_1}$ into the direct sum of defect spaces $\bigoplus_{i=k}^1 \overline{D_{A_i} \mc H_i}$ by
	\begin{equation}\label{Z_k}
		Z_k (D_A h) \coloneqq D_{A_k} A_{k-1} \cdots A_1 h \oplus D_{A_{k-1}} A_{k-2} \cdots A_1 h \oplus \cdots \oplus D_{A_1} h.
	\end{equation}
	
	For any $h \in \mc H_1$, the following computation demonstrates that $Z_k$ is an isometry:
	\begin{align*}
		\| Z_k D_A h \|^2 &= \| D_{A_k} A_{k-1} \cdots A_1 h \|^2 + \| D_{A_{k-1}} A_{k-2} \cdots A_1 h \|^2 + \cdots + \| D_{A_1} h \|^2 \\
		&= \left( \| A_{k-1} \cdots A_1 h \|^2 - \| A_k A_{k-1} \cdots A_1 h \|^2 \right) + \cdots + \left( \| h \|^2 - \| A_1 h \|^2 \right) \\
		&= \| h \|^2 - \| A h \|^2 \\
		&= \| D_A h \|^2.
	\end{align*}
	
	Since $Z_k$ preserves the norm on a dense subspace, it extends uniquely to a well-defined isometry from $\overline{D_A \mc H_1}$ into $\bigoplus_{i=k}^1 \overline{D_{A_i} \mc H_i}$.
	
	\begin{definition}\label{def_Z_k}
		The factorization $A = A_k \cdots A_1$ is called a \textit{$k$-regular factorization} if the associated isometry $Z_k$ defined in \eqref{Z_k} is a unitary operator.
	\end{definition}
	
	Consider the factorization \( A = A_k \cdots A_1 \) given in equation \eqref{k-factorization}. We partition the index set \( \{1, \dots, k\} \) into \( r \) disjoint subsets \( J_1, \dots, J_r \), defined as follows
	\[
	J_1 = \{ j_1, \ldots, 1 \}, ~ \dots, ~ J_i = \{ j_i, \ldots, j_{i-1} + 1 \}, ~ \dots, ~ J_r = \{ j_r, \ldots, j_{r-1} + 1 \},
	\]
	where \( 1 \leq j_1 < j_2 < \ldots < j_r = k \). Let \( A_{J_i} \) denote the product of the operators indexed by \( J_i \). That is,
	\[
	A_{J_1} \coloneqq A_{j_1} \cdots A_1, ~ \dots, ~ A_{J_i} \coloneqq A_{j_i} \cdots A_{j_{i-1} + 1}, ~ \dots, ~ A_{J_r} \coloneqq A_k \cdots A_{j_{r-1} + 1}.
	\]
	Since the operators \( A_j \) are assumed to be contractions, it follows that each operator \( A_{J_i} \) (for \( i = 1, \dots, r \)) is also a contraction. According to the Definition \ref{def_Z_k}, the aggregated factorization \( A = A_{J_r} \cdots A_{J_1} \) is referred to as an \( r \)-regular factorization if the isometry associated with this partitioned factorization, denoted by \( Z_r^{J_r, \dots, J_1} : \mathcal{D}_A \to \mathcal{D}_{A_{J_r}} \oplus \cdots \oplus \mathcal{D}_{A_{J_1}} \), which is defined by
	\[
	Z_r^{J_r, \dots, J_1}(D_A h) \coloneqq D_{A_{J_r}} A_{J_{r-1}} \cdots A_{J_1} h \oplus \cdots \oplus D_{A_{J_1}} h
	\]
	is a unitary operator. For notational convenience, when the partition consists entirely of singletons (i.e., $J_i = \{i\}$ for all $i = 1, \dots, k$), the associated isometry is denoted simply by $Z_k$ rather than $Z_k^{\{k\}, \{k-1\}, \dots, \{1\}}$. Furthermore, the sub-factorization $A_{J_i} = A_{j_i} \cdots A_{j_{i-1} + 1}$ is $|J_i|$-regular if and only if the associated isometry
	\[
	Z_{|J_i|}^{\{j_i\}, \dots, \{j_{i-1}+1\}}: \mc{D}_{A_{J_i}} \longrightarrow \mc{D}_{A_{j_i}} \oplus \mc{D}_{A_{j_i-1}} \oplus \cdots \oplus \mc{D}_{A_{j_{i-1}+1}}
	\]
	defined by
	\begin{align*}
		Z_{|J_i|}^{\{j_i\}, \dots, \{j_{i-1}+1\}}
		(D_{A_{J_i}} x)
		\coloneqq\,
		&D_{A_{j_i}} A_{j_i-1} \cdots A_{j_{i-1}+1} x
		\oplus D_{A_{j_i-1}} A_{j_i-2} \cdots A_{j_{i-1}+1} x
		\oplus \cdots \\
		&\qquad\oplus D_{A_{j_{i-1}+1}} x
	\end{align*}
	is a unitary operator, where  $x \in \mc H_{j_{i-1}+1}$ and $|J_i|$ denotes the cardinality of the set $J_i$. In the trivial case where $J_i$ is a singleton (i.e., $J_i = \{j_i\}$), the corresponding isometry
	\[
	Z_{|J_i|}^{\{j_i\}}: \mc{D}_{A_{j_i}} \longrightarrow \mc{D}_{A_{j_i}}
	\]
	is understood to be the identity operator on $\mc{D}_{A_{j_i}}$.

	\begin{proposition}\label{partition}
		Let $A \in B(\mc{H}, \mc{K})$ be a contraction such that $A = A_k \cdots A_1$, where each $A_i \in B(\mc{H}_i, \mc{H}_{i+1})$ is a contraction with $\mc{H}_1 = \mc{H}$ and $\mc{H}_{k+1} = \mc{K}$.
		\TFAE
		\begin{enumerate}[\rm (i)]
			\item The factorization \( A = A_k \cdots A_1 \) is a \( k \)-regular factorization.
			
			\item For every disjoint partition $J_1, \dots, J_r$ of $\{1, \dots, k\}$, the factorization $A = A_{J_r} \cdots A_{J_1}$ is $r$-regular, and for each $i \in \{1, \dots, r\}$, the sub-factorization $A_{J_i} = A_{j_i} \cdots A_{j_{i-1} + 1}$ is $|J_i|$-regular.
			
			\item There exists a disjoint partition $J_1, \dots, J_r$ of $\{1, \dots, k\}$ such that the factorization $A = A_{J_r} \cdots A_{J_1}$ is $r$-regular, and for each $i \in \{1, \dots, r\}$, the sub-factorization $A_{J_i} = A_{j_i} \cdots A_{j_{i-1} + 1}$ is $|J_i|$-regular.
		\end{enumerate}
	\end{proposition}
	
	\begin{proof}
		For an arbitrary partition of the index set $\{1,\dots,k\}$ into $r$ disjoint subsets $J_1,\dots,J_r$, we claim that the following identity holds
		\begin{equation*}
			Z_k = \left( Z_{|J_r|}^{\{k\}, \dots, \{j_{r-1}+1\}} \oplus Z_{|J_{r-1}|}^{\{j_{r-1}\}, \dots, \{j_{r-2}+1\}} \oplus \cdots \oplus Z_{|J_1|}^{\{j_1\}, \dots, \{1\}} \right) Z_r^{J_r, \dots, J_1}.
		\end{equation*}
		Recall that the isometry
		$$
		Z_r^{J_r, \dots, J_1}: \mc{D}_A \longrightarrow \mc{D}_{A_{J_r}} \oplus \cdots \oplus \mc{D}_{A_{J_1}}
		$$
		is densely defined on the range of $D_A$ by
		$$
		Z_r^{J_r, \dots, J_1}(D_A h) \coloneqq D_{A_{J_r}} A_{J_{r-1}} \cdots A_{J_1} h \oplus \cdots \oplus D_{A_{J_1}} h.
		$$
		Furthermore, for each $i=1,\dots,r$, the sub-factorization $A_{J_i} = A_{j_i} \cdots A_{j_{i-1} + 1}$ induces a corresponding isometry
		\[
		Z_{|J_i|}^{\{j_i\}, \dots, \{j_{i-1}+1\}}: \mc{D}_{A_{J_i}} \longrightarrow \mc{D}_{A_{j_i}} \oplus \mc{D}_{A_{j_i-1}} \oplus \cdots \oplus \mc{D}_{A_{j_{i-1}+1}}
		\]
		which is explicitly determined for $x\in \mc H_{j_{i-1}+1}$  by
		\begin{align*}
			Z_{|J_i|}^{\{j_i\}, \dots, \{j_{i-1}+1\}}
			(D_{A_{J_i}} x)
			\coloneqq\,
			&D_{A_{j_i}} A_{j_i-1} \cdots A_{j_{i-1}+1} x
			\oplus D_{A_{j_i-1}} A_{j_i-2} \cdots A_{j_{i-1}+1} x
			\oplus \cdots \\
			&\qquad\oplus D_{A_{j_{i-1}+1}} x.
		\end{align*}
		Now, observe that for any vector $h \in \mc{H}$, we have
		\begin{align*}
			&\left( Z_{|J_r|}^{\{k\}, \dots, \{j_{r-1}+1\}} \oplus Z_{|J_{r-1}|}^{\{j_{r-1}\}, \dots, \{j_{r-2}+1\}} \oplus \cdots \oplus Z_{|J_1|}^{\{j_1\}, \dots, \{1\}} \right) Z_r^{J_r, \dots, J_1} (D_A h) \\
			&= \left( Z_{|J_r|}^{\{k\}, \dots, \{j_{r-1}+1\}} \oplus \cdots \oplus Z_{|J_1|}^{\{j_1\}, \dots, \{1\}} \right) (D_{A_{J_r}} A_{J_{r-1}} \cdots A_{J_1} h \oplus \cdots \oplus D_{A_{J_1}} h) \\
			&= Z_{|J_r|}^{\{k\}, \dots, \{j_{r-1}+1\}}(D_{A_{J_r}} A_{J_{r-1}} \cdots A_{J_1} h) \oplus \cdots \oplus Z_{|J_1|}^{\{j_1\}, \dots, \{1\}}(D_{A_{J_1}} h) \\
			&= (D_{A_k} A_{k-1} \cdots A_{j_{r-1}+1} A_{J_{r-1}} \cdots A_{J_1} h \oplus \cdots \oplus D_{A_{j_{r-1}+1}} A_{J_{r-1}} \cdots A_{J_1} h) \oplus \cdots \\
			&\quad \oplus (D_{A_{j_1}} A_{j_1-1} \cdots A_1 h \oplus \cdots \oplus D_{A_1} h) \\
			&= D_{A_k} A_{k-1} \cdots A_1 h \oplus D_{A_{k-1}} A_{k-2} \cdots A_1 h \oplus \cdots \oplus D_{A_1} h \\
			&= Z_k(D_A h).
		\end{align*}
		Because this identity holds on the dense range of $D_A$, it extends uniquely to the closure, thereby establishing the desired result
		\begin{equation}\label{deco_relation_zK}
			Z_k = \left( \bigoplus_{i=r}^{1} Z_{|J_i|}^{\{j_i\}, \dots, \{j_{i-1}+1\}} \right) Z_r^{J_r, \dots, J_1}.
		\end{equation}
		The conclusion of the proposition follows from the elementary fact that if a composition \( V = V_2 V_1 \) consists of isometries \( V_1 \) and \( V_2 \), then \( V \) is unitary if and only if both \( V_1 \) and \( V_2 \) are unitary.
	\end{proof}
	
	\begin{proposition}\label{propo_k_regular}
		For a given factorization \( A = A_k \cdots A_1 \) with $k \geq 2$, the following statements are equivalent
		\begin{enumerate}[label=\rm(\roman*), ref=\thetheorem(\roman*)]
			\item The factorization \( A = A_k \cdots A_1 \) is a \( k \)-regular factorization.
			
			\item The factorization \( A = A_k(A_{k-1} \cdots A_1) \) is a $2$-regular factorization, and \( A_{k-1} \cdots A_1 \) is a \( (k-1) \)-regular factorization.
			
			\item The factorization \( A_k \cdots A_2 \) is a \( (k-1) \)-regular factorization, and the factorization \( A = (A_k \cdots A_2) A_1 \) is a $2$-regular factorization.
			\item The factorizations $(A_k \cdots A_{j+1})(A_j \cdots A_1)$ are $2$-regular factorizations for all $j = 1, \dots, k-1$.
		\end{enumerate}
	\end{proposition}
	\begin{proof} Conditions (i), (ii), and (iii) are equivalent, and (i) implies (iv), all of which follow from Proposition~\ref{partition}. It remains to prove that (iv) implies (i). Consider the sequence of partitions of the index set \( \{1, \dots, k\} \)
		\begin{align*}
			\{\{k\}, \{k-1, \dots, 1\}\},\dots, \{\{k, \dots, j+1\}, \{j, \dots, 1\}\}, \dots ,
			\{\{k, \dots, 2\}, \{1\}\}.
		\end{align*}
		Corresponding to the partition $\{k, \dots, j+1\}, \{j, \dots, 1\}$ with $j \in \{1, \dots, k-1\}$, the isometry $Z_k$ admits the following structural decomposition, as given in \eqref{deco_relation_zK} from the proof of Proposition \ref{partition}
		\[
		Z_k = \left( Z_{k-j}^{\{k\}, \dots, \{j+1\}} \oplus Z_{j}^{\{j\}, \dots, \{1\}} \right) Z_2^{\{k, \dots, j+1\}, \{j, \dots, 1\}}.
		\]
		For $j=k-1,$ since $A_k(A_{k-1}\cdots A_1)$ is 2-regular, then the image of $\mc D_A$ of under $Z_k$ is given as follows
		\begin{align}
			Z_k(\mc{D}_A)
			&=\left( Z_{1}^{\{k\}} \oplus Z_{k-1}^{\{k-1\}, \dots, \{1\}} \right) Z_2^{\{k\}, \{k-1, \dots, 1\}}(\mc D_A) \nonumber\\
			&= \mc{D}_{A_k} \oplus
			\overline{\left\{
				D_{A_{k-1}} A_{k-2} \cdots A_1 h_1 \oplus \cdots \oplus D_{A_1} h_1
				:\; h_1 \in \mc{H}_1
				\right\}}.\label{eq:k_1}
		\end{align}
		Similarly, evaluating the image of $\mc D_A$ of under $Z_k$ for the partition at $j = k-2$ yields
		\begin{align} \label{eq:k_2}
			Z_k(\mc{D}_A) &= \overline{ \left\{ D_{A_k} A_{k-1} h_{k-1} \oplus D_{A_{k-1}} h_{k-1} : h_{k-1} \in \mc{H}_{k-1} \right\} } \nonumber \\
			&\quad \oplus \overline{ \left\{ D_{A_{k-2}} A_{k-3} \cdots A_1 h_1 \oplus \cdots \oplus D_{A_1} h_1 : h_1 \in \mc{H}_1 \right\} }.
		\end{align}
		Comparing equations \eqref{eq:k_1} and \eqref{eq:k_2}, we have the following inclusion relation
		
		\[
		\mc{D}_{A_k} \oplus \{0\} \subseteq \overline{ \left\{ D_{A_k} A_{k-1} h_{k-1} \oplus D_{A_{k-1}} h_{k-1} : h_{k-1} \in \mc{H}_{k-1} \right\} }.
		\]
		From this, observe that for any element
		\(D_{A_k} A_{k-1} h_{k-1} \oplus D_{A_{k-1}} h_{k-1}\), the element
		\(D_{A_k} A_{k-1} h_{k-1} \oplus 0\) also belongs to
		\[
		\overline{\left\{
			D_{A_k} A_{k-1} h_{k-1} \oplus D_{A_{k-1}} h_{k-1}
			:\; h_{k-1} \in \mc{H}_{k-1}
			\right\}}.
		\]
		Consequently,
		\[
		0 \oplus D_{A_{k-1}} h_{k-1}\in \overline{\left\{
			D_{A_k} A_{k-1} h_{k-1} \oplus D_{A_{k-1}} h_{k-1}
			:\; h_{k-1} \in \mc{H}_{k-1}
			\right\}}.
		\]
		Hence, we have
		
		\[
		\{0\} \oplus \mc{D}_{A_{k-1}} \subseteq \overline{ \left\{ D_{A_k} A_{k-1} h_{k-1} \oplus D_{A_{k-1}} h_{k-1} : h_{k-1} \in \mc{H}_{k-1} \right\} }.
		\]
		Thus, we obtain
		\begin{equation} \label{eq:k_4}
			\overline{ \left\{ D_{A_k} A_{k-1} h_{k-1} \oplus D_{A_{k-1}} h_{k-1} : h_{k-1} \in \mc{H}_{k-1} \right\} } = \mc{D}_{A_k} \oplus \mc{D}_{A_{k-1}}.
		\end{equation}
		Proceeding to the subsequent partition at $j = k-3$, we get
		\begin{align} \label{eq:k_5}
			Z_k(\mc{D}_A) &= \overline{ \left\{ D_{A_k} A_{k-1} A_{k-2} h_{k-2} \oplus D_{A_{k-1}} A_{k-2} h_{k-2} \oplus D_{A_{k-2}} h_{k-2} : h_{k-2} \in \mc{H}_{k-2} \right\} } \nonumber \\
			&\quad \oplus \overline{ \left\{ D_{A_{k-3}} A_{k-4} \cdots A_1 h_1 \oplus \cdots \oplus D_{A_1} h_1 : h_1 \in \mc{H}_1 \right\} }.
		\end{align}
		Applying the same argument to equations \eqref{eq:k_4} and  \eqref{eq:k_5}, we obtain
		\begin{align*}
			\overline{ \left\{ D_{A_k} A_{k-1} A_{k-2} h_{k-2}  \oplus D_{A_{k-1}} A_{k-2} h_{k-2} \oplus D_{A_{k-2}} h_{k-2} : h_{k-2} \in \mc{H}_{k-2} \right\} }  \\ \quad= \mc{D}_{A_k} \oplus \mc{D}_{A_{k-1}} \oplus \mc{D}_{A_{k-2}}.
		\end{align*}
		By iterating this procedure for \( j = k-4, \dots, 2, 1 \), we systematically accumulate the individual defect spaces. At the stage \( j = 1 \), the process yields the identity
		\[
		\overline{ \left\{ D_{A_k} A_{k-1} \cdots A_2 h_2 \oplus \cdots \oplus D_{A_2} h_2 : h_2 \in \mc{H}_2 \right\} } = \mc{D}_{A_k} \oplus \mc{D}_{A_{k-1}} \oplus \cdots \oplus \mc{D}_{A_2}.
		\]
		Finally, substituting this identity  for the partition $j=1$ , we formally arrive at
		\begin{align*}
			Z_k(\mc{D}_A) &= \overline{ \left\{ D_{A_k} A_{k-1} \cdots A_2 h_2 \oplus \cdots \oplus D_{A_2} h_2 : h_2 \in \mc{H}_2 \right\} } \oplus \mc{D}_{A_1} \\
			&= \mc{D}_{A_k} \oplus \mc{D}_{A_{k-1}} \oplus \cdots \oplus \mc{D}_{A_2} \oplus \mc{D}_{A_1} \\
			&= \bigoplus_{i=k}^1 \mc{D}_{A_i}.
		\end{align*}
		Thus, the factorization $A = A_k \cdots A_1$ is a $k$-regular factorization, thereby completing the proof.
	\end{proof}
	
	\begin{remark}
		In particular, when \( k = 3 \), the above proposition establishes Lemma 4.1 of Chapter VII in \cite{NFBK10}.
	\end{remark}
	
	\begin{corollary} \label{cor3}
		Let $A = A_k \cdots A_1$ be a product of contractions with $k \geq 2$. Then $A = A_k \cdots A_1$ is a $k$-regular factorization if and only if the factorizations $A_k(A_{k-1} \cdots A_1)$, $A_{k-1}(A_{k-2} \cdots A_1), \dots, A_3(A_2A_1)$, and $A_2A_1$ are all $2$-regular factorizations.
	\end{corollary}
	The following proposition generalizes the classical result of Sz.-Nagy and Foia\c{s} \cite{NF74} to the broader framework of $k$-regular factorizations.

	\begin{proposition} \label{equivalent_prop} Let $A = A_k \cdots A_1$ be a product of contractions with $k \geq 2$. The following conditions are equivalent
		\begin{enumerate}[\rm (i)] \item $A = A_k \cdots A_1$ is a $k$-regular factorization. \item $\overline{ \{ D_{A_{j+1}} h_{j+1} \oplus D_{A_1^* \cdots A_j^*} h_{j+1} : h_{j+1} \in \mathcal{H}_{j+1} \} } = \mathcal{D}_{A_{j+1}} \oplus \mathcal{D}_{A_1^* \cdots A_j^*}$ for $j = 1, \dots, k-1$. \item ${D}_{A_{j+1}} \mathcal{H}_{j+1} \cap {D}_{A_1^* \cdots A_j^*} \mathcal{H}_{j+1} = \{0\}$ for $j = 1, \dots, k-1$. \end{enumerate} \end{proposition}
	
	\begin{proof}
		The proof  of this equivalence follows from the work of Sz.-Nagy and Foia\c{s} \cite{NF74} and Corollary \ref{cor3}.
	\end{proof}
	
	
	\begin{proposition}\label{k_regular_pro}
		Let $A = A_k \cdots A_1$ be a product of contractions with $k \geq 2$. The following assertions are equivalent
		\begin{enumerate}[\rm (i)]
			\item If $A = A_k \cdots A_1$ is a $k$-regular factorization, then $A^* = A_1^* \cdots A_k^*$ is also a $k$-regular factorization.
			
			\item If $A$ is an isometry (respectively, unitary), then $A = A_k \cdots A_1$ is a $k$-regular factorization if and only if $A_1, \dots, A_k$ are isometries (respectively, unitaries).
			
			\item If $A_2, \dots, A_k$ are isometries, or $A_{k-1}^*, \dots, A_1^*$ are isometries, then $A = A_k \cdots A_1$ is a $k$-regular factorization.
			
			\item Since $Z_k : \mathcal{D}_A \to \bigoplus_{i=k}^1 \mathcal{D}_{A_i}$ is an isometry, it follows that
			\[
			\dim(\mathcal{D}_A) \le \sum_{i=k}^1 \dim(\mathcal{D}_{A_i}).
			\]
			Furthermore, equality holds if $A = A_k \cdots A_1$ is a $k$-regular factorization. Moreover, if $\dim(\mathcal{D}_A) < \infty$, then $A = A_k \cdots A_1$ is a $k$-regular factorization if and only if
			\[
			\dim(\mathcal{D}_A) = \sum_{i=k}^1 \dim(\mathcal{D}_{A_i}).
			\]
		\end{enumerate}
	\end{proposition}
	
	\begin{proof}
		(i) Suppose that $A = A_k \cdots A_1$ is a $k$-regular factorization. By Proposition \ref{equivalent_prop}, we have
		\begin{align*}
			\overline{\mathcal{D}_{A_{j+1}} \mathcal{H}_{j+1} \cap \mathcal{D}_{A_1^* \cdots A_j^*} \mathcal{H}_{j+1}} = \{0\}, \quad j = 1, \dots, k-1.
		\end{align*}
		It follows that each factorization
		\[
		(A_1^* \cdots A_j^*)A_{j+1}^*, \quad j = 1, \dots, k-1,
		\]
		is $2$-regular. Hence, by Corollary \ref{cor3}, the product $A^* = A_1^* \cdots A_k^*$ is a $k$-regular factorization.
		
		(ii) If $A$ is an isometry, then $\mathcal{D}_A = \{0\}$. Thus, $Z_k$ is unitary if and only if $\mathcal{D}_{A_i} = \{0\}$ for all $i = 1, \dots, k$, which is equivalent to each $A_i$ being an isometry. The unitary case follows analogously.
		
		(iii) Suppose that $A_2, \dots, A_k$ are isometries. Then $A_k \cdots A_2$ is a $(k-1)$-regular factorization, and $(A_k \cdots A_2)A_1$ is a $2$-regular factorization. Hence, by Proposition \ref{propo_k_regular}, $A = A_k \cdots A_1$ is a $k$-regular factorization. The argument for $A_{k-1}^*, \dots, A_1^*$ is analogous.
		
		(iv) Since $Z_k$ is an isometry, the dimension inequality follows immediately. If $\dim(\mathcal{D}_A) < \infty$ and equality holds, then $Z_k$ is surjective and hence unitary. Therefore, the factorization is $k$-regular.
	\end{proof}
	
	We now recall several fundamental results concerning the geometry of unitary dilations of contractions. A detailed exposition of these classical results can be found in Chapter II of \cite{NFBK10}. While we will introduce the essential notation required for our present purposes, we refer the reader to \cite{NFBK10} for any undefined terminology or supplementary background.
	
	\begin{theorem}[Theorems 1.1 and 2.1, Chapter II, \cite{NFBK10}]\label{geometry_of_K}
		Let $T$ be a contraction on a Hilbert space $\mc{H}$. Suppose that $(U, \mc{K})$ is the minimal unitary dilation of $T$, and let $(U_+, \mc{K}_+)$ denote its corresponding minimal isometric dilation. Under these assumptions, the space $\mc{K}$ admits the following geometric decomposition
		\begin{equation}
			\mc{K} = \dots \oplus U^* \mc{L}^* \oplus \mc{L}^* \oplus \mc{H} \oplus \mc{L} \oplus U\mc{L} \oplus \dots,
		\end{equation}
		where $\mc{L}^* = \overline{(U^*-T^*) \mc{H}}$ and $\mc{L} = \overline{(U-T) \mc{H}}$.
		
		Furthermore, the following structural orthogonal decompositions hold
		\begin{align}
			\mc{K} &= M(\mc{L}_*) \oplus \mc{R}, \label{wold_unitary} \\
			\mc{K}_+ &= M_+(\mc{L}_*) \oplus \mc{R} = \mc{H} \oplus M_+(\mc{L}),
		\end{align}
		where the associated space $\mc{L}_* = \overline{(I - U T^*) \mc{H}}$. The spaces $\mc{L}$ and $\mc{L}_*$, which are closed subspaces of $\mc{K}_+$, are wandering for $U_+$ (and consequently $U$), and
		\[
		M_+(\mc{L}_*) = \bigvee_{n=0}^{\infty} U^n (\mc{L}_*), \quad M(\mc{L}_*) = \bigvee_{n=-\infty}^{\infty} U^n (\mc{L}_*).
		\]
		Here, $\mc{R}$ is the subspace of $\mc{K}_+$ that reduces $U_+$ (and consequently for $U$) to its unitary part. Moreover, these spaces satisfy the following properties
		\begin{gather}
			\mc{L} \cap \mc{L}_* = \{0\}, \label{eq:L_intersect} \\
			P^{\mc{L}_*}M_+(\mc{L}) \sub M_+(\mc{L}_*). \label{eq:P_subset}
		\end{gather}
		If $T$ is a c.n.u.\ contraction, then
		\begin{gather}
			M(\mc{L}) \vee M(\mc{L}_*) = \mc{K}, \label{eq:cnu_condition} \\
			\mc{R} = \overline{(I - P^{\mc{L}_*}) M(\mc{L})}, \label{eq:R_subspace}
		\end{gather}
		where $P^{\mc{L}_*}$ denotes the orthogonal projection of $\mc{K}$ onto $M(\mc{L}_*)$.
	\end{theorem}
	
	A \emph{bounded analytic function} $\{ \mc{E}, \mc{E}_*, \Theta(z) \}$ between Hilbert spaces $\mc{E}$ and $\mc{E}_*$ is an operator-valued map $\Theta: \mb{D} \to B(\mc{E}, \mc{E}_*)$ defined by a uniformly bounded power series $\Theta(z) = \sum_{k=0}^{\infty} z^k \Theta_k$ on the open unit disc $\mb{D}$, where the series converges in the strong operator topology for every $z \in \mb{D}$. It is called \emph{contractive} if $\sup_{z \in \mb{D}} \|\Theta(z)\| \leq 1$, and a contractive analytic function $\{ \mc{E}, \mc{E}_*, \Theta(z) \}$ is called \emph{purely contractive} if $\|\Theta(0)e\| < \|e\|$ for all $e \in \mc{E} \setminus \{0\}$.
	
	Such a contractive analytic function $\{ \mc{E}, \mc{E}_*, \Theta(z) \}$ naturally induces the multiplication operators $\Theta_+ : H^2(\mc{E}) \to H^2(\mc{E}_*)$ and $\Theta : L^2(\mc{E}) \to L^2(\mc{E}_*)$, along with their corresponding defect operators
	\begin{gather*}
		(\Theta_+ u)(z) \coloneqq \Theta(z)u(z) \quad \text{for } u \in H^2(\mc{E}), \\
		(\Theta v)(t) \coloneqq \Theta(e^{it})v(t) \quad \text{for } v \in L^2(\mc{E}), \\
		\Delta_{\Theta}(t) \coloneqq (I_{\mc{E}} - \Theta^*(e^{it})\Theta(e^{it}))^{1/2}, \quad \Delta_{*,\Theta}(t) \coloneqq (I_{\mc{E}_*} - \Theta(e^{it})\Theta(e^{it})^*)^{1/2},
	\end{gather*}
	where the boundary values $\Theta(e^{it}) = \mathop{\text{SOT-}\lim}_{r \to 1^-} \Theta(re^{it})$ exist almost everywhere (a.e.) on $[0, 2\pi]$. One of the most important results, which relates contractions and contractive analytic functions, is given below:
	
	\begin{lemma}[Lemma 3.1, Chapter V, \cite{NFBK10}]\label{fourier_rep}
		Let $U$ and $U'$ be bilateral shifts on the complex separable Hilbert spaces $\mc{K}$ and $\mc{K}'$, with generating subspaces $\mc{G}$ and $\mc{G}'$, respectively. Let $Q: \mc{K} \longrightarrow \mc{K}'$ be a contraction such that
		\begin{align*}
			Q U &= U' Q, \\
			Q M_+(\mc{G}) &\sub M_+(\mc{G}').
		\end{align*}
		Then there exists a unique contractive analytic function $\{\mc{G}, \mc{G}', \Theta(z)\}$ such that
		\begin{equation}
			\Phi^{\mc{G}'} Q = \Theta \Phi^{\mc{G}},
		\end{equation}
		where $\Phi^{\mc{G}'}: M(\mc{G}') \to L^2(\mc{G}')$ and $\Phi^{\mc{G}}: M(\mc{G}) \to L^2(\mc{G})$ are the Fourier representations of $M(\mc{G}')$ and $M(\mc{G})$, respectively, defined by
		\[
		\Phi^{\mc{G}'}\left(\sum_{n=-\infty}^{\infty} (U')^n g_n'\right) \coloneqq \sum_{n=-\infty}^{\infty} e^{int} g_n', \quad \Phi^{\mc{G}}\left(\sum_{n=-\infty}^{\infty} U^n g_n\right) \coloneqq \sum_{n=-\infty}^{\infty} e^{int} g_n.
		\]
	\end{lemma}
	
	\begin{proposition}[Proposition 2.1, Chapter VI, \cite{NFBK10}]\label{functional_model}
		Let $T$ be a c.n.u.\ contraction acting on the Hilbert space $\mc{H}$. Suppose $(U, \mc{K})$ is the minimal unitary dilation of $T$, and define the corresponding defect spaces by $\mc{L} \coloneqq \ov{(U-T)\mc{H}}$ and $\mc{L}_{*} \coloneqq \ov{(I - UT^{*})\mc{H}}$. Then there exists a purely contractive analytic function $\{\mc{L}, \mc{L}_{*}, \Theta_{\mc L}(z)\}$ such that
		$$
		\Phi^{\mc{L}_{*}} P^{\mc{L}_{*}} f = \Theta_{\mc L} \Phi^{\mc{L}} f \quad \text{for all } f \in M(\mc{L}).
		$$
		The map $\Phi : \mc{K} \to \wh{\mc{K}} \coloneqq L^{2}(\mc{L}_{*}) \oplus \ov{\Delta_{\mc L} L^{2}(\mc{L})}$ is a unitary transformation, decomposed as $\Phi = \Phi^{\mc{L}_{*}} \oplus \Phi_{\mc R}$, where the unitary operator $\Phi_{\mc R} : \mc{K} \to \ov{\Delta_{\mc L} L^{2}(\mc{L})}$ is defined by
		$$
		\Phi_{\mc R}(I - P^{\mc{L}_{*}})f \coloneqq \Delta_{\mc L}\Phi^{\mc{L}}f \quad \text{for all } f \in M(\mc{L}).
		$$
		This transformation $\Phi$ yields the ``Fourier representation'' of $\mc{K}$. In this representation, the unitary operator $U$ acts as multiplication by $e^{it}$ on $\wh{\mc{K}}$, and the subspace $\mc{K}_{+}$ of $\mc{K}$ corresponds to
		$$
		\wh{\mc{K}}_{+} = H^{2}(\mc{L}_{*}) \oplus \ov{\Delta_{\mc L}L^{2}(\mc{L})}.
		$$
		Furthermore, the Hilbert space $\mc{H}$ and the contraction $T$ are unitarily equivalent to the subspace $\wh{\mc{H}}$ and the operator $\wh{T}$ acting on $\wh{\mc{H}}$, respectively, given by
		\begin{equation*}
			\wh{\mc{H}} \coloneqq \left[ H^{2}(\mc{L}_{*}) \oplus \ov{\Delta_{\mc L}L^{2}(\mc{L})} \right] \ominus \{\Theta_{\mc L}u \oplus \Delta_{\mc L}u : u \in H^{2}(\mc{L})\}
		\end{equation*}
		and
		\begin{equation*}
			\wh{T}^{*}(u \oplus v) = e^{-it}[u - u(0)] \oplus e^{-it}v,\quad u\oplus v \in \wh{\mc{H}}.
		\end{equation*}
		If the function $\Theta_{\mc L}(z)$ is \emph{inner} {\rm (i.e., $\|\Theta_{\mc L}(e^{it})e\| = \|e\|$ a.e. for all $e \in \mc L$)}, the representations for $\wh{\mc{K}}$, $\wh{\mc{H}}$, and $\wh{T}$ simplify to
		\begin{align}
			\wh{\mc{K}} &= L^{2}(\mc{L}_{*}), \quad \wh{\mc{H}} = H^{2}(\mc{L}_{*}) \ominus \Theta_{\mc L}H^{2}(\mc{L}), \label{eq:inner_K_H} \\
			(\wh{T}^{*}u)(z) &= \frac{1}{z}\bigl[u(z) - u(0)\bigr],\text{ for } u \in \wh{\mc{H}}, \; z \in \mathbb{D}. \label{eq:inner_T_star}
		\end{align}
	\end{proposition}
	

	\begin{definition} \label{def:k_regular}
		Let $\{\mathcal{E}, \mathcal{E}_*, \Theta(z)\}$ be a contractive analytic function that admits a factorization into $k$ contractive analytic functions $\{\mathcal{E}_i, \mathcal{E}_{i+1}, \Theta_i(z)\}$ for $i = 1, \dots, k$, where we identify the spaces as $\mathcal{E}_1 = \mathcal{E}$ and $\mathcal{E}_{k+1} = \mathcal{E}_*$. That is,
		\[
		\Theta(z) = \Theta_k(z) \Theta_{k-1}(z) \cdots \Theta_1(z) \quad \text{for all } z \in \mathbb{D}.
		\]
		We define the canonical isometry
		\[
		Z_k : \ov{\Delta_{\Theta} L^2(\mathcal{E}_1)} \longrightarrow \ov{\Delta_k L^2(\mathcal{E}_k)} \oplus \cdots \oplus \ov{\Delta_1 L^2(\mathcal{E}_1)}
		\]
		by
		\[
		Z_k(\Delta_{\Theta} f) \coloneqq \Delta_k \Theta_{k-1} \cdots \Theta_1 f \oplus \cdots \oplus \Delta_2 \Theta_1 f \oplus \Delta_1 f, \quad \text{for all } f \in L^2(\mathcal{E}_1),
		\]
		where the associated defect operators are defined as $\Delta_{\Theta} = (I - \Theta^* \Theta)^{1/2}$ and $\Delta_i = (I - \Theta_i^* \Theta_i)^{1/2}$ for $i = 1, \dots, k$. By continuity, $Z_k$ extends uniquely to a bounded isometric operator from the closure $\overline{\Delta_{\Theta} L^2(\mathcal{E}_1)}$ into the direct sum of the defect spaces $\overline{\Delta_k L^2(\mathcal{E}_k)} \oplus \cdots \oplus \overline{\Delta_1 L^2(\mathcal{E}_1)}$. This extension is also denoted by $Z_k$. The factorization $\Theta(z) = \Theta_k(z) \cdots \Theta_1(z)$ is said to be \emph{$k$-regular} if the extended isometry $Z_k$ is a unitary operator.
	\end{definition}
	
	\begin{remark}
		From Definition \ref{def_Z_k} and Definition \ref{def:k_regular}, it is evident that the $k$-regularity of the factorization $\Theta(z) = \Theta_k(z) \cdots \Theta_1(z)$ for contractive analytic functions is exactly the  $k$-regularity of the corresponding operator factorization $\Theta = \Theta_k \cdots \Theta_1$ acting on their respective $L^2$ spaces.
		
	\end{remark}

	\begin{proposition}\label{local_prop}
		Let $\{\mathcal{E}_i, \mathcal{E}_{i+1}, \Theta_i(z)\}$ be contractive analytic functions such that
		\[
		\Theta(z) = \Theta_k(z) \cdots \Theta_1(z), \quad z \in \mathbb{D}.
		\]
		Then $\Theta = \Theta_k \cdots \Theta_1$ is a $k$-regular factorization if and only if the  factorization
		\[
		\Theta(e^{it}) = \Theta_k(e^{it}) \cdots \Theta_1(e^{it})
		\]
		is $k$-regular for almost every $t \in [0,2\pi]$.
	\end{proposition}
	\begin{proof}
		Let $\Theta = \Theta_k \cdots \Theta_1$ be a $k$-regular factorization; that is, the map
		\[ Z_k : \overline{\Delta_\Theta L^2(\mathcal{E})} \to \overline{\Delta_k L^2(\mathcal{E})} \oplus \cdots \oplus \overline{\Delta_1 L^2(\mathcal{E})} \]
		defined by
		\[ Z_k(\Delta_\Theta f) = \Delta_k \Theta_{k-1} \cdots \Theta_1 f \oplus \cdots \oplus \Delta_2 \Theta_1 f \oplus \Delta_1 f, \quad f \in L^2(\mathcal{E}), \]
		is a unitary map, which means
		\begin{equation} \label{eq1}
			\overline{\{ \Delta_k \Theta_{k-1} \cdots \Theta_1 f \oplus \cdots \oplus \Delta_2 \Theta_1 f \oplus \Delta_1 f : f \in L^2(\mathcal{E}) \}} = \bigoplus_{i=k}^1 \overline{\Delta_i L^2(\mathcal{E}_i)}.
		\end{equation}
		We claim that $Z_k(t) : \overline{\Delta_\Theta(t) \mathcal{E}} \to \bigoplus_{i=k}^1 \overline{\Delta_k(t) \mathcal{E}_i}$ defined by
		\[ Z_k(t)(\Delta_\Theta(t) e) = \Delta_k(t) \Theta_{k-1}(e^{it}) \cdots \Theta_1(e^{it}) e \oplus \cdots \oplus \Delta_2(t) \Theta_1(e^{it}) e \oplus \Delta_1(t) e \]
		is a unitary map, i.e.,
		\begin{equation} \label{eq2}
			\overline{\{ \Delta_k(t) \Theta_{k-1}(e^{it}) \cdots \Theta_1(e^{it}) e \oplus \cdots  \oplus \Delta_1(t) e: e\in\mc E \}} = \bigoplus_{i=k}^1 \overline{\Delta_i(t) \mathcal{E}_i}.
		\end{equation}
		Let $e_i \in \mathcal{E}_i$ be arbitrary constant vectors, and define the constant functions $f_i(t) = e_i$ in $L^2(\mathcal{E}_i)$ for $t \in \mathbb{T}$. By equation \eqref{eq1}, there exists a sequence of functions $\{f_m\}_{m=1}^\infty \subset L^2(\mathcal{E})$ such that, as $m \to \infty$,
		\[
		\Delta_k \Theta_{k-1} \cdots \Theta_1 f_m \oplus \cdots \oplus \Delta_1 f_m \longrightarrow \bigoplus_{i=k}^1 \Delta_i f_i \quad \text{in the norm of } \bigoplus_{i=k}^1 L^2(\mathcal{E}_i).
		\]
		It is a standard property of $L^2$ spaces that norm convergence implies the existence of a subsequence converging pointwise almost everywhere. Consequently, we can extract a subsequence $\{f_{m_j}\}_{j=1}^\infty$ and identify a subset $E_1 \subseteq \mathbb{T}$ of Lebesgue measure zero, such that for all $t \in \mathbb{T} \setminus E_1$, we obtain the pointwise limit
		\begin{align*}
			&\Delta_k(t) \Theta_{k-1}(e^{it}) \cdots \Theta_1(e^{it}) f_{m_j}(t) \oplus \cdots  \oplus \Delta_1(t) f_{m_j}(t) \longrightarrow \bigoplus_{i=k}^1 \Delta_i(t) e_i \quad \text{as } j \to \infty,
		\end{align*}
		where the convergence is in the norm topology of the direct sum space $\bigoplus_{i=k}^1 \mathcal{E}_i$; this implies that
		\[\bigoplus_{i=k}^1 \Delta_i(t) e_i\in   \overline{\{ \Delta_k(t) \Theta_{k-1}(e^{it}) \cdots \Theta_1(e^{it}) e \oplus \cdots \oplus \Delta_2(t) \Theta_1(e^{it}) e \oplus \Delta_1(t) e: e\in\mc E \}}\]
		Clearly, the set $E_1$ depends upon the choice of the vectors $e_1, \dots, e_k$.
		To eliminate the dependence on the chosen vectors, consider countable dense subsets
		$\{e_1^{j_1}\}_{j_1=1}^\infty, \dots, \{e_k^{j_k}\}_{j_k=1}^\infty$ of
		$\mathcal{E}_1, \dots, \mathcal{E}_k$, respectively.
		Let the vectors $e_1^{j_1}, \dots, e_k^{j_k}$ range over these sets, and let $E$ denote the union of the corresponding sets of Lebesgue measure zero. Then the set $E$ has Lebesgue measure zero, is independent of the choice of $j_1, \dots, j_k$, and for all $t \in \mathbb{T} \setminus E$ we have
		\[
		\bigoplus_{i=k}^1 \Delta_i(t) e_i^{j_i} \in
		\overline{\left\{
			\Delta_k(t) \Theta_{k-1}(e^{it}) \cdots \Theta_1(e^{it}) e
			\oplus \cdots \oplus
			\Delta_2(t) \Theta_1(e^{it}) e
			\oplus \Delta_1(t) e
			\right\}}.
		\]
		Therefore, we have
		\[ \overline{\{ \Delta_k(t) \Theta_{k-1}(e^{it}) \cdots \Theta_1(e^{it}) e \oplus \cdots \oplus \Delta_2(t) \Theta_1(e^{it}) e \oplus \Delta_1(t) e \}} = \bigoplus_{i=k}^1 \overline{\Delta_i(t) \mathcal{E}_i}. \]
		
		Conversely, assume that \eqref{eq2} holds. Let
		\[
		\bigoplus_{i=k}^1 f_i \in \bigoplus_{i=k}^1 \Delta_i L^2(\mathcal{E}_i)
		\]
		be such that
		\begin{align*}
			\left\langle \bigoplus_{i=k}^1 f_i,\,
			\Delta_k \Theta_{k-1} \cdots \Theta_1 f
			\oplus \cdots \oplus
			\Delta_2 \Theta_1 f
			\oplus \Delta_1 f \right\rangle
			= 0
		\end{align*}
		for all $f \in L^2(\mathcal{E})$. Then
		\begin{align*}
			\left\langle
			\Theta_1^* \cdots \Theta_{k-1}^* \Delta_k f_k
			+ \cdots + \Delta_1 f_1,\,
			f
			\right\rangle
			= 0
			\quad \text{for all } f \in L^2(\mathcal{E}),
		\end{align*}
		which implies
		\begin{align*}
			\Theta_1^* \cdots \Theta_{k-1}^* \Delta_k f_k
			+ \cdots + \Delta_1 f_1 = 0.
		\end{align*}
		Hence, for a.e.~$t \in [0,2\pi]$,
		\begin{align*}
			\Theta_1^*(e^{it}) \cdots \Theta_{k-1}^*(e^{it}) \Delta_k(t) f_k(t)
			+ \cdots + \Delta_1(t) f_1(t) = 0.
		\end{align*}
		Thus, for a.e.~$t \in [0,2\pi]$,
		\[
		\bigoplus_{i=k}^1 f_i(t) \perp
		\overline{\left\{
			\Delta_k(t) \Theta_{k-1}(e^{it}) \cdots \Theta_1(e^{it}) e
			\oplus \cdots \oplus
			\Delta_2(t) \Theta_1(e^{it}) e
			\oplus \Delta_1(t) e
			:\, e \in \mc{E}
			\right\}}.
		\]
		Since $f_i(t) \in \overline{\Delta_i(t)\mc{E}_i}$, it follows from \eqref{eq2} that
		\[
		\bigoplus_{i=k}^1 f_i(t) = 0 \quad \text{a.e.}
		\]
		and hence $f_i = 0$ a.e.\ for each $i = 1, \dots, k$. Therefore, \eqref{eq1} holds.
	\end{proof}


	\section{Main Theorem}
	The following theorem presents our main result and establishes the correspondence between chains of invariant subspaces and the $k$-regular factorizations of the characteristic function of a c.n.u.\ contraction.
	
	\begin{theorem}\label{k_r_invariant_theorem}
		Let $T \in B(\mc{H})$ be a c.n.u.\ contraction, and let $\{\mc{E}, \mc{E}_*, \wt{\Theta}(z)\}$ be a contractive analytic function which coincides with the characteristic function of $T$. Then the operator $T$ is unitarily equivalent to the operator $\wt{T}$ acting on the Hilbert space
		\[
		\wt{\mc{H}} \coloneqq \left[ H^2(\mc{E}_{*}) \oplus \overline{\wt{\Delta} L^2(\mc{E})} \right] \ominus \wt{\mc{G}},
		\]
		where
		\begin{gather*}
			\wt{T}^*(u \oplus v) \coloneqq e^{-it} (u - u(0)) \oplus e^{-it} v, \\
			\wt{\mc{G}} \coloneqq \{ \wt{\Theta} u \oplus \wt{\Delta} u :u \in H^2(\mc{E}) \}, \quad \wt{\Delta} \coloneqq (I - \wt{\Theta}^* \wt{\Theta})^{1/2}.
		\end{gather*}
		Assume further that $k \geq 2$ and that
		\[
		\mc{M}_1 \sub \dots \sub \mc{M}_{k-1}
		\]
		are invariant subspaces for $T$. Then there exists a $k$-regular factorization of the form
		\begin{equation}\label{k_regular_fact}
			\wt{\Theta}(z) = \wt{\Theta}_k(z) \cdots \wt{\Theta}_1(z) \quad \text{for } z \in \mathbb{D},
		\end{equation}
		where $\{\mc{E}_i, \mc{E}_{i+1}, \wt{\Theta}_i(z)\}$ are contractive analytic functions for $i=1,\dots,k$ with $\mc{E}_1 = \mc{E}$ and $\mc{E}_{k+1} = \mc{E}_*$. With respect to the $k$-regular factorization \eqref{k_regular_fact}, the corresponding invariant subspaces admit the following representations
		\begin{align*}
			\wt{\mc{M}}_1
			&= \Big\{
			\wt{\Theta}_k \cdots \wt{\Theta}_2 u_2
			\oplus
			\wt{Z}_k^*\big(
			\wt{\Delta}_k \wt{\Theta}_{k-1} \cdots \wt{\Theta}_2 u_2
			\oplus \dots
			\oplus \wt{\Delta}_2 u_2
			\oplus v_1
			\big) \\
			&\qquad :
			u_2 \in H^2(\mc{E}_2), \;
			v_1 \in \overline{\wt{\Delta}_1 L^2(\mc{E}_1)}
			\Big\}
			\ominus \wt{\mc{G}}, \\[1ex]
			&\vdots \\
			\wt{\mc{M}}_{i}
			&= \Big\{
			\wt{\Theta}_k \cdots \wt{\Theta}_{i+1} u_{i+1}
			\oplus
			\wt{Z}_k^*\big(
			\wt{\Delta}_k \wt{\Theta}_{k-1} \cdots \wt{\Theta}_{i+1} u_{i+1}
			\oplus \dots
			\oplus \wt{\Delta}_{i+1} u_{i+1}
			\oplus v_i
			\oplus \cdots \\
			&\qquad
			\oplus v_1
			\big)  :
			u_{i+1} \in H^2(\mc{E}_{i+1}), \;
			v_j \in \overline{\wt{\Delta}_j L^2(\mc{E}_j)}, \; j=1,\dots,i
			\Big\}
			\ominus \wt{\mc{G}}, \\
			&\vdots \\
			\wt{\mc{M}}_{k-1}
			&= \Big\{
			\wt{\Theta}_k u_k
			\oplus
			\wt{Z}_k^*\big(
			\wt{\Delta}_k u_k
			\oplus v_{k-1}
			\oplus \dots
			\oplus v_1
			\big) \\
			&\qquad :
			u_k \in H^2(\mc{E}_k), \;
			v_j \in \overline{\wt{\Delta}_j L^2(\mc{E}_j)}, \; j=1,\dots,k-1
			\Big\}
			\ominus \wt{\mc{G}}.
		\end{align*}
		Moreover, for $i = 1, \dots, k-1$, the orthogonal complement $\wt{\mc{N}}_i = \wt{\mc{H}} \ominus \wt{\mc{M}}_i$ is given by
		\begin{align*}
			\wt{\mc{N}}_i &= \left[ H^2(\mc{E}_*) \oplus \wt{Z}_k^* \left( \overline{\wt{\Delta}_k L^2(\mc{E}_{k})} \oplus \dots \oplus \overline{\wt{\Delta}_{i+1} L^2(\mc{E}_{i+1})} \oplus \{0\} \oplus \dots \oplus \{0\} \right) \right] \\
			&\quad \ominus \left\{ \wt{\Theta}_k \cdots \wt{\Theta}_{i+1} u_{i+1} \oplus \wt{Z}_k^* \big( \wt{\Delta}_k \wt{\Theta}_{k-1} \cdots \wt{\Theta}_{i+1} u_{i+1} \oplus \dots \oplus \wt{\Delta}_{i+1} u_{i+1} \right. \\
			&\qquad \left. \oplus 0 \oplus \dots \oplus 0 \big) :u_{i+1} \in H^2(\mc{E}_{i+1}) \right\}.
		\end{align*}
		Here,
		\[
		\wt{Z}_k :
		\overline{\wt{\Delta} L^2(\mc{E})}
		\longrightarrow
		\overline{\wt{\Delta}_k L^2(\mc{E}_k)}
		\oplus \dots \oplus
		\overline{\wt{\Delta}_1 L^2(\mc{E}_1)}
		\]
		is the isometry defined by
		\begin{equation*}
			\wt{Z}_k \wt{\Delta} v
			=
			\wt{\Delta}_k \wt{\Theta}_{k-1} \cdots \wt{\Theta}_1 v
			\oplus \dots \oplus
			\wt{\Delta}_2 \wt{\Theta}_1 v
			\oplus
			\wt{\Delta}_1 v,
			\qquad
			v \in L^2(\mc{E}).
		\end{equation*}
		Conversely, if
		\[
		\wt{\Theta}(z)
		=
		\wt{\Theta}_k(z)\cdots\wt{\Theta}_1(z)
		\]
		is a $k$-regular factorization and the subspaces
		$\wt{\mc{M}}_i$ and $\wt{\mc{N}}_i$ are defined as above,
		then each $\wt{\mc{M}}_i$ is invariant under $\wt{T}$, the collection $\{\wt{\mc{M}}_i\}_{i=1}^{k-1}$ forms an increasing chain
		\[
		\wt{\mc{M}}_1
		\sub \dots \sub
		\wt{\mc{M}}_{k-1},
		\]
		and for every $i=1,\dots,k-1$ we have the orthogonal decomposition
		\[
		\wt{\mc{H}}
		=
		\wt{\mc{M}}_i
		\oplus
		\wt{\mc{N}}_i.
		\]
	\end{theorem}

	\begin{proof}
		Let $\mc{M}_1 \sub \mc{M}_2 \sub \cdots \sub \mc{M}_{k-1}$ be a chain of invariant subspaces for a c.n.u.\ contraction $T$ acting on a Hilbert space $\mc{H}$. Let $(U, \mc{K})$ be the minimal unitary dilation of $T$, and define $\mc{K}_+ \coloneqq \bigvee_{n=0}^{\infty} U^n \mc{H}$. Then the restriction $U_+ \coloneqq U|_{\mc{K}_+}$ is the minimal isometric dilation of $T$, and we have the relation
		\begin{equation*}
			T^* = U_+^*|_{\mc{H}}.
		\end{equation*}
		Since $\mc{M}_1, \dots, \mc{M}_{k-1}$ are invariant subspaces for $T$, their orthogonal complements $\mc N_i \coloneqq \mc{H} \ominus \mc{M}_i$, for $i=1, \dots, k-1$, are invariant subspaces for $T^*$. Consequently, each $\mc N_i$ is invariant under $U_+^*$, which implies that the spaces $\mc{K}_i \coloneqq \mc{K}_+ \ominus \mc N_i = \mc{K}_+ \ominus (\mc{H} \ominus \mc{M}_i)$ are invariant under $U_+$. By applying the Wold decomposition to the isometry $U_+|_{\mc{K}_i}$, we obtain
		\begin{equation}\label{eq:wold_Ki}
			\mc{K}_i = M_+(\mc{F}_i) \oplus \mc{R}_i,
		\end{equation}
		where $U_+|_{M_+(\mc{F}_i)}$ is a unilateral shift with wandering subspace $\mc{F}_i \coloneqq \mc{K}_i \ominus U_+ \mc{K}_i$, and $U_+|_{\mc{R}_i}$ is unitary with residual space $\mc{R}_i \coloneqq \bigcap_{j=0}^{\infty} U_+^j \mc{K}_i$ for each $i=1, \dots, k-1$.
		
		The inclusion $\mc{M}_1 \sub \mc{M}_2 \sub \cdots \sub \mc{M}_{k-1}$ implies that $\mc N_1 \supseteq \mc N_2 \supseteq \dots \supseteq \mc N_{k-1}$, and thus $\mc{K}_1 \sub \mc{K}_2 \sub \cdots \sub \mc{K}_{k-1}$. Therefore, we have $M_+(\mc{F}_i) \oplus \mc{R}_i \sub M_+(\mc{F}_{i+1}) \oplus \mc{R}_{i+1}$ for $i=1, \dots, k-2$. We claim that $\mc{R}_i \sub \mc{R}_{i+1}$. Indeed, since $\mc{K}_i \sub \mc{K}_{i+1}$ and since $\mc{R}_i$ and $\mc{R}_{i+1}$ are the maximal reducing subspaces of $U_+|_{\mc{K}_i}$ and $U_+|_{\mc{K}_{i+1}}$, respectively, on which these restrictions act unitarily, the maximality of $\mc{R}_{i+1}$ implies that $\mc{R}_i \sub \mc{R}_{i+1}$. Hence, we obtain the chain $\mc{R}_1 \sub \mc{R}_2 \sub \dots \sub \mc{R}_{k-1}$.
		
		Define the orthogonal differences $\mc W_1 \coloneqq \mc{R}_1$ and $\mc W_i \coloneqq \mc{R}_i \ominus \mc{R}_{i-1}$ for $i=2, \dots, k-1$. Because $\mc{K}_i \sub \mc{K}_{i+1}$, we deduce that
		\begin{equation}\label{F_iF_i+1}
			M_+(\mc{F}_i) \sub M_+(\mc{F}_{i+1}) \oplus \mc W_{i+1} \quad \text{for } i=1, \dots, k-2.
		\end{equation}
		Consider the Wold decomposition of $U_+|_{\mc{K}_+}$, given by $\mc{K}_+ = M_+(\mc{L}_*) \oplus \mc{R}$, where $\mc{R} \coloneqq \bigcap_{n=0}^{\infty} U_+^n \mc{K}_+$. Let $\mc{R}_k \coloneqq \mc{R} \ominus \mc{R}_{k-1}$ and define $\mc W_k \coloneqq \mc{R}_k$. It follows that $\mc{R}$ admits the orthogonal decomposition
		\begin{equation}\label{eq:R_decomp}
			\mc{R} = \mc W_k \oplus \mc W_{k-1} \oplus \dots \oplus \mc W_1.
		\end{equation}
		Recall the geometric decomposition of $\mc{K}_+$ from Theorem \ref{geometry_of_K}
		\begin{equation}\label{eq:K_plus_decomp}
			\mc{K}_+ = \mc{H} \oplus M_+(\mc{L}) = M_+(\mc{L}_*) \oplus \mc{R}.
		\end{equation}
		Using the first decomposition of $\mc{K}_+$ together with \eqref{eq:wold_Ki}, we obtain that
		\begin{equation*}
			[\mc{H} \oplus M_+(\mc{L})] \ominus (\mc{H} \ominus \mc{M}_i) = M_+(\mc{F}_i) \oplus \mc{R}_i,
		\end{equation*}
		which yields
		\begin{equation*}
			M_+(\mc{L}) \oplus \mc{M}_i = M_+(\mc{F}_i) \oplus \mc{R}_i.
		\end{equation*}
		This immediately implies that for $i=1, \dots, k-1$
		\begin{equation}\label{eq:M_plus_L_inclusion}
			M_+(\mc{L}) \sub M_+(\mc{F}_i) \oplus \mc{R}_i.
		\end{equation}
		Similarly, using the second decomposition of $\mc{K}_+$  together with \eqref{eq:wold_Ki}, we obtain
		\begin{equation*}
			[M_+(\mc{L}_*) \oplus \mc{R}] \ominus (\mc{H} \ominus \mc{M}_i) = M_+(\mc{F}_i) \oplus \mc{R}_i,
		\end{equation*}
		which rearranges to
		\begin{equation*}
			M_+(\mc{F}_i) = [M_+(\mc{L}_*) \oplus (\mc{R} \ominus \mc{R}_i)] \ominus (\mc{H} \ominus \mc{M}_i).
		\end{equation*}
		This yields
		\begin{equation}\label{eq:M_plus_F}
			M_+(\mc{F}_i) \sub M_+(\mc{L}_*) \oplus (\mc{R} \ominus \mc{R}_i).
		\end{equation}
		By utilizing \eqref{eq:K_plus_decomp},\eqref{eq:M_plus_L_inclusion},  \eqref{F_iF_i+1} and \eqref{eq:M_plus_F}, we obtain the following sequence of inclusions
		\begin{align}
			M_+(\mc{L}) &\sub M_+(\mc{L}_*) \oplus \mc{R}\label{eq:in_0} \\
			M_+(\mc{L}) &\sub M_+(\mc{F}_1) \oplus \mc W_1,\label{eq:in_1} \\
			M_+(\mc{F}_i) &\sub M_+(\mc{F}_{i+1}) \oplus \mc W_{i+1} \quad \text{for } i=1, \dots, k-2,\label{eq:in_2} \\
			M_+(\mc{F}_{k-1}) &\sub M_+(\mc{L}_*) \oplus \mc W_k.\label{eq:in_3}
		\end{align}
		For any wandering subspace $\mc W$ for $U$, we know that $M(\mc W) = \bigvee_{n \geq 0} U^{-n} M_+(\mc W)$. Furthermore, $U^{-n}$ maps $\mc{R}_i$ and $\mc{R}$ onto themselves, and leaves each $\mc W_i$ invariant. Applying $U^{-n}$ to the inclusions \eqref{eq:in_0}--\eqref{eq:in_3} for $n \ge 0$, we obtain
		\begin{align}
			M(\mc{L}) &\sub M(\mc{L}_*) \oplus \mc{R}\label{eq:inc_0} \\
			M(\mc{L}) &\sub M(\mc{F}_1) \oplus \mc W_1, \label{eq:inc_1} \\
			M(\mc{F}_i) &\sub M(\mc{F}_{i+1}) \oplus \mc W_{i+1} \quad \text{for } i=1, \dots, k-2, \label{eq:inc_2} \\
			M(\mc{F}_{k-1}) &\sub M(\mc{L}_*) \oplus \mc W_k. \label{eq:inc_3}
		\end{align}
		
		Let $P^{\mc{L}}, P^{\mc{L}_*}, P^{\mc{F}_i}$ (for $i=1, \dots, k-1$), $P_{\mc{R}}, $ and $P_{\mc W_i}$ (for $i=1, \dots, k$) denote the orthogonal projections from $\mc{K}$ onto $M(\mc{L}), M(\mc{L}_*), M(\mc{F}_i), \mc{R}$, and $\mc W_i$, respectively. From  \eqref{eq:in_0}--\eqref{eq:inc_3}, we deduce the following projection relations
		\begin{align}
			P^{\mc{L}_*} M_+(\mc{L})
			&\sub M_+(\mc{L}_*),
			\label{eq:inclusion-F0} \\
			P^{\mc{F}_1} M_+(\mc{L})
			&\sub M_+(\mc{F}_1),
			\label{eq:inclusion-F1} \\
			P^{\mc{F}_{i+1}} M_+(\mc{F}_i)
			&\sub M_+(\mc{F}_{i+1}),
			\quad (i=1,\dots,k-2),
			\label{eq:inclusion-Fi} \\
			P^{\mc{L}_*} M_+(\mc{F}_{k-1})
			&\sub M_+(\mc{L}_*).
			\label{eq:inclusion-Lstar}
		\end{align}
		For an arbitrary vector $l \in M(\mc{L})$, using equation \eqref{eq:inc_1}, we can write $l = P^{\mc{F}_1}l \oplus P_{\mc{W}_1}l$. Iterating this orthogonal decomposition through the chain of inclusions given in equation \eqref{eq:inc_2}, we obtain
		\begin{align*}
			l &= P^{\mc{F}_2} P^{\mc{F}_1} l \oplus P_{\mc W_2} P^{\mc{F}_1} l \oplus P_{\mc W_1} l \\
			&= P^{\mc{F}_3} P^{\mc{F}_2} P^{\mc{F}_1} l \oplus P_{\mc W_3} P^{\mc{F}_2} P^{\mc{F}_1} l \oplus P_{\mc W_2} P^{\mc{F}_1} l \oplus P_{\mc W_1} l \\
			&\;\,\vdots \\
			&= P^{\mc{F}_{k-1}} \dots P^{\mc{F}_1} l \oplus P_{\mc W_{k-1}} P^{\mc{F}_{k-2}} \dots P^{\mc{F}_1} l \oplus \dots \oplus P_{\mc W_2} P^{\mc{F}_1} l \oplus P_{\mc W_1} l.
		\end{align*}
		Finally, since $M(\mc{F}_{k-1}) \sub M(\mc{L}_*) \oplus \mc W_k$, we arrive at the full decomposition
		\begin{equation*}
			l = P^{\mc{L}_*} P^{\mc{F}_{k-1}} \dots P^{\mc{F}_1} l \oplus P_{\mc W_k} P^{\mc{F}_{k-1}} \dots P^{\mc{F}_1} l \oplus \dots \oplus P_{\mc W_2} P^{\mc{F}_1} l \oplus P_{\mc W_1} l.
		\end{equation*}
		As $\mc{R} = \mc W_k \oplus \mc W_{k-1} \oplus \dots \oplus \mc W_1$, this yields explicit formulas for the projections onto $M(\mc{L}_*)$ and $\mc{R}$
		\begin{align}
			P^{\mc{L}_*} l &= P^{\mc{L}_*} P^{\mc{F}_{k-1}} \dots P^{\mc{F}_1} l, \label{eq:proj_L_star} \\
			P_{\mc{R}} l &= P_{\mc W_k} P^{\mc{F}_{k-1}} \dots P^{\mc{F}_1} l \oplus \dots \oplus P_{\mc W_2} P^{\mc{F}_1} l \oplus P_{\mc W_1} l. \label{eq:proj_R}
		\end{align}
		The inclusion relations \eqref{eq:inc_0}--\eqref{eq:inc_3} imply the following relations
		\begin{align*}
			P_{\mc R} l &= (I - P^{\mc L_*}) l,\quad l \in M(\mc{L})\\
			P_{\mc W_1} l &= (I - P^{\mc{F}_1}) l, \quad l \in M(\mc{L})\\
			P_{\mc W_{i+1}} f_i &= (I - P^{\mc{F}_{i+1}}) f_i \quad \text{for } i = 1, \dots, k-2,
			\; f_i \in M(\mc F_i), \\
			P_{\mc W_k} f_{k-1} &= (I - P^{\mc{L}_*}) f_{k-1} \quad f_{k-1} \in M(\mc F_{k-1}) .
		\end{align*}
		Furthermore, from  equation \eqref{eq:cnu_condition}, given in  Theorem \ref{geometry_of_K}, we have
		\begin{equation}\label{eq:closure_PR}
			\ov{P_{\mc{R}} M(\mc{L})} = \ov{P_{\mc{R}} (M(\mc{L}) \vee M(\mc{L}_*))} = \mc{R}.
		\end{equation}
		Using \eqref{eq:closure_PR}, \eqref{eq:R_decomp} and \eqref{eq:proj_R}, we deduce the following subspace identities
		\begin{gather}
			\ov{P_{\mc W_1} M(\mc{L})} = \mc W_1, \\
			\ov{P_{\mc W_{i+1}} P^{\mc{F}_i} \dots P^{\mc{F}_1} M(\mc{L})} = \mc W_{i+1} \quad \text{for } i=1, \dots, k-2, \\
			\ov{P_{\mc W_k} P^{\mc{F}_{k-1}} \dots P^{\mc{F}_1} M(\mc{L})} = \mc W_k.
		\end{gather}
		Since $M(\mc{L}) \sub M(\mc{F}_1) \oplus \mc W_1$, it follows that $P^{\mc{F}_1} M(\mc{L}) \sub M(\mc{F}_1)$. Consequently,
		\begin{equation*}
			\ov{P_{\mc W_2} P^{\mc{F}_1} M(\mc{L})} \sub \ov{P_{\mc W_2} M(\mc{F}_1)} \sub \mc W_2 \implies \ov{P_{\mc W_2} M(\mc{F}_1)} = \mc W_2.
		\end{equation*}
		Proceeding iteratively, we obtain
		\begin{gather}
			\ov{P_{\mc W_{i+1}} M(\mc{F}_i)} = \mc W_{i+1} \quad \text{for } i=1, \dots, k-2, \\
			\ov{P_{\mc W_k} M(\mc{F}_{k-1})} = \mc W_k.
		\end{gather}
		
		
		Let \( \Phi^{\mc L_*}, \Phi^{\mc L} \), and \( \Phi^{\mc F_i} \) denote the Fourier representations associated with the bilateral shift operators on \( M(\mc L_*), M(\mc L) \), and \( M(\mc F_i) \), respectively. We now introduce the following contraction operators
		\begin{align}
			P^{\mc{L}_*}\big|_{M(\mc{L})}
			&: M(\mc{L}) \longrightarrow M(\mc{L}_*),
			\label{eq:map-Lstar-from-L} \\
			P^{\mc{F}_1}\big|_{M(\mc{L})}
			&: M(\mc{L}) \longrightarrow M(\mc{F}_1),
			\label{eq:map-F1-from-L} \\
			P^{\mc{F}_{i+1}}\big|_{M(\mc{F}_i)}
			&: M(\mc{F}_i) \longrightarrow M(\mc{F}_{i+1}),
			\quad (i=1,\dots,k-2),
			\label{eq:map-Fi+1-from-Fi} \\
			P^{\mc{L}_*}\big|_{M(\mc{F}_{k-1})}
			&: M(\mc{F}_{k-1}) \longrightarrow M(\mc{L}_*).
			\label{eq:map-Lstar-from-Fk-1}
		\end{align}
		Since the subspaces $M(\mc{L})$, $M(\mc{L}_*)$, and $M(\mc{F}_i)$ reduce $U$, their associated orthogonal projections $P^{\mc{L}}$, $P^{\mc{L}_*}$, and $P^{\mc{F}_i}$ commute with $U$. Consequently, we have
		\begin{align}
			P^{\mc{L}_*}\big|_{M(\mc{L})} \, U\big|_{M(\mc{L})}
			= (P^{\mc{L}_*} U)\big|_{M(\mc{L})}  
			= (U P^{\mc{L}_*})\big|_{M(\mc{L})}
			= U\big|_{M(\mc{L}_*)} \, P^{\mc{L}_*}\big|_{M(\mc{L})}.
			\label{eq:intertwine-L}
		\end{align}
		Similarly, we have
		\begin{align}
			P^{\mc{F}_1}\big|_{M(\mc{L})} \, U\big|_{M(\mc{L})}
			&= U\big|_{M(\mc{F}_1)} \, P^{\mc{F}_1}\big|_{M(\mc{L})},
			\label{eq:intertwine-F1} \\
			P^{\mc{F}_{i+1}}\big|_{M(\mc{F}_i)} \, U\big|_{M(\mc{F}_i)}
			&= U\big|_{M(\mc{F}_{i+1})} \, P^{\mc{F}_{i+1}}\big|_{M(\mc{F}_i)},
			\quad (i=1,\dots,k-2),
			\label{eq:intertwine-Fi} \\
			P^{\mc{L}_*}\big|_{M(\mc{F}_{k-1})} \, U\big|_{M(\mc{F}_{k-1})}
			&= U\big|_{M(\mc{L}_*)} \, P^{\mc{L}_*}\big|_{M(\mc{F}_{k-1})}.
			\label{eq:intertwine-last}
		\end{align}
		By applying the Lemma \ref{fourier_rep} to the contractions given in \eqref{eq:map-Lstar-from-L}--\eqref{eq:map-Lstar-from-Fk-1}, all the required conditions follow from the equations \eqref{eq:inclusion-F0}--\eqref{eq:inclusion-Lstar} and \eqref{eq:intertwine-L}--\eqref{eq:intertwine-last}. Therefore, there exist contractive analytic functions $\{\mc{L}, \mc{L}_*, \Theta_{\mc{L}}(z)\}$, $\{\mc{L}, \mc{F}_1, \Theta_1(z)\}$, $\{\mc{F}_i, \mc{F}_{i+1}, \Theta_{i+1}(z)\}$ (for $i=1, \dots, k-2$), and $\{\mc{F}_{k-1}, \mc{L}_*, \Theta_k(z)\}$ such that
		\begin{align}
			\Phi^{\mc{L}_*} P^{\mc{L}_*} l &= \Theta_{\mc{L}} \Phi^\mc{L} l, \quad l \in M(\mc{L}), \label{eq:theta_L} \\
			\Phi^{\mc{F}_1} P^{\mc{F}_1} l &= \Theta_1 \Phi^\mc{L} l, \quad l \in M(\mc{L}), \\
			\Phi^{\mc{F}_{i+1}} P^{\mc{F}_{i+1}} f_i &= \Theta_{i+1} \Phi^{\mc{F}_i} f_i, \quad f_i \in M(\mc{F}_i), \\
			\Phi^{\mc{L}_*} P^{\mc{L}_*} f_{k-1} &= \Theta_k \Phi^{\mc{F}_{k-1}} f_{k-1}, \quad f_{k-1} \in M(\mc{F}_{k-1}).\label{eq:theta L_*}
		\end{align}
		Applying $\Phi^{\mc{L}_*}$ to \eqref{eq:proj_L_star} yields
		\begin{align*}
			\Phi^{\mc{L}_*} P^{\mc{L}_*} l &= \Phi^{\mc{L}_*} P^{\mc{L}_*} (\Phi^{\mc{F}_{k-1}})^* \Phi^{\mc{F}_{k-1}} P^{\mc{F}_{k-1}} \dots (\Phi^{\mc{F}_1})^* \Phi^{\mc{F}_1} P^{\mc{F}_1} l \\
			&= \Theta_k \Theta_{k-1} \dots \Theta_1 \Phi^\mc{L} l.
		\end{align*}
		Comparing this with \eqref{eq:theta_L}, we obtain the factorization
		\begin{equation}\label{theta_l_factorization}
			\Theta_{\mc{L}} = \Theta_k \Theta_{k-1} \dots \Theta_1.
		\end{equation}
		Using the relation \( \ov{P_{\mc R} M(\mc L)} = \mc R \), we define a unitary operator
		\(\Phi_{\mc R} : \mc R \longrightarrow \ov{\Delta_{\mc L} L^2(\mc L)}\) by
		\begin{equation*}
			\Phi_{\mc{R}}(P_{\mc{R}} l) = \Delta_{\mc{L}} \Phi^\mc{L} l, \quad \text{where } \Delta_{\mc{L}}(t) = [I - \Theta_{\mc{L}}(e^{it})^* \Theta_{\mc{L}}(e^{it})]^{1/2}.
		\end{equation*}
		To verify the isometry property, observe that
		\begin{align*}
			\|P_{\mc{R}} l\|^2 &= \|(I - P^{\mc{L}_*}) l\|^2 = \|\Phi^\mc{L} l\|^2 - \|\Phi^{\mc{L}_*} P^{\mc{L}_*} l\|^2 \\
			&= \|\Phi^\mc{L} l\|^2 - \|\Theta_{\mc{L}} \Phi^\mc{L} l\|^2 \\
			&= \frac{1}{2\pi} \int_0^{2\pi} \left[ \|(\Phi^\mc{L} l)(t)\|_{\mc{L}}^2 - \|\Theta_{\mc{L}}(e^{it})(\Phi^\mc{L} l)(t)\|_{\mc{L}_*}^2 \right] dt \\
			&= \frac{1}{2\pi} \int_0^{2\pi} \|\Delta_{\mc{L}}(t)(\Phi^\mc{L} l)(t)\|_{\mc{L}}^2 \, dt = \|\Delta_{\mc{L}} \Phi^\mc{L} l\|^2.
		\end{align*}
		Analogous reasoning provides the following unitary maps for the subspaces \(\mc W_1,\dots,\mc W_k\)
		\begin{align}
			\Phi_{\mc W_1} &: \mc W_1 \to \ov{\Delta_1 L^2(\mc{L})}, \quad \Phi_{\mc W_1}(P_{\mc W_1} l) = \Delta_1 \Phi^\mc{L} l. \label{eq:W_1}
		\end{align}
		For $i=1, \dots, k-2,$ we have
		\begin{align}
			\Phi_{\mc W_{i+1}} &: \mc W_{i+1} \to \ov{\Delta_{i+1} L^2(\mc{F}_i)}, \quad\Phi_{\mc W_{i+1}}(P_{\mc W_{i+1}} f_i) = \Delta_{i+1} \Phi^{\mc{F}_i} f_i \\
			\Phi_{\mc W_k} &: \mc W_k \to \ov{\Delta_k L^2(\mc{F}_{k-1})}, \quad\Phi_{\mc W_k}(P_{\mc W_k} f_{k-1}) = \Delta_k \Phi^{\mc{F}_{k-1}} f_{k-1},\label{eq:W_k}
		\end{align}
		where $\Delta_j(t) \coloneqq [I - \Theta_j^*(e^{it})\Theta_j(e^{it})]^{1/2}$ for $j=1, \dots, k$.
		Since $\mc{R} = \mc W_k \oplus \dots \oplus \mc W_1$, the operator defined by
		\begin{equation*}
			Z_k \coloneqq (\Phi_{\mc W_k} \oplus \dots \oplus \Phi_{\mc W_1}) \Phi_{\mc{R}}^{-1}
		\end{equation*}
		is a unitary map from $\ov{\Delta_{\mc{L}} L^2(\mc{L})}$ onto $\ov{\Delta_k L^2(\mc{F}_{k-1})} \oplus \dots \oplus \ov{\Delta_1 L^2(\mc{L})}$. Using \eqref{eq:proj_R}, we compute the action of $Z_k$ on elements of the form $\Delta_{\mc{L}} \Phi^{\mc{L}} l$ as follows
		\begin{align*}
			Z_k \Delta_{\mc{L}} \Phi^\mc{L} l &= (\Phi_{\mc W_k} \oplus \dots \oplus \Phi_{\mc W_1}) P_{\mc{R}} l \\
			&= \Phi_{\mc W_k} P_{\mc W_k} P^{\mc{F}_{k-1}} \dots P^{\mc{F}_1} l \oplus \dots \oplus \Phi_{\mc W_2} P_{\mc W_2} P^{\mc{F}_1} l \oplus \Phi_{\mc W_1} P_{\mc W_1} l \\
			&= \Delta_k \Theta_{k-1} \dots \Theta_1 \Phi^\mc{L} l \oplus \dots \oplus \Delta_2 \Theta_1 \Phi^\mc{L} l \oplus \Delta_1 \Phi^\mc{L} l.
		\end{align*}
		Because $\Phi^\mc{L} l$ is dense in $L^2(\mc{L})$, we can extend this to all $v \in L^2(\mc{L})$
		\begin{equation*}
			Z_k \Delta_{\mc{L}} v = \Delta_k \Theta_{k-1} \dots \Theta_1 v \oplus \dots \oplus \Delta_2 \Theta_1 v \oplus \Delta_1 v.
		\end{equation*}
		Since multiplication by \( e^{it} \) commutes with the defect operators \( \Delta_j \) and the operators \( \Theta_i \), it follows that it also commutes with \( Z_k \). We now define the unitary operator
		\[
		\Phi : \mc{K} \to \wh{\mc{K}} \coloneqq L^2(\mc{L}_*) \oplus \ov{\Delta_{\mc{L}} L^2(\mc{L})},
		\]
		by setting
		\begin{equation}\label{Phi_for_model}
			\Phi = \Phi^{\mc{L}_*} \oplus \Phi_{\mc{R}}
			= \Phi^{\mc{L}_*} \oplus Z_k^{*} (\Phi_{\mc W_k} \oplus \dots \oplus \Phi_{\mc W_1}).
		\end{equation}
		By virtue of Proposition \ref{functional_model}, the operator $\Phi$ maps the subspace $\mc{K}_+$ onto
		\[
		\wh{\mc{K}}_+ \coloneqq H^2(\mc{L}_*) \oplus \ov{\Delta_{\mc{L}} L^2(\mc{L})},
		\]
		carries the subspace $M_+(\mc{L})$ onto
		\begin{equation*}
			\wh{\mc G} \coloneqq \{\Theta_{\mc{L}} u \oplus \Delta_{\mc{L}} u : u \in H^2(\mc{L})\},
		\end{equation*}
		and, consequently, maps the subspace $\mc{H} = \mc{K}_+ \ominus M_+(\mc{L})$ onto
		\begin{equation*}
			\wh{\mc{H}} \coloneqq \wh{\mc{K}}_+ \ominus \wh{\mc G}.
		\end{equation*}
		Under this unitary transformation, the operator $T$ is unitarily equivalent to the model operator $\wh{T}$, whose adjoint is defined by
		\begin{equation}\label{model_operator}
			\wh{T}^*(u \oplus v) = e^{-it}[u - u(0)] \oplus e^{-it} v,\quad u \oplus v \in \wh{\mc{H}}.
		\end{equation}
		We now compute the images of the invariant subspaces $\mc{M}_i$ under the unitary transformation $\Phi$. Recall that
		\begin{equation*}
			\mc{M}_i = (M_+(\mc{F}_i) \oplus \mc{R}_i) \ominus M_+(\mc{L}) \quad \text{for } i=1, \dots, k-1.
		\end{equation*}
		Applying the unitary transformation $\Phi$, which preserves orthogonal direct sums and orthogonal complements, we obtain
		\begin{equation}\label{eq:phi_M2_decomp}
			\Phi(\mc M_i) = \Phi(M_+(\mc{F}_i) \oplus \mc R_i) \ominus \Phi(M_+(\mc{L})).
		\end{equation}
		Now, utilizing the relations \eqref{eq:inc_2} and \eqref{eq:inc_3}, we obtain
		\begin{align*}
			\Phi(M_+(\mc{F}_i))
			&= \Phi \Big\{ P^{\mc{F}_{i+1}} f_i\oplus P_{\mc{W}_{i+1}} f_i :\,
			f_i \in M_+(\mc{F}_i)
			\Big\} \\
			&= \Phi \Big\{
			P^{\mc{L}_*} P^{\mc{F}_{k-1}} \cdots P^{\mc{F}_{i+1}} f_i
			\oplus
			P_{\mc{W}_k}P^{\mc{F}_{k-1}} \cdots P^{\mc{F}_{i+1}} f_i
			\oplus \cdots
			\oplus
			P_{\mc{W}_{i+1}} f_i  \\
			&\qquad \qquad\qquad
			:\,
			f_i \in M_+(\mc{F}_i)
			\Big\} \\
			&=
			\Big\{
			\Theta_k \cdots \Theta_{i+1} u_i
			\oplus
			Z_k^{*}
			\big(
			\Delta_k \Theta_{k-1} \cdots \Theta_{i+1} u_i
			\oplus \cdots
			\oplus
			\Delta_{i+1} u_i   \\
			&\qquad
			\oplus
			\underbrace{0 \oplus \cdots \oplus 0}_{i \text{ terms}}
			\big)
			:\,
			u_i \in H^2(\mc{F}_i)
			\Big\}.
		\end{align*}
		where $u_i = \Phi^{\mc{F}_i} f_i \in H^2(\mc{F}_i)$. Since,
		by definition, $\mc R_i = \mc W_i \oplus\cdots \oplus \mc W_1$,  we have
		\begin{align*}
			\Phi(\mc R_i) &= \Phi(\mc W_i \oplus \mc W_{i-1} \oplus \cdots \oplus \mc W_1) \\
			&= Z_k^{*} \left( \underbrace{\{0\} \oplus \cdots \oplus \{0\}}_{k-i \text{ terms}} \oplus \ov{\im(\Delta_i)} \oplus \cdots \oplus \ov{\im(\Delta_1)} \right).
		\end{align*}
		Summing these orthogonal components, we establish the image of the combined space
		\begin{align*}
			\Phi(&M_+(\mc{F}_i) \oplus \mc R_i)\\
			&= \Big\{
			\Theta_k \cdots \Theta_{i+1} u_i \oplus Z_k^{*} \big(
			\Delta_k \Theta_{k-1} \cdots \Theta_{i+1} u_i
			\oplus \cdots \oplus \Delta_{i+1} u_i  
			\oplus v_i \oplus \cdots \oplus v_1
			\big)\notag \\    
			&\qquad :  u_i \in H^2(\mc{F}_i), \
			v_j \in \ov{\im(\Delta_j)}
			\text{ for } j=1,\dots,i
			\, \Big\}.
		\end{align*}
		Substituting these explicit forms back into equation \eqref{eq:phi_M2_decomp} and noting that $\Phi(M_+(\mc{L})) =  \{\Theta_{\mc{L}} u \oplus \Delta_{\mc{L}} u : u \in H^2(\mc{L})\}$, we arrive at the complete generalized expression for $\wh{\mc{M}}_i \coloneqq \Phi(\mc{M}_i)$:
		\begin{align}
			\wh{\mc{M}}_i
			= \Big\{ \,
			&\Theta_k \cdots \Theta_{i+1} u_i \oplus Z_k^{*} \big(
			\Delta_k \Theta_{k-1} \cdots \Theta_{i+1} u_i
			\oplus \cdots \oplus \Delta_{i+1} u_i
			\oplus v_i \oplus \cdots  \label{eq:phi_Mi_detailed}\\
			& \oplus v_1
			\big): u_i \in H^2(\mc{F}_i), \
			v_j \in \ov{\im(\Delta_j)}
			\text{ for } j=1,\dots,i
			\, \Big\} \notag\\
			&\ominus \Big\{
			\Theta_{\mc{L}} u \oplus \Delta_{\mc{L}} u
			: u \in H^2(\mc{L})
			\Big\}.\notag
		\end{align}
		In particular, for $i = k-1$, we get
		\begin{multline}\label{eq:phi_Mk_minus_1}
			\wh{\mc M}_{k-1} = \left\{ \Theta_k u_{k-1} \oplus Z_k^{*} (\Delta_k u_{k-1} \oplus v_{k-1} \oplus \cdots \oplus v_1) :
			u_{k-1} \in H^2(\mc{F}_{k-1}),\right.\\ \left. \ v_j \in \ov{\im( \Delta_j)} \text{ for } j=1,\dots,k-1 \right\} \ominus \left\{ \Theta_{\mc{L}} u \oplus \Delta_{\mc{L}} u : u \in H^2(\mc{L}) \right\}.
		\end{multline}
		Moreover, we evaluate the orthogonal complement $\wh{\mc N}_i$ as follows
		\begin{equation*}
			\begin{aligned}
				\wh{\mc N}_i &= \left[ \left( H^2(\mc L_*) \oplus \ov{\Delta_{\mc L} L^2(\mc L)} \right) \ominus \{ \Theta_{\mc L} u \oplus \Delta_{\mc L} u : u \in H^2(\mc L) \} \right] \\
				&\quad \ominus \left[ \left\{ \Theta_k \cdots \Theta_{i+1} u_i \oplus Z_k^{*}(\Delta_k\Theta_{k-1} \cdots \Theta_{i+1} u_i \oplus \cdots \oplus \Delta_{i+1} u_i \oplus v_i \oplus \cdots \oplus v_1) \right.\right. \\ &\qquad \left.\left.:
				u_i \in H^2(\mc F_i), \ v_j \in \ov{\im(\Delta_j)} \right\} \vphantom{Z_k^{*}} \ominus \{ \Theta_{\mc L} u \oplus \Delta_{\mc L} u : u \in H^2(\mc L) \} \right] \\
				&= \left[ H^2(\mc L_*) \oplus Z_k^{*} \left( \ov{\Delta_k L^2(\mc F_{k-1})} \oplus \cdots \oplus \ov{\Delta_{i+1} L^2(\mc F_i)} \oplus \underbrace{\{0\} \oplus \cdots \oplus \{0\}}_{i \text{ times}} \right) \right] \\
				&\quad \ominus \left\{ \Theta_k \cdots \Theta_{i+1} u_i \oplus Z_k^{*}(\Delta_k \Theta_{k-1} \cdots  \Theta_{i+1} u_i \oplus \cdots \oplus \Delta_{i+1} u_i \oplus 0 \oplus \cdots \oplus 0) \right. \\
				&\qquad \qquad\left.: u_i \in H^2(\mc F_i) \right\}.
			\end{aligned}
		\end{equation*}

		Recall from Chapter VI of \cite{NFBK10} that Sz.-Nagy and Foia\c{s} proved that $\Theta_{\mc L}$ (defined in \eqref{eq:theta_L}) coincides with the characteristic function $\Theta_T$ of the contraction $T$. Hence, $\wt{\Theta}$ coincides with $\Theta_{\mc L}$. Consequently, the $k$-regular factorization given by \eqref{theta_l_factorization} yields the corresponding $k$-regular factorization
		\begin{equation*}
			\wt {\Theta}(z) = \wt {\Theta}_k(z) \cdots \wt {\Theta}_1(z) \quad \text{for } z \in \mathbb{D},
		\end{equation*}
		where each \( \{\mc{E}_i, \mc{E}_{i+1}, \wt {\Theta}_i(z)\} \) is a contractive analytic function for \( i=1,\dots,k \), with initial space \( \mc{E}_1 = \mc{E} \) and final space \( \mc{E}_{k+1} = \mc{E}_* \). Moreover, the functional model pair \( (\wh{T}, \wh{\mc H}) \) corresponds to the functional model pair \( (\wt{T}, \wt{\mc H}) \), and for \( i=1,\dots,k-1 \), the subspaces \( \wh{\mc M}_i \) and \( \wh{\mc N}_i \) correspond to \( \wt{\mc M}_i \) and \( \wt{\mc N}_i \), respectively.
		
		Conversely, let $\wt{\Theta} = \wt{\Theta}_k \cdots \wt{\Theta}_1$ be a $k$-regular factorization, where each \( \{\mc{E}_i, \mc{E}_{i+1}, {\wt{\Theta}}_i(z)\} \) is a contractive analytic function for \( i=1,\dots,k \), with initial space \( \mc{E}_1 = \mc{E} \) and final space \( \mc{E}_{k+1} = \mc{E}_* \). For each \( i = 1, \dots, k-1 \), let the subspaces \( \wt{\mc M}_i \) and \( \wt{\mc N} _i\) be defined by
		\begin{align*}
			\wt{\mc{M}}_i
			&= \Big\{
			{\wt{\Theta}}_k \cdots {\wt{\Theta}}_{i+1} u_{i+1}
			\oplus
			\wt{Z}_k^*\big(
			\wt \Delta_k {\wt{\Theta}}_{k-1} \cdots {\wt{\Theta}}_{i+1} u_{i+1}
			\oplus \cdots
			\oplus \wt \Delta_{i+1} u_{i+1}
			\oplus v_i
			\oplus \cdots \\
			& \qquad
			\oplus v_1
			\big) :  
			u_{i+1} \in H^2(\mc{E}_{i+1}), \;
			v_j \in \overline{\wt \Delta_j L^2(\mc{E}_j)}, \; j=1,\dots,i
			\Big\}
			\ominus {\wt{\mc G}},\\
			\wt{\mc N}_i&= \left[ H^2(\mc{E}_*) \oplus \wt{Z}_k^* \left( \ov{\wt \Delta_k L^2(\mc{E}_{k})} \oplus \cdots \oplus \ov{\wt \Delta_{i+1} L^2(\mc{E}_{i+1})} \oplus \{0\} \oplus \cdots \oplus \{0\} \right) \right] \\
			&\quad \ominus \left\{  {\wt{\Theta}}_k \cdots  {\wt{\Theta}}_{i+1} u_{i+1} \oplus \wt{Z}_k^* ( \wt \Delta_k  {\wt{\Theta}}_{k-1} \cdots {\wt{\Theta}}_{i+1} u_{i+1} \oplus \cdots \oplus \wt \Delta_{i+1} u_{i+1}  \right. \\
			&  \qquad \left.\oplus 0 \oplus \cdots \oplus 0 )  : u_{i+1} \in H^2(\mc{E}_{i+1}) \right\}.
		\end{align*}
		Since $\wt{Z}_k$ is unitary, we define the subspace $\wt{\mc G}_i$ by
		\begin{align*}
			\wt{\mc G}_i &\coloneqq \left\{ \wt{\Theta}_k \cdots \wt{\Theta}_{i+1} u_{i+1} + \wt{Z}_k^{*} ( \wt \Delta_k \wt{\Theta}_{k-1} \cdots \wt{\Theta}_{i+1} u_{i+1} \oplus \cdots \oplus \wt \Delta_{i+1} u_{i+1} \oplus v_i \oplus \cdots \right. \\
			&\qquad \left. \oplus v_1 )  : u_{i+1} \in H^2(\mc{E}_{i+1}), \ v_j \in \ov{\im(\wt \Delta_j)}, \ j=1,\dots,i \right\} \\
			&\supseteq \left\{ \wt{\Theta}_k \cdots \wt{\Theta}_1 u \oplus \wt{Z}_k^{*} ( \wt \Delta_k \wt{\Theta}_{k-1} \cdots \wt{\Theta}_1 u \oplus \cdots \oplus \wt \Delta_{i+1} \wt{\Theta}_i \cdots \wt{\Theta}_1 u \oplus \cdots \oplus \wt \Delta_1 u ) \right. \\
			&\qquad \left. : u \in H^2(\mc{E}) \right\} \\
			&= \{ \wt{\Theta} u \oplus \wt \Delta u : u \in H^2(\mc{E}) \} = \wt{\mc G}.
		\end{align*}
		It follows that $ \wt{\mc G}_i \supseteq \wt{\mc G},$ and we have $ \wt{\mc M}_i = \wt{\mc G}_i \ominus \wt{\mc G}.$
		Observe that
		\begin{align*}
			&(H^2(\mc{E}_*) \oplus \ov{\wt \Delta L^2(\mc{E})}) \ominus \wt{\mc G}_i \\
			&= \left[ H^2(\mc{E}_*) \oplus \ov{\wt \Delta (L^2(\mc{E}))} \right] \ominus \left\{ \wt{\Theta}_k \cdots \wt{\Theta}_{i+1} u_{i+1} + \wt{Z}_k^{*} ( \wt \Delta_k \wt{\Theta}_{k-1} \cdots \wt{\Theta}_{i+1} u_{i+1} \oplus \cdots \right. \\
			&\qquad \left.  \oplus \wt \Delta_{i+1} u_{i+1} \oplus v_i \oplus \cdots \oplus v_1 ) : u_{i+1} \in H^2(\mc{E}_{i+1}), \ v_j \in \ov{\im(\wt \Delta_j)}, \ j=1,\dots,i \right\} \\
			&= \left[ H^2(\mc{E}_*) \oplus \wt{Z}_k^{*} \left( \ov{\wt \Delta_k L^2(\mc{E}_{k})} \oplus \cdots \oplus \ov{\wt \Delta_{i+1} L^2(\mc{E}_{i+1})} \oplus \{0\} \oplus \cdots \oplus \{0\} \right) \right] \\
			&\quad \ominus \left\{ \wt{\Theta}_k \cdots \wt{\Theta}_{i+1} u_{i+1} \oplus \wt{Z}_k^{*} ( \wt \Delta_k \wt{\Theta}_{k-1} \cdots \wt{\Theta}_{i+1} u_{i+1} \oplus \cdots \oplus \wt \Delta_{i+1} u_{i+1} \oplus 0 \oplus \cdots \oplus 0 ) \right. \\
			&\qquad \left. : u_{i+1} \in H^2(\mc{E}_{i+1}) \right\} \\
			&= \wt{\mc N}_i.
		\end{align*}
		Hence we have $\wt{\mc M}_i \oplus \wt{\mc N}_i = (H^2(\mc{E}_*) \oplus \ov{\wt \Delta L^2(\mc{E})})\ominus\wt{\mc G}=\wt{\mc H}$.
		We shall prove that the subspaces $\wt{\mc M}_1, \dots, \wt{\mc M}_{k-1}$ are invariant under $\wt{T}$. To this end, it is sufficient to prove that $\wt{T}^*(\wt{\mc N}_i) \sub \wt{\mc N}_i$ for $i = 1, \dots, k-1$, where the orthocomplements are given by
		\begin{align*}
			\wt{\mc N}_i &= \left[ H^2(\mc{E}_*) \oplus \wt{Z}_k^{*} \left( \ov{\wt \Delta_k L^2(\mc{E}_{k})} \oplus \cdots \oplus \ov{\wt \Delta_{i+1} L^2(\mc{E}_{i+1})} \oplus \{0\} \oplus \cdots \oplus \{0\} \right) \right] \\
			&\quad \ominus \left\{ \wt{\Theta}_k \cdots \wt{\Theta}_{i+1} u_{i+1} \oplus \wt{Z}_k^{*} ( \wt \Delta_k \wt{\Theta}_{k-1} \cdots \wt{\Theta}_{i+1} u_{i+1} \oplus \cdots \oplus \wt \Delta_{i+1} u_{i+1} \oplus 0 \oplus \cdots \right. \\
			&\qquad \left. \quad \oplus 0 )  : u_{i+1} \in H^2(\mc{E}_{i+1}) \right\}
		\end{align*}
		If $x \in \wt{\mc N}_i$, there exist elements $u \in H^2(\mc E_*)$ and $v_j \in \ov{\im(\wt \Delta_j)}$ for $j = i+1, \dots, k$ such that
		\begin{equation*}
			x = u \oplus \wt{Z}_k^{*}(v_k \oplus \dots \oplus v_{i+1} \oplus 0 \oplus \dots \oplus 0),
		\end{equation*}
		satisfying the orthogonality condition for all $u_{i+1} \in H^2(\mc E_{i+1})$
		\begin{align*}
			\langle u \oplus \wt{Z}_k^{*}(v_k \oplus \dots \oplus v_{i+1} \oplus 0 \oplus \cdots \oplus 0), \ \wt{\Theta}_k \cdots \wt{\Theta}_{i+1} u_{i+1} \oplus \wt{Z}_k^{*}(\wt \Delta_k \wt{\Theta}_{k-1} \cdots\\ \wt{\Theta}_{i+1} u_{i+1} \oplus \dots \oplus \wt \Delta_{i+1} u_{i+1} \oplus 0 \oplus \cdots \oplus 0) \rangle = 0.
		\end{align*}
		Evaluating this inner product yields
		\begin{equation*}
			\langle \wt{\Theta}_{i+1}^* \cdots \wt{\Theta}_k^* u + \wt{\Theta}_{i+1}^* \cdots \wt{\Theta}_{k-1}^* \wt \Delta_k v_k + \dots + \wt \Delta_{i+1} v_{i+1}, \ u_{i+1} \rangle = 0.
		\end{equation*}
		Since this equality holds for all $u_{i+1} \in H^2(\mc E_{i+1})$, it implies that
		\begin{equation}\label{eq:orthogonality_i}
			\wt{\Theta}_{i+1}^* \cdots \wt{\Theta}_k^* u + \wt{\Theta}_{i+1}^* \cdots \wt{\Theta}_{k-1}^* \wt \Delta_k v_k + \dots + \wt \Delta_{i+1} v_{i+1} \perp H^2(\mc E_{i+1}).
		\end{equation}
		Applying $\wt{T}^*$ to $x$ via the model operator definition, we have
		\begin{align*}
			\wt{T}^*(x) &= e^{-it}(u - u(0)) \oplus e^{-it} \wt{Z}_k^{*}(v_k \oplus \dots \oplus v_{i+1} \oplus 0 \oplus \dots \oplus 0) \\
			&= u^\prime \oplus \wt{Z}_k^{*}(v_k^\prime \oplus \dots \oplus v_{i+1}^\prime \oplus 0 \oplus \dots \oplus 0),
		\end{align*}
		where $u^\prime(e^{it}) = e^{-it}(u(e^{it}) - u(0))$ and $v_j^\prime = e^{-it} v_j$ for $j = i+1, \dots, k$.
		To show $\wt{T}^*(x) \in \wt{\mc N}_i$, we verify that the components of $\wt{T}^*(x)$ satisfy the orthogonality condition \eqref{eq:orthogonality_i}
		\begin{align*}
			\wt{\Theta}_{i+1}^* \cdots \wt{\Theta}_k^* u^\prime &+ \wt{\Theta}_{i+1}^* \cdots \wt{\Theta}_{k-1}^* \wt \Delta_k v_k^\prime + \dots + \wt \Delta_{i+1} v_{i+1}^\prime \\
			&= e^{-it} \left( \wt{\Theta}_{i+1}^* \cdots \wt{\Theta}_k^* u + \wt{\Theta}_{i+1}^* \cdots \wt{\Theta}_{k-1}^* \wt \Delta_k v_k + \dots + \wt \Delta_{i+1} v_{i+1} \right)\\
			&\qquad - e^{-it} (\wt{\Theta}_{i+1}^* \cdots \wt{\Theta}_k^*)(e^{it}) u(0).
		\end{align*}
		By \eqref{eq:orthogonality_i}, the first term belongs to $L^2 \ominus H^2$, a property preserved under multiplication by $e^{-it}$. Since the second term involving $u(0)$ also belongs to $L^2 \ominus H^2$, the entire expression is orthogonal to $H^2(\mc E_{i+1})$.
		Thus, $\wt{T}^*(\wt{\mc N}_i) \sub \wt{\mc N}_i$ for any $i \in \{1, \dots, k-1\}$, confirming that the $\wt{\mc M}_i$ are invariant subspaces for $\wt{T}$.
		Finally, we demonstrate the nesting $\wt{\mc M}_1 \sub \wt{\mc M}_2 \sub \dots \sub \wt{\mc M}_{k-1}$. For any  $i \in \{1, \dots, k-2\}$, recall that
		\begin{multline*}
			\wt{\mc M}_{i+1} = \left\{ \wt{\Theta}_k \cdots \wt{\Theta}_{i+2} u_{i+2} \oplus \wt{Z}_k^{*} (\wt \Delta_k \wt{\Theta}_{k-1} \cdots \wt{\Theta}_{i+2} u_{i+2} \oplus \dots \oplus \wt \Delta_{i+2} u_{i+2} \oplus v_{i+1} \right. \\
			\left. \oplus  \dots \oplus v_1)  : u_{i+2} \in H^2(\mc{E}_{i+2}), \ v_j \in \ov{\im(\wt \Delta_j)} \right\} \ominus \wt{\mc G}.
		\end{multline*}
		By choosing $u_{i+2} = \wt{\Theta}_{i+1} u_{i+1}$, it is clear that the elements generating $\wt{\mc M}_i$ are contained within $\wt{\mc M}_{i+1}$. Hence, the inclusion $\wt{\mc M}_i \sub \wt{\mc M}_{i+1}$ follows directly.
	\end{proof}
	
	
	In the following theorem, we will construct a functional model for a given $k$-regular factorization of a purely contractive analytic function. Corresponding to this regular $k$-regular factorization, we will characterize the invariant subspaces of the associated model operator.

	\begin{theorem}\label{functional_model_1}
		Let \( \{\mathcal{E}, \mathcal{E}_*, \Theta(z)\} \) be a purely contractive analytic function that coincides with the characteristic function of a c.n.u.\ contraction \( T \in B(\mathcal{H}) \).
		Assume that \( k \geq 2 \) and that
		\begin{equation}\label{k_regular_fact2}
			\Theta(z) = \Theta_k(z) \cdots \Theta_1(z),
			\qquad z \in \mathbb{D},
		\end{equation}
		is a $k$-regular factorization, where
		\( \{\mc{E}_i, \mc{E}_{i+1}, \Theta_i(z)\} \)
		are contractive analytic functions for \( i = 1, \dots, k \),
		with \( \mc{E}_1 = \mc{E} \) and \( \mc{E}_{k+1} = \mc{E}_* \).
		Then the operator \( T \) is unitarily equivalent to the operator \( \boldsymbol{T} \),
		which is defined on the Hilbert space
		\[
		\boldsymbol{\mc{H}}
		\coloneqq
		\Bigl[
		H^2(\mc{E}_{k+1})
		\oplus
		\overline{\Delta_k L^2(\mc{E}_k)}
		\oplus \cdots \oplus
		\overline{\Delta_1 L^2(\mc{E}_1)}
		\Bigr]
		\ominus \boldsymbol{\mc G},
		\]
		where
		\[
		\boldsymbol{T}^*(u \oplus v_k \oplus \cdots \oplus v_1)
		=
		e^{-it}\bigl[u- u(0)\bigr]
		\oplus
		e^{-it} v_k
		\oplus \cdots \oplus
		e^{-it} v_1,
		\]
		and
		\( \boldsymbol{\mc G}
		=
		\left\{
		\Theta_k \cdots \Theta_1 u
		\oplus
		\Delta_k \Theta_{k-1} \cdots \Theta_1 u
		\oplus \cdots \oplus
		\Delta_1 u
		:
		u \in H^2(\mc{E}_1)
		\right\},~ \Delta_i =(I - \Theta_i^* \Theta_i)^{1/2}
		.\)
		Moreover, for \( i = 1, \dots, k-1 \), under the $k$-regular factorization \eqref{k_regular_fact2}, the invariant subspaces are described by
		\begin{align*}
			\boldsymbol{\mc{M}}_i
			=
			\Big\{
			&\Theta_k \cdots \Theta_{i+1} u_{i+1}
			\oplus
			\Delta_k \Theta_{k-1} \cdots \Theta_{i+1} u_{i+1}
			\oplus \cdots \oplus
			\Delta_{i+1} u_{i+1} \\
			&\oplus v_i \oplus \cdots \oplus v_1
			:
			u_{i+1} \in H^2(\mc{E}_{i+1}),
			\;
			v_j \in \overline{\Delta_j L^2(\mc{E}_j)},
			\ j = 1, \dots, i
			\Big\}
			\ominus \boldsymbol{\mc G},
		\end{align*}
		and
		\[
		\boldsymbol{\mc{M}}_1 \sub \cdots \sub \boldsymbol{\mc{M}}_{k-1}.
		\]
	\end{theorem}

	\begin{proof}
		The result follows directly from the proof of Theorem \ref{k_r_invariant_theorem} upon replacing the unitary operator \( \Phi \) defined in \eqref{Phi_for_model} by the unitary operator
		\[
		\boldsymbol{\Phi} : \mc{K} \longrightarrow \boldsymbol{\mc{K}}
		\coloneqq
		L^2(\mc{E}_*)
		\oplus
		\overline{\Delta_k L^2(\mc{E}_k)}
		\oplus \cdots \oplus
		\overline{\Delta_1 L^2(\mc{E}_1)},
		\]
		defined by
		\[
		\boldsymbol{\Phi}
		=
		\Phi^{\mc{L}_*}
		\oplus
		\bigl(
		\Phi_{\mc{W}_k}
		\oplus \cdots \oplus
		\Phi_{\mc{W}_1}
		\bigr).
		\]
	\end{proof}

	In the case of $2$-regular factorizations, Sz.-Nagy and Foia\c{s} made the following remark concerning two functional models, which extends naturally to the setting of $k$-regular factorizations.
	\begin{remark}
		From the preceding discussion, we observe that two functional model pairs arise naturally, namely
		$(\wt{\mc H}, \wt{T})$ and $(\boldsymbol{\mc H}, \boldsymbol{T})$.
		The model pair $(\boldsymbol{\mc H}, \boldsymbol{T})$ is particularly convenient when working with a fixed $k$-regular factorization.
		In contrast, when comparing two distinct $k$-regular factorizations, the model pair $(\wt{\mc H}, \wt{T})$ is more appropriate.
	\end{remark}

	\begin{remark}
		The above Theorem \ref{k_r_invariant_theorem} extends the classical result of Sz.-Nagy and Foiaș (see Chapter VI of \cite{NFBK10}), which establishes a correspondence between invariant subspaces of a c.n.u.\ contraction and regular ($2$-regular) factorizations of its characteristic function. While the classical case corresponds to $k=2$, the present result treats chains of invariant subspaces
		\[
		\mc M_1 \sub \cdots \sub \mc M_{k-1},
		\]
		and shows that they are in correspondence with $k$-regular factorizations
		\[
		\Theta = \Theta_k \cdots \Theta_1,
		\]
		thereby providing a higher-order structural generalization of the Sz.-Nagy–Foiaș theory.
	\end{remark}

	Next, we recall the fundamental structure of the Sz.-Nagy--Foia\c{s} functional model associated with a contractive analytic function, as presented in Chapter VI of \cite{NFBK10}. The corresponding model operator $T_\Theta$ acts on the Hilbert space
	\begin{equation*}
		\mc{H}(\Theta)
		\coloneqq
		\left[
		H^2(\mc{E}_*) \oplus \overline{\Delta_{\Theta} L^2(\mc{E})}
		\right]
		\ominus
		\left\{
		\Theta w \oplus \Delta_{\Theta} w : w \in H^2(\mc{E})
		\right\}.
	\end{equation*}
	The adjoint operator $T_\Theta^*$ is given by
	\begin{equation*}
		T_\Theta^*(u \oplus v)
		\coloneqq
		e^{-it}\bigl(u - u(0)\bigr) \oplus e^{-it} v,\quad u\oplus v\in \mc{H}(\Theta).
	\end{equation*}
	
	\begin{theorem}\label{unitary_equivalent}
		Let \( \{\mathcal{E}, \mathcal{E}_*, \Theta(z)\} \) be a purely contractive analytic function admitting a $k$-regular factorization
		\[
		\Theta(z) = \Theta_k(z) \cdots \Theta_1(z),
		\qquad z \in \mathbb{D},
		\]
		where \( \{\mc{E}_i, \mc{E}_{i+1}, \Theta_i(z)\} \) are contractive analytic functions for \( i = 1, \dots, k \), with \( \mc{E}_1 = \mc{E} \) and \( \mc{E}_{k+1} = \mc{E}_* \).
		Let \( \mc{M}_1 \sub \cdots \sub \mc{M}_{k-1} \) denote the corresponding chain of invariant subspaces for the model contraction $T_\Theta$ acting on $\mc{H}(\Theta)$. Define the successive orthogonal subspaces as
		\[
		\mc{H}_1 \coloneqq \mc{M}_1, \quad
		\mc{H}_i \coloneqq \mc{M}_i \ominus \mc{M}_{i-1} \ (i=2,\dots,k-1), \quad
		\mc{H}_k \coloneqq \mc{H}(\Theta) \ominus \mc{M}_{k-1}.
		\]
		With respect to the orthogonal direct sum decomposition
		\[
		\mc{H}(\Theta) = \mc{H}_1 \oplus \mc{H}_2 \oplus \cdots \oplus \mc{H}_k,
		\]
		the operator $T_\Theta$ admits the following block upper triangular matrix representation
		\begin{equation*}
			T_\Theta = \begin{bmatrix}
				A_1 & * & \cdots & * \\
				0 & A_2 & \cdots & * \\
				\vdots & \vdots & \ddots & \vdots \\
				0 & 0 & \cdots & A_k
			\end{bmatrix}.
		\end{equation*}
		Then, for each $i = 1, \dots, k$, the diagonal block $A_i \coloneqq P_{\mc{H}_i} T_\Theta |_{\mc{H}_i}$ is unitarily equivalent to $T_{\Theta_i}$, the standard functional model operator associated with the contractive analytic function $\Theta_i$.
	\end{theorem}
	
	\begin{proof}
		We begin by establishing the necessary notation. For $i=1,\dots,k$, we define the functions $\Phi_i = \Theta_k \cdots \Theta_i$ and the associated operators
		\begin{equation*}
			\Lambda_i = [\Delta_k \Theta_{k-1} \cdots \Theta_i, \dots, \Delta_{i+1} \Theta_i, \Delta_i]^T.
		\end{equation*}
		A direct computation yields
		\begin{align*}
			\Lambda_i^* \Lambda_i &= [\Theta_i^* \cdots \Theta_{k-1}^* \Delta_k, \dots, \Theta_i^* \Delta_{i+1}, \Delta_i] [\Delta_k \Theta_{k-1} \cdots \Theta_i, \dots, \Delta_{i+1} \Theta_i, \Delta_i]^T \nonumber \\
			&= \Theta_i^* \cdots \Theta_{k-1}^* \Delta_k^2 \Theta_{k-1} \cdots \Theta_i + \cdots + \Theta_i^* \Delta_{i+1}^2 \Theta_i + \Delta_i^2 \nonumber \\
			&= -\Theta_i^* \cdots \Theta_k^* \Theta_k \cdots \Theta_i + I_{H^2(\mc E_i)} \nonumber \\
			&= I_{H^2(\mc E_i)} - \Phi_i^* \Phi_i.
		\end{align*}
		Consequently, we obtain the fundamental identity $\Lambda_i^* \Lambda_i + \Phi_i^* \Phi_i = I_{H^2(\mc E_i)}$.
		Recall the structural decomposition established in Theorem \ref{k_r_invariant_theorem}
		\begin{align*}
			\mc{H}_1 = \mc M_1& = \left\{ \Phi_2 u_2 \oplus Z_k^*(\Lambda_2 u_2 \oplus v_1) : u_2 \in H^2(\mc{E}_2), \ v_1 \in \overline{\im(\Delta_1)} \right\} \\&\quad\quad\ominus \left\{ \Theta u \oplus \Delta_{\Theta} u : u \in H^2(\mc{E}) \right\},
		\end{align*}
		where $Z_k$ denotes the unitary operator associated with the $k$-regular factorization $\Theta = \Theta_k \cdots \Theta_1$. Assume $h_1 \in \mc{H}_1$. Then $h_1$ admits the representation $h_1 = \Phi_2 u_2 \oplus Z_k^*(\Lambda_2 u_2 \oplus v_1)$ for some $u_2 \in H^2(\mc{E}_2)$ and $v_1 \in \overline{\im(\Delta_1)}$, satisfying the orthogonality condition
		\begin{equation*}
			\langle \Phi_2 u_2 \oplus Z_k^*(\Lambda_2 u_2 \oplus v_1), \ \Theta u \oplus \Delta_{\Theta} u \rangle = 0 \quad \text{for all } u \in H^2(\mc{E}).
		\end{equation*}
		Since $Z_k (\Delta_{\Theta} u) = \Lambda_1 u = \Lambda_2 \Theta_1 u \oplus \Delta_1 u$ and $\Phi_1 = \Phi_2 \Theta_1$, this sequentially implies the following relations
		\begin{align*}
			0
			&= \langle \Phi_2 u_2, \Phi_1 u \rangle
			+ \langle \Lambda_2 u_2 \oplus v_1, \Lambda_2 \Theta_1 u \oplus \Delta_1 u \rangle  \\
			&= \langle \Phi_2 u_2, \Phi_2 \Theta_1 u \rangle
			+ \langle \Lambda_2 u_2, \Lambda_2 \Theta_1 u \rangle
			+ \langle v_1, \Delta_1 u \rangle  \\
			&= \langle \Phi_2^* \Phi_2 u_2, \Theta_1 u \rangle
			+ \langle \Lambda_2^* \Lambda_2 u_2, \Theta_1 u \rangle
			+ \langle v_1, \Delta_1 u \rangle  \\
			&= \langle (\Phi_2^* \Phi_2 + \Lambda_2^* \Lambda_2) u_2, \Theta_1 u \rangle
			+ \langle v_1, \Delta_1 u \rangle \\
			&= \langle u_2, \Theta_1 u \rangle + \langle v_1, \Delta_1 u \rangle.
		\end{align*}
		This orthogonality precisely means that $u_2 \oplus v_1 \in \mc{H}(\Theta_1)$. Consequently, the space $\mc{H}_1$ is completely characterized as
		\begin{equation*}
			\mc{H}_1 = \left\{ \Phi_2 u_2 \oplus Z_k^*(\Lambda_2 u_2 \oplus v_1) : u_2 \oplus v_1 \in \mc{H}(\Theta_1) \right\}.
		\end{equation*}
		Define the operator
		\[
		U_1 : \mc{H}_1 \longrightarrow \mc{H}(\Theta_1)
		\]
		by
		\begin{equation*}
			U_1\bigl(\Phi_2 u_2 \oplus Z_k^*(\Lambda_2 u_2 \oplus v_1)\bigr)
			= u_2 \oplus v_1.
		\end{equation*}
		To establish that $U_1$ is a unitary operator, it suffices to verify that it acts as an isometry. For any $u_2 \oplus v_1 \in \mc{H}(\Theta_1)$, we evaluate the norm
		\begin{align*}
			\| \Phi_2 u_2 \oplus Z_k^*(\Lambda_2 u_2 \oplus v_1) \|^2 &= \| \Phi_2 u_2 \|^2 + \| \Lambda_2 u_2 \|^2 + \| v_1 \|^2 \nonumber \\
			&= \langle (\Phi_2^* \Phi_2 + \Lambda_2^* \Lambda_2) u_2, u_2 \rangle + \| v_1 \|^2 \nonumber \\
			&= \| u_2 \|^2 + \| v_1 \|^2.
		\end{align*}
		Next, we compute explicitly the action of the compressed adjoint block $A_1^*$. We have
		\begin{align}
			(U_1 A_1^* U_1^*)(u_2 \oplus v_1)
			&= U_1 P_{\mc H_1} T_\Theta^*
			\bigl(\Phi_2 u_2 \oplus Z_k^*(\Lambda_2 u_2 \oplus v_1)\bigr) \nonumber \\
			&= U_1 P_{\mc H_1}
			\Bigl[
			e^{-it}\bigl(\Phi_2 u_2 - (\Phi_2 u_2)(0)\bigr)
			\;\oplus
			Z_k^*\bigl(
			\Lambda_2 e^{-it}u_2
			\oplus e^{-it} v_1
			\bigr)
			\Bigr] \nonumber \\
			&= U_1 P_{\mc H_1}
			\Bigl[
			\bigl(\Phi_2 u_2' \oplus Z_k^*(\Lambda_2 u_2' \oplus v_1')\bigr)
			+ \bigl(\tilde{u}_2 \oplus Z_k^* \tilde{v}_2\bigr)
			\Bigr], \label{eq:U1_adjoint}
		\end{align}
		where
		\[
		\begin{aligned}
			u_2' &= e^{-it}(u_2 - u_2(0)), & \quad v_1' &= e^{-it} v_1, \\
			\tilde{u}_2 &= e^{-it}\bigl(\Phi_2 u_2(0) - (\Phi_2 u_2)(0)\bigr), & \quad \tilde{v}_2 &= e^{-it} \Lambda_2 u_2(0).
		\end{aligned}
		\]
		The assertion that $u_2' \oplus v_1' \in \mc{H}(\Theta_1)$ follows immediately from the fact that the adjoint model operator inherently satisfies $T_{\Theta_1}^*(u_2 \oplus v_1) \in \mc{H}(\Theta_1)$.
		We now claim that the residual vector $\tilde{u}_2 \oplus Z_k^* \tilde{v}_2$ is orthogonal to $\mc{H}_1$. Taking an arbitrary element $\Phi_2 w \oplus Z_k^*(\Lambda_2 w \oplus y) \in \mc{H}_1$, we compute the inner product
		\begin{align*}
			\langle \tilde{u}_2 \oplus Z_k^* \tilde{v}_2,& \ \Phi_2 w \oplus Z_k^*(\Lambda_2 w \oplus y) \rangle\\
			&= \langle e^{-it}(\Phi_2 u_2(0) - (\Phi_2 u_2)(0)), \ \Phi_2 w \rangle + \langle e^{-it} \Lambda_2 u_2(0), \ \Lambda_2 w \rangle \nonumber \\
			&= \langle e^{-it}(\Phi_2^* \Phi_2 + \Lambda_2^* \Lambda_2) u_2(0), \ w \rangle - \langle e^{-it}(\Phi_2 u_2)(0), \ \Phi_2 w \rangle \nonumber \\
			&= \langle e^{-it} u_2(0), \ w \rangle - \langle e^{-it}(\Phi_2 u_2)(0), \ \Phi_2 w \rangle \nonumber \\
			&= 0.
		\end{align*}
		This inner product evaluates to zero since the terms $e^{-it} u_2(0)$ and $e^{-it}(\Phi_2 u_2)(0)$ reside in the orthogonal complement $L^2 \ominus H^2$, making them explicitly orthogonal to the analytic functions $w$ and $\Phi_2 w$ in $H^2$. By applying this orthogonal reduction to equation \eqref{eq:U1_adjoint}, we obtain
		\begin{align*}
			(U_1 A_1^* U_1^*)(u_2 \oplus v_1) &= U_1 \left( \Phi_2 u_2' \oplus Z_k^*(\Lambda_2 u_2' \oplus v_1') \right) \nonumber \\
			&= u_2' \oplus v_1' \nonumber \\
			&= T_{\Theta_1}^*(u_2 \oplus v_1).
		\end{align*}
		Proceeding to the intermediate blocks $i = 2, \dots, k-1$, we recall that $\mc{H}_i = \mc M_i \ominus \mc M_{i-1}$, which explicitly expands to
		\begin{multline}
			\mc{H}_i =\Big\{ \Phi_{i+1} u_{i+1} \oplus Z_k^*(\Lambda_{i+1} u_{i+1} \oplus v_i \oplus 0 \oplus \cdots \oplus 0) : u_{i+1} \in H^2(\mc{E}_{i+1}), \\ \ v_i \in \overline{\im(\Delta_i)} \Big\}
			\ominus \left\{ \Phi_i u_i \oplus Z_k^*(\Lambda_i u_i \oplus 0 \oplus \cdots \oplus 0) : u_i \in H^2(\mc{E}_i) \right\}.
		\end{multline}
		If $h_i \in \mc{H}_i$, it can be represented as $h_i = \Phi_{i+1} u_{i+1} \oplus Z_k^*(\Lambda_{i+1} u_{i+1} \oplus v_i \oplus 0 \oplus \cdots \oplus 0)$ for some elements $u_{i+1} \in H^2(\mc{E}_{i+1})$ and $v_i \in \overline{\im(\Delta_i)}$, satisfying the orthogonality condition
		\begin{equation*}
			\langle \Phi_{i+1} u_{i+1} \oplus Z_k^*(\Lambda_{i+1} u_{i+1} \oplus v_i \oplus 0 \oplus \cdots \oplus 0), \ \Phi_i u_i \oplus Z_k^*(\Lambda_i u_i \oplus 0 \oplus \cdots \oplus 0) \rangle = 0
		\end{equation*}
		for all $u_i \in H^2(\mc{E}_i)$. Applying the recursive factorization relations $\Phi_i = \Phi_{i+1} \Theta_i$ and $\Lambda_i = [\Lambda_{i+1} \Theta_i, \Delta_i]^T$, we systematically deduce
		\begin{align*}
			0&=\langle \Phi_{i+1} u_{i+1}, \Phi_i u_i \rangle + \langle \Lambda_{i+1} u_{i+1} \oplus v_i, \Lambda_i u_i \rangle \\
			&= \langle \Phi_{i+1} u_{i+1}, \Phi_{i+1} \Theta_i u_i \rangle + \langle \Lambda_{i+1} u_{i+1}, \Lambda_{i+1} \Theta_i u_i \rangle + \langle v_i, \Delta_i u_i \rangle \\
			&= \langle (\Phi_{i+1}^* \Phi_{i+1} + \Lambda_{i+1}^* \Lambda_{i+1}) u_{i+1}, \Theta_i u_i \rangle + \langle v_i, \Delta_i u_i \rangle \\
			&= \langle u_{i+1}, \Theta_i u_i \rangle + \langle v_i, \Delta_i u_i \rangle.
		\end{align*}
		for all $u_i \in H^2(\mc{E}_i)$. Since this equality holds for an arbitrary $u_i \in H^2(\mc{E}_i)$, it strictly implies that $u_{i+1} \oplus v_i \in \mc{H}(\Theta_i)$.
		Consequently, for each $i = 2, \dots, k-1$, the subspace $\mc{H}_i$ simplifies to
		\begin{equation}
			\mc{H}_i = \left\{ \Phi_{i+1} u_{i+1} \oplus Z_k^*(\Lambda_{i+1} u_{i+1} \oplus v_i \oplus 0 \oplus \cdots \oplus 0) : u_{i+1} \oplus v_i \in \mc{H}(\Theta_i) \right\}.
		\end{equation}
		We define the transformation $U_i : \mc{H}_i \to \mc{H}(\Theta_i)$ by setting
		\begin{equation}
			U_i(\Phi_{i+1} u_{i+1} \oplus Z_k^*(\Lambda_{i+1} u_{i+1} \oplus v_i \oplus 0 \oplus \cdots \oplus 0)) \coloneq u_{i+1} \oplus v_i.
		\end{equation}
		A straightforward norm computation yields
		\begin{align*}
			\| \Phi_{i+1} u_{i+1} &\oplus Z_k^*(\Lambda_{i+1} u_{i+1} \oplus v_i \oplus 0 \oplus \cdots \oplus 0) \|^2 \\
			&= \| \Phi_{i+1} u_{i+1} \|^2 + \| \Lambda_{i+1} u_{i+1} \|^2 + \| v_i \|^2 \nonumber \\
			&= \| u_{i+1} \|^2 + \| v_i \|^2.
		\end{align*}
		This confirms that $U_i$ is indeed a unitary operator from $\mc{H}_i$ onto $\mc{H}(\Theta_i)$.
		For any $u_{i+1} \oplus v_i \in \mc{H}(\Theta_i)$, we calculate the adjoint action as follows
		\begin{align}
			(U_i A_i^* U_i^*)&(u_{i+1} \oplus v_i) \nonumber \\
			&= U_i P_{\mc{H}_i} T_\Theta^* (\Phi_{i+1} u_{i+1} \oplus Z_k^*(\Lambda_{i+1} u_{i+1} \oplus v_i \oplus 0 \oplus \cdots \oplus 0)) \nonumber \\
			&= U_i P_{\mc{H}_i} \Bigl[ e^{-it}(\Phi_{i+1} u_{i+1} - (\Phi_{i+1} u_{i+1})(0)) \oplus e^{-it} Z_k^*(\Lambda_{i+1} u_{i+1} \oplus v_i  \nonumber \\
			&\qquad\qquad\oplus 0 \oplus \cdots \oplus 0) \Bigr] \nonumber \\
			&= U_i P_{\mc{H}_i} \Bigl[ \{ \Phi_{i+1} u_{i+1}' \oplus Z_k^*(\Lambda_{i+1} u_{i+1}' \oplus v_i') \}  + \{\tilde{u}_{i+1} \oplus Z_k^* \tilde{v}_{i+1}\} \Bigr], \label{eq:Ui_adjoint}
		\end{align}
		where
		\begin{align*}
			u_{i+1}' &= e^{-it}\bigl(u_{i+1} - u_{i+1}(0)\bigr), \quad& v_i' &= e^{-it}v_i, \\
			\tilde{u}_{i+1} &= e^{-it}\bigl(\Phi_{i+1} u_{i+1}(0) - (\Phi_{i+1} u_{i+1})(0)\bigr),  & \tilde{v}_{i+1} &= e^{-it} \Lambda_{i+1} u_{i+1}(0).
		\end{align*}
		We show that the vector \( \tilde{u}_{i+1} \oplus Z_k^* \tilde{v}_{i+1} \) is orthogonal to \( \mc{H}_i \). Let  
		\( \Phi_{i+1} w \oplus Z_k^*(\Lambda_{i+1} w \oplus y) \in \mc{H}_i\)
		be arbitrary. We compute the inner product
		\begin{align*}
			\langle \tilde{u}_{i+1}& \oplus Z_k^* \tilde{v}_{i+1}, \ \Phi_{i+1} w \oplus Z_k^*(\Lambda_{i+1} w \oplus y) \rangle\\
			&= \langle e^{-it}(\Phi_{i+1} u_{i+1}(0) - (\Phi_{i+1} u_{i+1})(0)), \ \Phi_{i+1} w \rangle  + \langle e^{-it} \Lambda_{i+1} u_{i+1}(0), \ \Lambda_{i+1} w \rangle \nonumber \\
			&= \langle e^{-it} u_{i+1}(0), \ w \rangle - \langle e^{-it}(\Phi_{i+1} u_{i+1})(0), \ \Phi_{i+1} w \rangle \nonumber \\
			&= 0.
		\end{align*}
		Using this orthogonality in \eqref{eq:Ui_adjoint}, we obtain
		\begin{align*}
			(U_i A_i^* U_i^*)(u_{i+1} \oplus v_i) &= U_i \left( \Phi_{i+1} u_{i+1}' \oplus Z_k^*(\Lambda_{i+1} u_{i+1}' \oplus v_i') \right) \nonumber \\
			&= u_{i+1}' \oplus v_i' \nonumber \\
			&= T_{\Theta_i}^*(u_{i+1} \oplus v_i).
		\end{align*}
		This verifies the unitary equivalence $U_i A_i^* U_i^* = T_{\Theta_i}^*$ for the intermediate blocks. Finally, we address $i=k$. The subspace $\mc{H}_k = \mc{H}(\Theta) \ominus \mc M_{k-1}$ is given by
		\begin{align*}
			\mc{H}_k &= \left\{ H^2(\mc{E}_*) \oplus \overline{\Delta_{\Theta} L^2(\mc{E})} \right\} \ominus \Big\{ \Theta_k u_k \oplus Z_k^*(\Delta_k u_k \oplus v_{k-1} \oplus \cdots \oplus v_1)\\
			&\hspace{7cm}: v_j \in \overline{\im(\Delta_j)} \Big\} \nonumber \\
			&= \left\{ H^2(\mc{E}_*) \oplus Z_k^*(\overline{\Delta_k L^2(\mc{E}_k)} \oplus \{0\} \oplus \cdots \oplus \{0\}) \right\} \ominus \Big\{ \Theta_k u_k \oplus Z_k^*(\Delta_k u_k \oplus 0 \oplus  \\& \hspace{8cm}\cdots \oplus 0): u_k \in H^2(\mc{E}_k) \Big\}.
		\end{align*}
		Assuming $h_k \in \mc{H}_k$, we can express it as $h_k = u \oplus Z_k^*(v_k \oplus 0 \oplus \cdots \oplus 0)$ subject to the orthogonality constraint
		\begin{equation*}
			\langle u \oplus Z_k^*(v_k \oplus 0 \oplus \cdots \oplus 0), \ \Theta_k u_k \oplus Z_k^*(\Delta_k u_k \oplus 0 \oplus \cdots \oplus 0) \rangle = 0 ~ \text{for all } u_k \in H^2(\mc{E}_k).
		\end{equation*}
		Evaluating this constraint yields
		\begin{equation*}
			\langle u, \Theta_k u_k \rangle + \langle v_k, \Delta_k u_k \rangle = 0 \quad \text{for all } u_k \in H^2(\mc{E}_k),
		\end{equation*}
		which ensures that $u \oplus v_k \in \mc{H}(\Theta_k)$. Hence, the subspace is exactly characterized by
		\begin{equation}
			\mc{H}_k = \left\{ u \oplus Z_k^*(v_k \oplus 0 \oplus \cdots \oplus 0) : u \oplus v_k \in \mc{H}(\Theta_k) \right\}.
		\end{equation}
		We define the corresponding mapping $U_k : \mc{H}_k \to \mc{H}(\Theta_k)$ according to
		\begin{equation}
			U_k(u \oplus Z_k^*(v_k \oplus 0 \oplus \cdots \oplus 0)) \coloneq u \oplus v_k.
		\end{equation}
		Because the norm identity $\| u \oplus Z_k^*(v_k \oplus 0 \oplus \cdots \oplus 0) \|^2 = \| u \|^2 + \| v_k \|^2$ holds trivially, $U_k$ is a unitary operator. Applying the adjoint block $A_k^*$ to an element $u \oplus v_k \in \mc{H}(\Theta_k)$, we obtain
		\begin{align*}
			(U_k A_k^* U_k^*)(u \oplus v_k) &= U_k T_\Theta^* (u \oplus Z_k^*(v_k \oplus 0 \oplus \cdots \oplus 0)) \nonumber \\
			&= U_k \left( e^{-it}(u - u(0)) \oplus Z_k^*(e^{-it}v_k \oplus 0 \oplus \cdots \oplus 0) \right) \nonumber \\
			&= e^{-it}(u - u(0)) \oplus e^{-it}v_k \nonumber \\
			&= T_{\Theta_k}^*(u \oplus v_k).
		\end{align*}
		We conclude that for every  $i=1,\dots,k$, the diagonal block operator $A_i$ is unitarily equivalent to the functional model operator $T_{\Theta_i}$, completing the proof.
	\end{proof}
	
	
	\begin{corollary}
		Under the assumptions of Theorem $\ref{unitary_equivalent}$, for each $i = 1, \dots, k$, the characteristic function of the operator $A_i$ coincides with the purely contractive part of $\Theta_i$.
	\end{corollary}
	
	\begin{corollary}
		Assume the hypotheses of Theorem $\ref{unitary_equivalent}$. Then, for each $i = 1, \dots, k-2$, the equality $\mc{M}_i = \mc{M}_{i+1}$ holds if and only if $\Theta_{i+1}$ is a unitary constant; that is, there exists a unitary operator $U : \mc{E}_{i+1} \to \mc{E}_{i+2}$ such that
		\[
		\Theta_{i+1}(z) = U \quad \text{for all } z \in \mathbb{D}.
		\]
	\end{corollary}
	\begin{proof}
		Fix $i \in \{1, \dots, k-2\}$. By Theorem \ref{unitary_equivalent}, the orthogonal difference space
		\[
		\mc{H}_{i+1} \coloneqq \mc{M}_{i+1} \ominus \mc{M}_i
		\]
		is unitarily equivalent to the model space $\mc{H}(\Theta_{i+1})$. It is a fundamental result in Sz.-Nagy--Foia\c{s} theory that a model space $\mc{H}(\Theta_{i+1})$ is trivial if and only if its defining contractive analytic function $\Theta_{i+1}$ is a unitary constant. The desired conclusion follows by combining these equivalences.
	\end{proof}

	\begin{proposition} \label{prop:aggregated_partition}
		Let \( \{\mathcal{E}, \mathcal{E}_*, \Theta(z)\} \) be a purely contractive analytic function admitting a $k$-regular factorization
		\[
		\Theta(z) = \Theta_k(z) \cdots \Theta_1(z),
		\qquad z \in \mathbb{D},
		\]
		where \( \{\mc{E}_i, \mc{E}_{i+1}, \Theta_i(z)\} \) are contractive analytic functions for \( i = 1, \dots, k \), with \( \mc{E}_1 = \mc{E} \) and \( \mc{E}_{k+1} = \mc{E}_* \).  Let
		\[
		\mc{M}_1 \subseteq \mc{M}_2 \subseteq \cdots \subseteq \mc{M}_{k-1}
		\]
		be the corresponding chain of invariant subspaces for the model contraction $T_\Theta$ acting on $\mc{H}(\Theta)$. Let the index set $\{1, \dots, k\}$ be partitioned into $r$ disjoint  subsets $J_1, \dots, J_r$ defined by
		\[
		J_1 = \{ j_1, \ldots, 1 \},\quad J_i = \{j_i, \dots, j_{i-1}+1\} \ (i=2,\dots,r),
		\]
		where $1 \le j_1 < j_2 < \cdots < j_r = k$.
		Consider the associated aggregated $r$-regular factorization
		\[
		\Theta = \Theta_{J_r} \cdots \Theta_{J_1},
		\quad \text{where} \quad
		\Theta_{J_i} \coloneqq \Theta_{j_i} \cdots \Theta_{j_{i-1}+1}.
		\]
		If
		\[
		\mc{M}_{J_1} \subseteq \cdots \subseteq \mc{M}_{J_{r-1}}
		\]
		denotes the corresponding chain of invariant subspaces for this aggregated factorization, then
		\[
		\mc{M}_{J_i} = \mc{M}_{j_i} \quad \text{for all } i = 1, \dots, r-1.
		\]
	\end{proposition}

	\begin{proof}
		Fix an index $i \in \{1, \dots, r-1\}$; by definition of the partition, we clearly have $j_i < k$. By Theorem \ref{k_r_invariant_theorem}, the invariant subspace $\mc{M}_{j_i}$ corresponding to the $k$-regular factorization $\Theta = \Theta_k \cdots \Theta_1$ admits the explicit representation
		\begin{align}\label{eq:M_ji_pdf}
			\mc{M}_{j_i} &= \Big\{ \Theta_k \cdots \Theta_{j_i+1} u \oplus Z_k^* \big( \Delta_k \Theta_{k-1} \cdots \Theta_{j_i+1} u \oplus \cdots \oplus \Delta_{j_i+1} u \oplus v_{j_i}\\
			&\qquad \oplus \cdots \oplus v_1 \big) :   u \in H^2(\mc{E}_{j_i+1}), \; v_l \in \overline{\Delta_l L^2(\mc{E}_l)} \text{ for } l = 1, \dots, j_i \Big\} \ominus \mc{G},\nonumber
		\end{align}
		where $\mc{G} = \{ \Theta u \oplus \Delta_{\Theta} u : u \in H^2(\mc{E}) \}$.
		Applying the same theorem to the aggregated $r$-regular factorization $\Theta = \Theta_{J_r} \cdots \Theta_{J_1}$, the corresponding invariant subspace $\mc{M}_{J_i}$ is given by
		\begin{align}\label{eq:M_Ji_pdf}
			&\mc{M}_{J_i} = \Big\{ \Theta_{J_r} \cdots \Theta_{J_{i+1}} u \oplus (Z_r^{J_r, \dots, J_1})^* \big( \Delta_{J_r} \Theta_{J_{r-1}} \cdots \Theta_{J_{i+1}} u \oplus \cdots \oplus \Delta_{J_{i+1}} u  \oplus \\
			&~w_i  \oplus \cdots  \oplus w_1 \big) : u \in H^2(\mc{E}_{j_i+1}), \; w_m \in \overline{\Delta_{J_m} L^2(\mc{E}_{j_{m-1}+1})},  m = 1, \dots, i  \Big\} \ominus \mc{G},\nonumber
		\end{align}
		where $\Delta_{J_i}=(I-\Theta_{J_i}^*\Theta_{J_i})^{1/2}.$
		Let us recall the relation established in equation \eqref{deco_relation_zK} during the proof of Proposition \ref{partition}
		\[
		Z_k = \left( \bigoplus_{m=r}^1 Z_{|J_m|}^{\{j_m\}, \dots, \{j_{m-1}+1\}} \right) Z_r^{J_r, \dots, J_1}.
		\]
		Since all operators in this decomposition are unitary maps, we get
		\[
		(Z_r^{J_r, \dots, J_1})^* = Z_k^* \left( \bigoplus_{m=r}^1 \big( Z_{|J_m|}^{\{j_m\}, \dots, \{j_{m-1}+1\}} \big) \right).
		\]
		We now evaluate the action of $(Z_r^{J_r, \dots, J_1})^*$ on the components of the defect space. Observe that
		\begin{align*}
			&\big( Z_{|J_m|}^{\{j_m\}, \dots, \{j_{m-1}+1\}} \big) \big( \Delta_{J_m} \Theta_{J_{m-1}} \cdots \Theta_{J_{i+1}} u \big) \\
			&\quad = \Delta_{j_m} \Theta_{j_m-1} \cdots \Theta_{j_i+1} u \oplus \cdots \oplus \Delta_{j_{m-1}+1} \Theta_{j_{m-1}} \cdots \Theta_{j_i+1} u.
		\end{align*}
		For $m = i, \dots, 1$, as the arbitrary vector $w_m\in \overline{\im \Delta_{J_m} }$, and $$Z_{|J_m|}^{\{j_m\}, \dots, \{j_{m-1}+1\}}( \overline{\im\Delta_{J_m}})=\overline{\im\Delta_{j_m}}\oplus\cdots\oplus\overline{\im\Delta_{j_{m-1}+1}}$$
		Using these relations, we obtain
		\begin{align*}
			&(Z_r^{J_r, \dots, J_1})^* \big( \Delta_{J_r} \Theta_{J_{r-1}} \cdots \Theta_{J_{i+1}} u \oplus \cdots \oplus \Delta_{J_{i+1}} u \oplus w_i \oplus \cdots \oplus w_1 \big) \\
			&\quad = Z_k^* \big( \Delta_k \Theta_{k-1} \cdots \Theta_{j_i+1} u \oplus \cdots \oplus \Delta_{j_i+1} u \oplus v_{j_i} \oplus \cdots \oplus v_1 \big).
		\end{align*}
		By utilizing this equivalence, alongside the relation $$\Theta_{J_r} \cdots \Theta_{J_{i+1}} = \Theta_k \cdots \Theta_{j_i+1},$$ we conclude that $\mc{M}_{J_i} = \mc{M}_{j_i}$, which completes the proof.
	\end{proof}


	\section{Examples}

	\begin{example}
		Let $k \ge 2$ be an integer, and consider the purely contractive analytic function
		$\Theta : \mathbb{D} \to \mathbb{C}$ defined by $\Theta(z) = z^k$. We examine the factorization
		\begin{equation} \label{eq:factorization}
			\Theta(z) = \Theta_k(z)\cdots \Theta_1(z),
			\quad \text{where } \Theta_i(z) = z \text{ for each } i = 1, \dots, k.
		\end{equation}
		Each factor $\Theta_i$ is an inner function. Hence, the corresponding defect operators vanish almost everywhere on the unit circle $\mathbb{T}$; that is,
		\[
		\Delta_{\Theta}(t) = 0 \quad \text{and} \quad \Delta_i(t)= 0 \quad \text{for a.e. } t \in \mathbb{T}, \text{ for all } i \in \{1, \dots, k\}.
		\]
		Consequently, the associated isometric map
		\[
		Z_k : \overline{\Delta_{\Theta} L^2(\mathbb{T})} \to
		\overline{\Delta_k L^2(\mathbb{T})} \oplus \cdots \oplus \overline{\Delta_1 L^2(\mathbb{T})}
		\]
		is trivially a unitary operator between zero-dimensional Hilbert spaces. Therefore, the factorization \eqref{eq:factorization} is a $k$-regular factorization.
		
		The functional model operator associated with this $k$-regular factorization acts on the Hilbert space
		\[
		\boldsymbol{\mc H} = \left\{ H^2(\mathbb{T}) \oplus \overline{\Delta_k L^2(\mathbb{T})} \oplus \cdots \oplus \overline{\Delta_1 L^2(\mathbb{T})} \right\} \ominus \boldsymbol{\mc G}
		\]
		via the adjoint action
		\[
		\boldsymbol{T}^* (u \oplus v_k \oplus \cdots \oplus v_1) = e^{-it}[u(e^{it}) - u(0)] \oplus e^{-it}v_k(t) \oplus \cdots \oplus e^{-it}v_1(t),
		\]
		where the subspace $\boldsymbol {\mc G}$ is defined as
		\[
		\boldsymbol{\mc G} = \{ \Theta_k \cdots \Theta_1 u \oplus \Delta_k \Theta_{k-1} \cdots \Theta_1 u \oplus \cdots \oplus \Delta_1 u : u \in H^2(\mathbb{T}) \}.
		\]
		Given that $\Theta_i(z) = z$ and $\Delta_i(t) = 0$ for each $i$, it follows that
		\[
		\boldsymbol{\mc G} = z^k H^2(\mathbb{T}).
		\]
		Consequently, the corresponding model space is
		\[
		\boldsymbol{\mc H} = H^2(\mathbb{T}) \ominus z^k H^2(\mathbb{T}).
		\]
		In this case, the operator $\boldsymbol{T}$ can be identified as the compression of the forward shift to $\boldsymbol{\mc H}$. Observe that
		\[
		\boldsymbol{\mc H} = \operatorname{span} \{ 1, z, z^2, \dots, z^{k-1} \}.
		\]
		The invariant subspaces corresponding to this $k$-regular factorization are given as follows
		\begin{align*}
			\boldsymbol{\mc M}_1 &= \left\{ \Theta_k \cdots \Theta_2 u_2 \oplus \Delta_k \Theta_{k-1} \cdots \Theta_2 u_2 \oplus \cdots \oplus \Delta_2 u_2 \oplus v_1 \right.
			\\
			&\qquad \left. \quad: u_2 \in H^2(\mathbb{T}), v_1 \in \overline{\Delta_1 L^2(\mathbb{T})} \right\} \ominus \boldsymbol{\mc G} \\
			&= z^{k-1} H^2(\mathbb{T}) \ominus z^k H^2(\mathbb{T}) = \text{span} \{ z^{k-1} \}, \\
			\boldsymbol{\mc M}_2 &= z^{k-2} H^2(\mathbb{T}) \ominus z^k H^2(\mathbb{T}) = \text{span} \{ z^{k-2}, z^{k-1} \}, \\
			&\vdots \\
			\boldsymbol{\mc M}_{k-1} &= z H^2(\mathbb{T}) \ominus z^k H^2(\mathbb{T}) = \text{span} \{ z, \dots, z^{k-1} \}.
		\end{align*}
		Consequently, from the above expressions for the invariant subspaces, we obtain the following inclusions
		\[
		\boldsymbol{\mc M}_1 \sub \boldsymbol{\mc M}_2 \sub \cdots \sub \boldsymbol{\mc M}_{k-1}.
		\]
	\end{example}
	
	
	\begin{example}
		Let $c \in \mathbb{C}$ be a constant such that $|c| < 1$. Define the purely contractive analytic function
		$\Theta : \mathbb{D} \to {B}(\mathbb{C}^2)$ by
		\begin{equation}
			\Theta(z) =
			\begin{pmatrix}
				z^3 & 0 \\
				0 & c
			\end{pmatrix},
			\qquad z \in \mathbb{D}.
		\end{equation}
		Consider the factorization
		\begin{equation}\label{eq:4_regular_fact}
			\Theta(z) = \Theta_4(z)\Theta_3(z)\Theta_2(z)\Theta_1(z),
		\end{equation}
		where the factors are given by
		\begin{equation}
			\Theta_4(z) =
			\begin{pmatrix}
				1 & 0 \\
				0 & c
			\end{pmatrix},
			\qquad
			\Theta_3(z) = \Theta_2(z) = \Theta_1(z) =
			\begin{pmatrix}
				z & 0 \\
				0 & 1
			\end{pmatrix}.
		\end{equation}
		The corresponding defect operators are given by
		\begin{align}\label{eq:defect_ops}
			\Delta_1(t) &= \Delta_2(t) = \Delta_3(t) = 0, \\
			~
			\Delta_{\Theta}(t) &= \Delta_4(t) = \operatorname{diag}\!\left(0, \sqrt{1 - |c|^2}\right),
			~ \text{for a.e. } t \in \mathbb{T}.
		\end{align}
		The isometric operator $Z_4$ associated with the factorization \eqref{eq:4_regular_fact} is defined by
		\begin{align*}
			Z_4 : \overline{\Delta_{\Theta} L^2(\mathbb{T}, \mathbb{C}^2)}
			\longrightarrow
			\overline{\Delta_4 L^2(\mathbb{T}, \mathbb{C}^2)}
			\oplus
			\overline{\Delta_3 L^2(\mathbb{T}, \mathbb{C}^2)}
			\oplus
			\overline{\Delta_2 L^2(\mathbb{T}, \mathbb{C}^2)}
			\oplus
			\overline{\Delta_1 L^2(\mathbb{T}, \mathbb{C}^2)}\\
			Z_4(\Delta_{\Theta} u)
			=
			\Delta_4 \Theta_3 \Theta_2 \Theta_1 u
			\oplus
			\Delta_3 \Theta_2 \Theta_1 u
			\oplus
			\Delta_2 \Theta_1 u
			\oplus
			\Delta_1 u,\quad u \in L^2(\mathbb{T}, \mathbb{C}^2) .
		\end{align*}
		Using the defect operator relations in \eqref{eq:defect_ops}, the codomain simplifies to
		\[
		\overline{\Delta_4 L^2(\mathbb{T}, \mathbb{C}^2)} \oplus \{0\} \oplus \{0\} \oplus \{0\}.
		\]
		Moreover, for a.e.\ $t \in \mathbb{T}$, we have \begin{align*}
			\Delta_4(t)\Theta_3(e^{it})\Theta_2(e^{it})\Theta_1(e^{it})
			&=
			\operatorname{diag}\!\left(0,\sqrt{1-|c|^2}\right)
			\operatorname{diag}\!\left(e^{3it},1\right) \\
			&=
			\operatorname{diag}\!\left(0,\sqrt{1-|c|^2}\right).
		\end{align*}
		Consequently, the mapping $Z_4$ reduces to
		\[
		Z_4(\Delta_{\Theta} u) = \Delta_4 u \oplus 0 \oplus 0 \oplus 0.
		\]
		Since $\Delta_{\Theta} = \Delta_4$, it follows that $Z_4$ is surjective. Therefore, the factorization \eqref{eq:4_regular_fact} is a $4$-regular factorization.
		
		The functional model operator $\boldsymbol{T}$ corresponding to the $4$-regular factorization \eqref{eq:4_regular_fact} acts on the Hilbert space
		\[
		\boldsymbol{\mc H}
		\coloneqq
		\Bigl[
		H^2(\mc E_{5})
		\oplus
		\overline{\Delta_4 L^2(\mathbb{T}, \mc E_4)}
		\oplus \cdots \oplus
		\overline{\Delta_1 L^2(\mathbb{T}, \mc E_1)}
		\Bigr]
		\ominus \boldsymbol{\mc G},
		\]
		via the adjoint action
		\[
		\boldsymbol{T}^*(u \oplus v_4 \oplus \cdots \oplus v_1)
		=
		\bigl(e^{-it}(u - u(0))\bigr)
		\oplus
		e^{-it} v_4
		\oplus \cdots \oplus
		e^{-it} v_1,
		\]
		where $u \in H^2(\mathbb{T}, \mc E_5)$ and $v_i \in \overline{\Delta_i L^2(\mathbb{T}, \mc E_i)}$ for $i = 1, \dots, 4$. The spaces are given by $\mc E_i = \mathbb{C}^2$ for $i = 1, \dots, 5$.
		The subspace $\boldsymbol{\mc G}$ is defined by
		\[
		\boldsymbol{\mc G}
		=
		\left\{
		\Theta_4 \cdots \Theta_1 u
		\oplus
		\Delta_4 \Theta_3 \cdots \Theta_1 u
		\oplus \cdots \oplus
		\Delta_1 u
		:
		u \in H^2(\mc E_1)
		\right\}.
		\]
		Moreover, under the $4$-regular factorization \eqref{eq:4_regular_fact}, the associated invariant subspaces $\boldsymbol{\mc{M}}_i$ for $i=1, 2, 3$ are structurally given by
		\begin{align*}
			\boldsymbol{\mc{M}}_i
			=
			\Big\{
			&\Theta_4 \cdots \Theta_{i+1} u_{i+1}
			\oplus
			\Delta_4 \Theta_{3} \cdots \Theta_{i+1} u_{i+1}
			\oplus \cdots \oplus
			\Delta_{i+1} u_{i+1} \\
			&\oplus v_i \oplus \cdots \oplus v_1
			:
			u_{i+1} \in H^2(\mc{E}_{i+1}),
			\;
			v_j \in \overline{\Delta_j L^2(\mc{E}_j)},
			\ j = 1, \dots, i
			\Big\}
			\ominus \boldsymbol{\mc G}.
		\end{align*}
		Utilizing the defect relations established in \eqref{eq:defect_ops}, the subspace $\boldsymbol{\mc{G}}$ explicitly evaluates to
		\begin{equation*}
			\boldsymbol{\mc{G}} = \left\{
			\begin{pmatrix} z^3 & 0 \\ 0 & c \end{pmatrix} u
			\oplus
			\begin{pmatrix} 0 \\ \sqrt{1-|c|^2} u_2 \end{pmatrix}
			\oplus 0 \oplus 0 \oplus 0 : u = \begin{pmatrix} u_1 \\ u_2 \end{pmatrix} \in H^2(\mb C^2)
			\right\}.
		\end{equation*}
		The functional model space is therefore reduced to
		\begin{equation*}
			\boldsymbol{\mc{H}} = \left[ H^2(\mb C^2) \oplus \ov{\Delta_{\Theta} L^2(\mb C^2)} \oplus \{0\} \oplus \{0\} \oplus \{0\} \right] \ominus \boldsymbol{\mc{G}}.
		\end{equation*}
		The adjoint of the model operator $\boldsymbol{T}^*$ effectively acts as
		\begin{equation*}
			\boldsymbol{T}^*(u \oplus v_4 \oplus 0 \oplus 0 \oplus 0) = e^{-it}\bigl(u(e^{it}) - u(0)\bigr) \oplus e^{-it}v_4(t) \oplus 0 \oplus 0 \oplus 0.
		\end{equation*}
		The invariant subspaces $\boldsymbol{\mc{M}}_1, \boldsymbol{\mc{M}}_2, \boldsymbol{\mc{M}}_3$ are determined as follows.
		Using the relations $\Theta_4 \Theta_3 \Theta_2 = \diag(z^2, c)$ and $\Delta_4\Theta_{3}\Theta_2=\Delta_{\Theta}$, we obtain
		\begin{align*}
			\boldsymbol{\mc{M}}_1 = \left\{
			\begin{pmatrix} z^2 & 0 \\ 0 & c \end{pmatrix} w \oplus \Delta_{\Theta} w \oplus 0 \oplus 0 \oplus 0 : w \in H^2(\mb C^2)
			\right\} \ominus \boldsymbol{\mc{G}}.
		\end{align*}
		Using the partial product $\Theta_4 \Theta_3 = \diag(z, c)$ and $\Delta_4\Theta_{3}=\Delta_{\Theta}$, we get
		\begin{equation*}
			\boldsymbol{\mc{M}}_2 = \left\{
			\begin{pmatrix} z & 0 \\ 0 & c \end{pmatrix} w \oplus \Delta_{\Theta} w \oplus 0 \oplus 0 \oplus 0 : w \in H^2(\mb C^2)
			\right\} \ominus \boldsymbol{\mc{G}}.
		\end{equation*}
		Using the relations $\Theta_4 = \diag(1, c)$ and $\Delta_4=\Delta_{\Theta}$, we have
		\begin{equation*}
			\boldsymbol{\mc{M}}_3 = \left\{
			\begin{pmatrix} 1 & 0 \\ 0 & c \end{pmatrix} w \oplus \Delta_{\Theta} w \oplus 0 \oplus 0 \oplus 0 : w \in H^2(\mb C^2)
			\right\} \ominus \boldsymbol{\mc{G}}.
		\end{equation*}
		The chain of inclusions $\boldsymbol{\mc{M}}_1 \sub \boldsymbol{\mc{M}}_2 \sub \boldsymbol{\mc{M}}_3$ follows directly from the nested structure of the scalar Hardy spaces $z^2 H^2(\mb C) \sub z H^2(\mb C) \sub H^2(\mb C)$.
	\end{example}
	
	
	In Theorem 2 of \cite{NF67}, Sz.-Nagy and Foia\c{s} explicitly compute the factorization of the characteristic function for a contraction in the block upper triangular matrix form. In the following example, we consider a specific contraction of this type to determine whether the associated factorization of its characteristic function is $3$-regular.
	
	\begin{example}
		Let $S: H^2(\mathbb{D}) \longrightarrow H^2(\mathbb{D})$ be the shift operator on the scalar-valued Hardy space. Consider
		\[
		A = \bigoplus_{0}^{\infty} S, \quad B = \bigoplus_{0}^{\infty} S^*
		\]
		and let $L: \mathcal{D}_B \longrightarrow \mathcal{D}_{A^*}$ be a proper contraction (i.e., $\|Lx\| < \|x\|$ for all non-zero $x \in \mathcal{D}_B$). Then the operator
		\[
		T = \begin{bmatrix} A & D_{A^*} L D_B \\ 0 & B \end{bmatrix} : \begin{matrix} \mathcal{H}_1 \\ \oplus \\ \mathcal{H}_2 \end{matrix} \longrightarrow \begin{matrix} \mathcal{H}_1 \\ \oplus \\ \mathcal{H}_2 \end{matrix}
		\]
		is a c.n.u. contraction (see \cite[Condition 1.10]{NF67}). Consider the factorization of the characteristic function given in \cite{NF67}:
		\begin{align}\label{clasical_fact}
			\Theta_T(z) = \tau_*^{-1} \begin{bmatrix} \Theta_B(z) & 0 \\ 0 & I_{\mathcal{D}_{L^*}} \end{bmatrix} \begin{bmatrix} L^* & D_L \\ D_{L^*} & -L \end{bmatrix} \begin{bmatrix} \Theta_A(z) & 0 \\ 0 & I_{\mathcal{D}_L} \end{bmatrix} \tau
		\end{align}
		where $\tau \in \mathcal{B}(\mathcal{D}_T, \mathcal{D}_A \oplus \mathcal{D}_L)$ and $\tau_* \in \mathcal{B}(\mathcal{D}_{T^*}, \mathcal{D}_{B^*} \oplus \mathcal{D}_{L^*})$ are unitary operators.
		Set
		\[
		\Theta_1(z) = \begin{bmatrix} \Theta_A(z) & 0 \\ 0 & I_{\mathcal{D}_L} \end{bmatrix} \tau, ~
		\Theta_2(z) = \begin{bmatrix} L^* & D_L \\ D_{L^*} & -L \end{bmatrix}, \quad
		\Theta_3(z) = \tau_*^{-1} \begin{bmatrix} \Theta_B(z) & 0 \\ 0 & I_{\mathcal{D}_{L^*}} \end{bmatrix}.
		\]
		By Propositions \ref{local_prop} and \ref{equivalent_prop}, the factorization $\Theta_T(z) = \Theta_3(z)\Theta_2(z)\Theta_1(z)$ is $3$-regular if and only if
		\begin{equation}\label{condition_1}
			\Delta_{\Theta_2}(t) (\mathcal{D}_{A^*} \oplus \mathcal{D}_L) \cap \Delta_{\Theta_1^*}(t) (\mathcal{D}_{A^*} \oplus \mathcal{D}_L) = \{0\} \quad \text{ for a.e. t }\in [0,2 \pi]
		\end{equation}
		and
		\begin{equation}\label{condition_2}
			\Delta_{\Theta_3}(t) (\mathcal{D}_B \oplus \mathcal{D}_{L^*}) \cap \Delta_{\Theta_2^* \Theta_1^*}(t) (\mathcal{D}_B \oplus \mathcal{D}_{L^*}) = \{0\} \quad  \text{ for a.e. t }\in [0,2 \pi]
		\end{equation}
		Because $\Theta_2(e^{it})$ is a unitary operator, we have $\Delta_{\Theta_2}(t) = 0$ a.e. Consequently, condition \eqref{condition_1} trivially holds.
		Now, observe that
		\[
		\Delta_{\Theta_3}^2(t) = I_{\mathcal{D}_B \oplus \mathcal{D}_{L^*}} -  \begin{bmatrix} \Theta_B^*(e^{it}) & 0 \\ 0 & I_{\mathcal{D}_{L^*}} \end{bmatrix} \tau_* \tau_*^{-1} \begin{bmatrix} \Theta_B(e^{it}) & 0 \\ 0 & I_{\mathcal{D}_{L^*}} \end{bmatrix};
		\]
		hence, we obtain
		\begin{equation}\label{delta3}
			\Delta_{\Theta_3}(t) = \begin{bmatrix} \Delta_{\Theta_B}(t) & 0 \\ 0 & 0 \end{bmatrix}.
		\end{equation}
		Similarly, since $\Delta_{\Theta_1^* \Theta_2^*}^2(t) = I - \Theta_2(e^{it}) \Theta_1 (e^{it})\Theta_1^*(e^{it}) \Theta_2^*(e^{it})$, we find
		\[
		\Delta_{\Theta_1^* \Theta_2^*}^2(t) = \begin{bmatrix} L^* & D_L \\ D_{L^*} & -L \end{bmatrix} \begin{bmatrix} \Delta_{\Theta_A^*}^2(t) & 0 \\ 0 & 0 \end{bmatrix} \begin{bmatrix} L & D_{L^*} \\ D_L & -L^* \end{bmatrix},
		\]
		which implies
		\begin{equation}\label{delta21}
			\Delta_{\Theta_1^* \Theta_2^*}(t) = \begin{bmatrix} L^* & D_L \\ D_{L^*} & -L \end{bmatrix} \begin{bmatrix} \Delta_{\Theta_A^*}(t) & 0 \\ 0 & 0 \end{bmatrix} \begin{bmatrix} L & D_{L^*} \\ D_L & -L^* \end{bmatrix}.
		\end{equation}
		Using relations \eqref{delta3} and \eqref{delta21}, condition \eqref{condition_2} reduces to
		\[
		\left(\Delta_{\Theta_B}(t)\mathcal{D}_B \oplus \{0\}\right) \cap \begin{bmatrix} L^* & D_L \\ D_{L^*} & -L \end{bmatrix} \begin{bmatrix} \Delta_{\Theta_A^*}(t) & 0 \\ 0 & 0 \end{bmatrix} \begin{bmatrix} \mathcal{D}_{A^*} \\ \mathcal{D}_{L^*} \end{bmatrix} = \{0\}.
		\]
		
		Since $S$ is a shift operator, the characteristic functions $\Theta_{S}$ and $\Theta_{S^*}$ are identically zero. Thus, $\Theta_B \equiv 0$ and $\Theta_A \equiv 0$, which yields $\Delta_{\Theta_B}(t) = I_{\mc D_B}$ and $\Delta_{\Theta_A^*}(t) = I_{\mc D_{A^*}}$. We then have
		\begin{align*}
			\left(\Delta_{\Theta_B}(t)\mathcal{D}_B \oplus \{0\}\right) &\cap \begin{bmatrix} L^* & D_L \\ D_{L^*} & -L \end{bmatrix} \begin{bmatrix} \Delta_{\Theta_A^*}(t) & 0 \\ 0 & 0 \end{bmatrix} \begin{bmatrix} \mathcal{D}_{A^*} \\ \mathcal{D}_{L^*} \end{bmatrix}\\
			&=\left(\mathcal{D}_B \oplus \{0\}\right) \cap \begin{bmatrix} L^* & D_L \\ D_{L^*} & -L \end{bmatrix} \begin{bmatrix} \mathcal{D}_{A^*} \\ \{0\} \end{bmatrix} \\
			&= \left(\mathcal{D}_B \oplus \{0\}\right) \cap \{ L^*x \oplus D_{L^*}x : x \in \mathcal{D}_{A^*} \}.
		\end{align*}
		Since $D_{L^*}$ is injective (as $L$ is a proper contraction), we obtain
		\[
		\left(\mathcal{D}_B \oplus \{0\}\right) \cap \{ L^*x \oplus D_{L^*}x : x \in \mathcal{D}_{A^*} \} = \{0\}.
		\]
		Hence, the factorization
		\[
		\Theta_T(z) = \tau_*^{-1} \begin{bmatrix} \Theta_B(z) & 0 \\ 0 & I_{\mathcal{D}_{L^*}} \end{bmatrix} \begin{bmatrix} L^* & D_L \\ D_{L^*} & -L \end{bmatrix} \begin{bmatrix} \Theta_A(z) & 0 \\ 0 & I_{\mathcal{D}_L} \end{bmatrix} \tau
		\]
		is a $3$-regular factorization.
	\end{example}
	
	For contractions in the block upper triangular matrix form, it was shown in \cite{HM2026} that the classical factorization \eqref{clasical_fact} of the characteristic function, introduced by Sz.-Nagy and Foia\c{s} in \cite{NF67}, is not regular in general; counterexamples were also provided. Moreover, it is proved that for a large class of c.n.u.\ contractions, including pure contractions, the factorization is $3$-regular.

	\section{$k$-Regular Factorizations Associated with Commuting $k$-Tuples of Contractions}
	For an integer \(k \geq 2\), let \((T_1, T_2, \dots, T_k)\) be a \(k\)-tuple of commuting contractions on a Hilbert space \(\mc H\). Let \(S_k\) denote the symmetric group consisting of all permutations of the index set \(\{1,2,\dots,k\}\). For each permutation \(\sigma =\{\sigma(1),\dots,\sigma(k)\}\in S_k\), we consider the ordered product
	\[
	T_\sigma \coloneqq T_{\sigma(1)} T_{\sigma(2)} \cdots T_{\sigma(k)}.
	\]
	Since \(T_1, T_2, \dots, T_k\) commute pairwise, the product \(T_\sigma\) is independent of the choice of \(\sigma \in S_k\); therefore, we denote it by \(T\). For each \(\sigma \in S_k\), define the associated isometry
	\[
	Z_\sigma :
	\mathcal D_{T}
	\longrightarrow
	\bigoplus_{r=1}^{k} \mc D_{T_{\sigma(r})}
	\]
	by
	\[
	Z_\sigma(D_T h)
	\coloneqq
	\bigoplus_{r=1}^{k}
	D_{T_{\sigma(r)}}
	T_{\sigma(r+1)}\cdots T_{\sigma(k)} h,
	\qquad h \in \mc H,
	\]
	where, for \(r=k\), the product \(T_{\sigma(r+1)} \cdots T_{\sigma(k)}\) is understood to be \(I_{\mc H}\).
	The factorization
	\begin{equation}\label{sigma_fac}
		T
		=
		T_{\sigma(1)} T_{\sigma(2)} \cdots T_{\sigma(k)}
	\end{equation}
	is \(k\)-regular whenever the associated operator \(Z_\sigma\) is unitary. Furthermore, if the factorization \eqref{sigma_fac}
	is a \(k\)-regular factorization for every permutation \(\sigma \in S_k\), then the commuting \(k\)-tuple \((T_1, T_2, \dots, T_k)\) is called a \emph{symmetric \(k\)-regular tuple}.
	A natural question arising in this context is the following: if, for some \(\sigma \in S_k\), the factorization \eqref{sigma_fac} is \(k\)-regular, does it follow that the commuting \(k\)-tuple \((T_1, T_2, \dots, T_k)\) is symmetric \(k\)-regular? For pairs of commuting contractions, this question was answered by J.~A.~Ball and H.~Sau in Theorem~4.3.6 of \cite{Ba23a}. We extend their result to commuting \(k\)-tuples as follows.
	\begin{theorem}\label{k_commuting}
		Let \((T_1, T_2, \dots, T_k)\) be a commuting \(k\)-tuple of contractions on a Hilbert space \(\mc H\).
		\begin{enumerate}
			\item Assume that the defect spaces \(\mc{D}_{T}\) is finite-dimensional, where $T=T_1\cdots T_k$. If there exists \(\sigma \in S_k\) such that the factorization
			\[
			T=T_{\sigma(1)} T_{\sigma(2)} \cdots T_{\sigma(k)}
			\]
			is \(k\)-regular, then the commuting \(k\)-tuple \((T_1, T_2, \dots, T_k)\) is symmetric \(k\)-regular.
			
			\item In the infinite-dimensional setting (that is, when $\dim \mathcal D_T=\infty$), the existence of \(\sigma \in S_k\) such that
			\[
			T=T_{\sigma(1)} T_{\sigma(2)} \cdots T_{\sigma(k)}
			\]
			is \(k\)-regular does not imply that the commuting \(k\)-tuple \((T_1, T_2, \dots, T_k)\) is symmetric \(k\)-regular.
			
			\item If \((T_1, T_2, \dots, T_k)\) is a commuting \(k\)-tuple of isometries, then \((T_1, T_2, \dots, T_k)\) is symmetric \(k\)-regular.
		\end{enumerate}
	\end{theorem}
	\begin{proof}
		A counterexample for part~(ii) is provided in the proof of Theorem~4.3.6 of \cite{Ba23a}, while part~(iii) follows from Proposition~\eqref{k_regular_pro}(iii). Thus, it remains only to prove part~(i). Since the factorization
		\[
		T=T_{\sigma(1)} T_{\sigma(2)} \cdots T_{\sigma(k)}
		\]
		is \(k\)-regular, the associated isometry
		\[
		Z_\sigma :
		\mathcal D_{T}
		\longrightarrow
		\bigoplus_{r=1}^{k} \mc D_{T_{\sigma(r)}}
		\]
		is unitary. Consequently,
		\[
		\dim(\mathcal D_T)
		=
		\dim(\mathcal D_{T_{\sigma(1)}})
		+\cdots+
		\dim(\mathcal D_{T_{\sigma(k)}}).
		\]
		Let \(\tau \in S_k\) be arbitrary, and consider the isometry associated with the factorization
		\[
		T=T_{\tau(1)} T_{\tau(2)} \cdots T_{\tau(k)},
		\]
		namely,
		\[
		Z_\tau :
		\mathcal D_T
		\longrightarrow
		\bigoplus_{r=1}^{k} \mc D_{T_{\tau(r)}}.
		\]
		Since permutations preserve the collection of defect spaces,
		\[
		\dim(\mathcal D_{T_{\sigma(1)}})
		+\cdots+
		\dim(\mathcal D_{T_{\sigma(k)}})
		=
		\dim(\mathcal D_{T_{\tau(1)}})
		+\cdots+
		\dim(\mathcal D_{T_{\tau(k)}}),
		\]
		it follows that
		\[
		\dim(\mathcal D_T)
		=
		\dim(\mathcal D_{T_{\tau(1)}})
		+\cdots+
		\dim(\mathcal D_{T_{\tau(k)}}).
		\]
		Hence, by Proposition~\eqref{k_regular_pro}(iv), the operator \(Z_\tau\) is unitary. Therefore, the factorization
		\[
		T=T_{\tau(1)} T_{\tau(2)} \cdots T_{\tau(k)}
		\]
		is \(k\)-regular. Since \(\tau \in S_k\) was arbitrary, the commuting \(k\)-tuple \((T_1,T_2,\dots,T_k)\) is symmetric \(k\)-regular.
	\end{proof}

	Let \((T_1,\dots,T_k)\) be a \(k\)-tuple of commuting contractions on a Hilbert space \(\mathcal H\). A \(k\)-tuple of commuting isometries \((V_1,\dots,V_k)\) on a Hilbert space \(\mathcal K \supseteq \mathcal H\) is called a \emph{commuting isometric dilation} of \((T_1,\dots,T_k)\) if
	\[
	T_1^{n_1}\cdots T_k^{n_k}
	=
	P_{\mathcal H}
	V_1^{n_1}\cdots V_k^{n_k}\big|_{\mathcal H}
	\]
	for all \((n_1,\dots,n_k)\in \mathbb N^k.\) A commuting \(k\)-tuple of contractions \((T_1,\dots,T_k)\) is said to satisfy the \emph{von Neumann inequality} if
	\[
	\|p(T_1,\dots,T_k)\|
	\le
	\sup_{(z_1,\dots,z_k)\in \mathbb D^k}
	|p(z_1,\dots,z_k)|
	\]
	for every polynomial \(p\in \mathbb C[z_1,\dots,z_k],\) where
	\[
	\mathbb D^k
	=
	\{(z_1,\dots,z_k)\in \mathbb C^k:\ |z_i|<1,\ 1\le i\le k\}
	\]
	denotes the  $k$-polydisk in $\mathbb{C}^k.$
	
	To formalize our investigation, we define the following classes of commuting contractions
	\begin{align*}
		\mathcal{VNI}_k&(\mathcal H)\\
		&\coloneqq
		\Big\{
		(T_1,\dots,T_k)\in \mathcal B(\mathcal H)^k :
		T_i T_j = T_j T_i \text{ for all } i,j, \\
		&\qquad\ \ \text{and } (T_1,\dots,T_k)
		\text{ satisfies the von Neumann inequality}
		\Big\},\\
		\mathcal{CID}_k&(\mathcal H)\\
		&\coloneqq
		\Big\{
		(T_1,\dots,T_k)\in \mathcal B(\mathcal H)^k :
		T_i T_j = T_j T_i \text{ for all } i,j, \\
		&\qquad\ \ \text{and } (T_1,\dots,T_k)
		\text{ admits a commuting isometric dilation}
		\Big\},\\
		\mathcal{SR}_k&(\mathcal H)\\
		&\coloneqq
		\Big\{
		(T_1,\dots,T_k)\in \mathcal B(\mathcal H)^k :
		T_i T_j = T_j T_i \text{ for all } i,j, \\
		&\qquad\ \ \text{and } (T_1,\dots,T_k)
		\text{ is a symmetric } k\text{-regular tuple}
		\Big\}.
	\end{align*}
	
	It is a well-established result that if a commuting $k$-tuple of contractions admits a commuting isometric dilation, it necessarily satisfies the von Neumann inequality. Consequently, we have the inclusion:
	$$\mathcal{CID}_k(\mathcal H) \sub \mathcal{VNI}_k(\mathcal H).$$
	It is a classical result of Sz.-Nagy \cite{Sz53a} that a single contraction (the $k=1$ case) always admits an isometric dilation. Furthermore, T.~And\^o \cite{An63a} established that any pair of commuting contractions (the $k=2$ case) possesses a commuting isometric dilation. Therefore, for $k \le 2$, commuting isometric dilations universally exist, and the von Neumann inequality universally holds.
	
	However, for $k \ge 3$, this elegant theory breaks down, and the inclusion $\mathcal{CID}_k(\mathcal H) \subsetneq \mathcal{VNI}_k(\mathcal H)$ becomes strictly proper. S.~Parrott demonstrated this strict proper inclusion by constructing a commuting $3$-tuple that satisfies the von Neumann inequality yet fails to admit a commuting isometric dilation. Furthermore, the von Neumann inequality itself fails to hold in general for three or more operators, as demonstrated by the explicit counterexamples provided by S.~Kaijser and N.~Th.~Varopoulos, along with M.~J.~Crabb and A.~M.~Davie.
	
	The structural failures that emerge for $k \ge 3$ motivate further exploration to better understand the conditions under which these dilations exist. To this end, we investigate the class of symmetric $k$-regular tuples, denoted $\mathcal{SR}_k(\mathcal H)$, and explore its relationship with $\mathcal{VNI}_k(\mathcal H)$ and $\mathcal{CID}_k(\mathcal H)$. We first observe that for $k=2$, the classes $\mathcal{VNI}_2(\mathcal H)$ and $\mathcal{CID}_2(\mathcal H)$ strictly contain $\mathcal{SR}_2(\mathcal H)$. This occurs because there exist commuting contractions whose product is not $2$-regular, despite automatically admitting a commuting isometric dilation via And\^o's theorem. Given this distinction, our primary objective is to analyze the classical $3$-tuple counterexamples constructed by Parrott, Crabb--Davie, and Kaijser--Varopoulos, to determine precisely whether these specific tuples satisfy the conditions of symmetric $3$-regularity.

	The following two examples from \cite{Va74a,CD75a} violate the von Neumann inequality and therefore fail to admit commuting isometric dilations.
	\begin{example}[S.~Kaijser and N.~Th.~Varopoulos,\cite{Va74a}]
		Let \(\mc H=\mb C^5\), and consider the commuting  contractions \(T_1,T_2,T_3\in B(\mb C^5)\) given by
		\[
		T_1=
		\begin{pmatrix}
			0&0&0&0&0\\
			1&0&0&0&0\\
			0&0&0&0&0\\
			0&0&0&0&0\\
			0&\frac1{\sqrt3}&-\frac1{\sqrt3}&-\frac1{\sqrt3}&0
		\end{pmatrix},
		\quad
		T_2=
		\begin{pmatrix}
			0&0&0&0&0\\
			0&0&0&0&0\\
			1&0&0&0&0\\
			0&0&0&0&0\\
			0&-\frac1{\sqrt3}&\frac1{\sqrt3}&-\frac1{\sqrt3}&0
		\end{pmatrix},
		\]
		\[
		T_3=
		\begin{pmatrix}
			0&0&0&0&0\\
			0&0&0&0&0\\
			0&0&0&0&0\\
			1&0&0&0&0\\
			0&-\frac1{\sqrt3}&-\frac1{\sqrt3}&\frac1{\sqrt3}&0
		\end{pmatrix}.
		\]
		A direct computation yields
		\[
		T=T_1T_2T_3=0.
		\]
		Hence
		\[
		D_T=(I-T^*T)^{1/2}=I_5,
		\qquad
		\dim\mc D_T=5.
		\]
		Moreover,
		\[
		D_{T_1}=(I-T_1^*T_1)^{1/2},\qquad
		D_{T_2}=(I-T_2^*T_2)^{1/2},\qquad
		D_{T_3}=({I-T_3^*T_3})^{1/2}
		\]
		and each of these operators is an orthogonal projection of rank \(3\). Consequently,
		\[
		\dim\mc D_{T_1}
		=
		\dim\mc D_{T_2}
		=
		\dim\mc D_{T_3}
		=
		3.
		\]
		Therefore,
		\[
		\dim(\mc D_{T_1}\oplus\mc D_{T_2}\oplus\mc D_{T_3})
		=
		3+3+3
		=
		9.
		\]
		Since
		\[
		\dim\mc D_T
		<
		\dim(\mc D_{T_1}\oplus\mc D_{T_2}\oplus\mc D_{T_3}),
		\]
		the isometry associated with $\sigma=\{1,2,3\}$
		\begin{align*}
			Z_{\sigma}:\mc D_T
			\longrightarrow
			\mc D_{T_1}\oplus\mc D_{T_2}\oplus\mc D_{T_3}
		\end{align*}
		cannot be surjective. Hence, the factorization $T=T_1T_2T_3$ fails to be $3$-regular, implying that $(T_1, T_2, T_3)$ is not a symmetric $3$-regular tuple.
	\end{example}
	
	\begin{example}[M. J. Crabb and A. M. Davie,\cite{CD75a}]
		Let
		\[
		\mc H
		=
		\operatorname{span}
		\{e,f_1,f_2,f_3,g_1,g_2,g_3,h\}
		\]
		with orthonormal basis
		\[
		\mc B
		=
		\{e,f_1,f_2,f_3,g_1,g_2,g_3,h\}.
		\]
		Define commuting contractions \(T_1,T_2,T_3\in B(\mc H)\) by
		\[
		T_i e=f_i,\qquad
		T_i f_i=-g_i,\qquad
		T_i f_j=g_k \quad (k\neq i,j),
		\]
		\[
		T_i g_j=\delta_{ij}h,
		\qquad
		T_i h=0,
		\]
		for \(i,j,k\in\{1,2,3\}\).
		A direct computation yields
		\[
		T=T_1T_2T_3
		=
		\begin{pmatrix}
			0&0&0&0&0&0&0&0\\
			0&0&0&0&0&0&0&0\\
			0&0&0&0&0&0&0&0\\
			0&0&0&0&0&0&0&0\\
			0&0&0&0&0&0&0&0\\
			0&0&0&0&0&0&0&0\\
			0&0&0&0&0&0&0&0\\
			1&0&0&0&0&0&0&0
		\end{pmatrix}.
		\]
		Hence
		\[
		T^*T
		=
		\operatorname{diag}(1,0,0,0,0,0,0,0),
		\]
		and therefore
		\[
		D_T
		=
		(I-T^*T)^{1/2}
		=
		\operatorname{diag}(0,1,1,1,1,1,1,1).
		\]
		Consequently,
		\[
		\dim\mc D_T=7.
		\]
		Further,
		\[
		D_{T_1}
		=
		\operatorname{diag}(0,0,0,0,0,1,1,1),
		\]
		\[
		D_{T_2}
		=
		\operatorname{diag}(0,0,0,0,1,0,1,1),
		\]
		\[
		D_{T_3}
		=
		\operatorname{diag}(0,0,0,0,1,1,0,1).
		\]
		Hence
		\[
		\dim\mc D_{T_1}
		=
		\dim\mc D_{T_2}
		=
		\dim\mc D_{T_3}
		=
		3.
		\]
		Therefore,
		\[
		\dim(\mc D_{T_1}\oplus\mc D_{T_2}\oplus\mc D_{T_3})
		=
		3+3+3
		=
		9.
		\]
		Since
		\[
		\dim\mc D_T
		=
		7
		<
		9
		=
		\dim(\mc D_{T_1}\oplus\mc D_{T_2}\oplus\mc D_{T_3}),
		\]
		the isometry associated with $\sigma=\{1,2,3\}$
		\begin{align*}
			Z_{\sigma}:\mc D_T
			\longrightarrow
			\mc D_{T_1}\oplus\mc D_{T_2}\oplus\mc D_{T_3}
		\end{align*}
		is not surjective. Therefore, the factorization
		\[
		T=T_1T_2T_3
		\]
		cannot be \(3\)-regular, which means $(T_1, T_2, T_3)$ is not a symmetric \(3\)-regular tuple.
	\end{example}
	
	The following example, due to \cite{Pa70a}, satisfies the von Neumann inequality but does not admit a commuting isometric dilation.
	\begin{example}[S. Parrott,\cite{Pa70a}]
		Let \(\mc H=\mc K\oplus\mc K\), where \(\mc K=\mathbb{C}^n\), and consider the commuting contractions
		\[
		T_1=
		\begin{pmatrix}
			0&0\\
			I&0
		\end{pmatrix},
		\qquad
		T_2=
		\begin{pmatrix}
			0&0\\
			U&0
		\end{pmatrix},
		\qquad
		T_3=
		\begin{pmatrix}
			0&0\\
			V&0
		\end{pmatrix},
		\]
		where \(U,V\in B(\mc K)\) are noncommuting unitary operators.
		Since
		\[
		T_2T_1=0,
		\]
		we obtain
		\[
		T=T_1T_2T_3=0.
		\]
		Hence
		\[
		D_T=(I-T^*T)^{1/2}=I_{\mc H},
		\qquad
		\mc D_T=\mc H\cong \mc K\oplus\mc K.
		\]
		Moreover, for each \(i=1,2,3\),
		\[
		T_i^*T_i=
		\begin{pmatrix}
			I_{\mc K}&0\\
			0&0
		\end{pmatrix},
		\]
		and therefore
		\[
		D_{T_i}
		=
		\begin{pmatrix}
			0&0\\
			0&I_{\mc K}
		\end{pmatrix}.
		\]
		Consequently,
		\[
		\mc D_{T_i}\cong \mc K,
		\qquad i=1,2,3.
		\]
		Thus
		\[
		\mc D_{T_1}\oplus\mc D_{T_2}\oplus\mc D_{T_3}
		\cong
		\mc K\oplus\mc K\oplus\mc K.
		\]
		Since
		\[
		\dim\mc D_T
		=
		2n
		<
		3n
		=
		\dim(\mc D_{T_1}\oplus\mc D_{T_2}\oplus\mc D_{T_3}),
		\]
		hencethe isometry associated with $\sigma=\{1,2,3\}$
		\begin{align*}
			Z_{\sigma}:\mc D_T
			\longrightarrow
			\mc D_{T_1}\oplus\mc D_{T_2}\oplus\mc D_{T_3}
		\end{align*}
		cannot be surjective. Hence, the factorization
		\[
		T=T_1T_2T_3
		\]
		is not \(3\)-regular, and consequently, $(T_1, T_2, T_3)$ is not a symmetric \(3\)-regular tuple.
	\end{example}
	The observation that these classical counterexamples fail to be symmetric $k$-regular ($k \geq 3$) motivates us to explore the interplay between the class $\mathcal{SR}_k(\mathcal H)$ and the classes $\mathcal{CID}_k(\mathcal H)$ and $\mathcal{VNI}_k(\mathcal H)$ for $k \geq 3$.\\

	\noindent\textbf{Acknowledgment.}
	The research of the first author is supported in part by the Indian Institute of Technology Goa (SEED Grant 2022/SG/KH/047) and the Anusandhan National Research Foundation (MATRICS Grant MTR/2022/000339). The second author is supported by a CSIR-SRF fellowship (File No. 09/1290(12920)/2021-EMR-I) from the Council of Scientific and Industrial Research (CSIR), India.


\begin{thebibliography}{99}
		
		\bibitem{An63a}
		T. And\^o,
		\textit{On a pair of commutative contractions},
		Acta Sci. Math. (Szeged), \textbf{24} (1963), 88--90.
		
		
		
		\bibitem{BS20a}
		J. A. Ball and H. Sau, 
		\textit{Functional Models for Commuting Hilbert-Space Contractions}, 
		Oper. Theory Adv. Appl., \textbf{278} (2020), 11--54.
		
		\bibitem{Ba23a}
		J. A. Ball and H. Sau, 
		\textit{Dilation and Model Theory for Pairs of Commuting Contractions}, 
		\url{https://arxiv.org/abs/2308.07589}, (2023).
		
		\bibitem{Bh05a}
		T. Bhattacharyya, J. Eschmeier and J. Sarkar,
		\textit{Characteristic function of a pure commuting contractive tuple},
		Integral Equations Operator Theory, \textbf{53} (2005), 23--32.
		
		\bibitem{Bh06a}
		T. Bhattacharyya, J. Eschmeier and J. Sarkar,
		\textit{On CNC commuting contractive tuples},
		Proc. Indian Acad. Sci. Math. Sci., \textbf{116} (2006), 299--316.
		
		\bibitem{CD75a}
		M. J. Crabb and A. M. Davie,
		\textit{von Neumann's inequality for Hilbert space operators},
		Bull. London Math. Soc., \textbf{7} (1975), 49--50.
		
		
		\bibitem{Fr82a}
		A. E. Frazho,
		\textit{Models for noncommuting operators},
		J. Funct. Anal., \textbf{48} (1982), 1--11.
		
		\bibitem{HM2026}
		K. J. Haria and A. K. Maurya,
		\textit{Regular factorizations of characteristic functions of contractions in triangular form},
		preprint.
		
		\bibitem{HM2026II}
		K. J. Haria and A. K. Maurya,
		\textit{$k$-Regular factorizations and joint invariant subspaces of completely non-coisometric row contractions}, \url{https://doi.org/10.13140/RG.2.2.35992.15363}, 	preprint.
		
		\bibitem{Kerchy03}
		L. K\'erchy,
		\textit{On the factorization of operator-valued functions},
		Acta Sci. Math. (Szeged), \textbf{69} (2003), 337--348.
		
		\bibitem{Khan90a}
		D. K. Khan,
		\textit{Factorization of transfer functions. I. (+)-Regular factorization},
		Ukrainian Math. J., \textbf{42} (1990), 279--282.
		
		\bibitem{Khan90b}
		D. K. Khan,
		\textit{Factorization of transfer functions. II. The minimality of passive scattering systems under a step-by-step combination},
		Ukrainian Math. J., \textbf{42} (1990), 432--437.
		
		\bibitem{Do94}
		D. C. Khan,
		\textit{($\pm$)-Regular factorization of transfer function and passive scattering system for cascade coupling},
		J. Operator Theory, (1994).
		
		\bibitem{Pa70a}
		S. Parrott,
		\textit{Unitary dilations for commuting contractions},
		Pacific J. Math., \textbf{34} (1970), 481--490.
		
		\bibitem{Po89a}
		G. Popescu,
		\textit{Isometric dilations for infinite sequences of noncommuting operators},
		Trans. Amer. Math. Soc., \textbf{316} (1989), 523--536.
		
		\bibitem{Po89b}
		G. Popescu,
		\textit{Characteristic functions for infinite sequences of noncommuting operators},
		J. Operator Theory, \textbf{22} (1989), 51--71.
		
		\bibitem{Po95a}
		G. Popescu,
		\textit{Multi-analytic operators on Fock spaces},
		Math. Ann., \textbf{303} (1995), 31--46.
		
		\bibitem{Po06}
		G. Popescu,
		\textit{Characteristic functions and joint invariant subspaces},
		J. Funct. Anal., \textbf{237} (2006), 277--320.
		
		\bibitem{Sa18a}
		H. Sau, 
		\textit{And\^o dilations for a pair of commuting contractions: two explicit constructions and functional models}, 
		arXiv:1710.11368, 2018.
		
		
		\bibitem{Sz53a}
		B. Sz.-Nagy,
		\textit{Sur les contractions de l'espace de Hilbert},
		Acta Sci. Math. (Szeged), \textbf{15} (1953), 87--92.
		
		\bibitem{NF64a}
		B. Sz.-Nagy and C. Foia\c{s},
		\textit{Une caractérisation de sous-espaces invariants pour une contraction de l’espace de Hilbert},
		C. R. Math. Acad. Sci. Paris, \textbf{258} (1964), 3426--3429.
		
		\bibitem{NF64b}
		B. Sz.-Nagy and C. Foia\c{s},
		\textit{Sur les contractions de l’espace de Hilbert. IX. Factorisations de la fonction caractéristique. Sous-espaces invariants},
		Acta Sci. Math. (Szeged), \textbf{25} (1964), 283--316.
		
		\bibitem{NF67}
		B. Sz.-Nagy and C. Foia\c{s},
		\textit{Forme triangulaire d'une contraction et factorisation de la fonction caract\'eristique},
		Acta Sci. Math. (Szeged), \textbf{28} (1967), 201--212.
		
		\bibitem{NF70}
		B. Sz.-Nagy and C. Foia\c{s},
		\textit{Harmonic Analysis of Operators on Hilbert Space},
		North-Holland--Akad\'emiai Kiad\'o, Amsterdam--Budapest, 1970.
		
		\bibitem{NF74}
		B. Sz.-Nagy and C. Foia\c{s},
		\textit{Regular factorizations of contractions},
		Proc. Amer. Math. Soc., \textbf{43} (1974), 91--93.
		
		\bibitem{NFBK10}
		B. Sz.-Nagy, C. Foia\c{s}, H. Bercovici and L. K{\'e}rchy,
		\textit{Harmonic analysis of operators on {H}ilbert space},
		2nd ed., Universitext, Springer, New York, 2010.
		
		
		\bibitem{Te75a}
		R. I. Teodorescu,
		\textit{Sur les d\'ecompositions directes des contractions de l'espace de {H}ilbert},
		J. Funct. Anal., \textbf{18} (1975), 414--428.
		
		\bibitem{Te76a}
		R. I. Teodorescu,
		\textit{Fonctions caract\'eristiques constantes},
		Acta Sci. Math. (Szeged), \textbf{38} (1976), 183--185.
		
		\bibitem{Te77a}
		R. I. Teodorescu,
		\textit{The direct decompositions of contractions},
		Stud. Cerc. Mat., \textbf{29} (1977), 57--84.
		
		
		\bibitem{Te78b}
		R. I. Teodorescu,
		\textit{Factorisations r\'eguli\`eres et sous-espaces hyperinvariants},
		Acta Sci. Math. (Szeged), \textbf{40} (1978), 389--396.
		
		\bibitem{Te79a}
		R. I. Teodorescu,
		\textit{Sur l'unicit\'e{} de la d\'ecomposition des contractions en somme directe},
		J. Funct. Anal., \textbf{31} (1979), 245--254.
		
		\bibitem{Te80a}
		R. I. Teodorescu,
		\textit{Factorisations r\'eguli\`eres et sous-espaces invariants},
		Acta Sci. Math. (Szeged), \textbf{42} (1980), 325--330.
		
		\bibitem{Tim20}
		D. Timotin,
		\textit{The invariant subspaces of $S \oplus S^*$},
		Concrete Operators, \textbf{7} (2020), 116--123.
		
		\bibitem{Va74a}
		N. Th. Varopoulos,
		\textit{On an inequality of von Neumann and an application of the metric theory of tensor products to operators theory},
		J. Funct. Anal., \textbf{16} (1974), 83--100.
		
		
		
		\bibitem{Wu78a}
		P. Y. Wu,
		\textit{Hyperinvariant subspaces of the direct sum of certain contractions},
		Indiana Univ. Math. J., \textbf{27} (1978), 267--274.
		
		\bibitem{Wu79b}
		P. Y. Wu,
		\textit{Hyperinvariant subspaces of weak contractions},
		Acta Sci. Math. (Szeged), \textbf{41} (1979), 259--266.
		
		\bibitem{Wu79c}
		P. Y. Wu,
		\textit{Hyperinvariant subspaces of {$C_{11}$} contractions},
		Proc. Amer. Math. Soc., \textbf{75} (1979), 53--58.
		
		\bibitem{Wu78b}
		P. Y. Wu,
		\textit{Hyperinvariant subspaces of {$C_{11}$} contractions. II},
		Indiana Univ. Math. J., \textbf{27} (1978), 805--812.
		
		
		\bibitem{Wu79a}
		P. Y. Wu,
		\textit{Conditions for completely nonunitary contractions to be spectral},
		J. Funct. Anal., \textbf{31} (1979), 1--12.
		
	\end{thebibliography}
\end{document}